%% file: Local_Type_C.tex
\documentclass[a4paper,12pt]{amsart}
\usepackage{amsaddr}
\usepackage[a4paper,left=24mm,right=24mm,top=34mm,bottom=34mm]{geometry}
\usepackage[utf8]{inputenc}
\usepackage{color}
\usepackage{xr-hyper} 
\usepackage{hyperref}
\usepackage{cleveref}
\usepackage{showlabels}
\usepackage{todonotes}

\hypersetup{
	colorlinks,
	citecolor=red,
	filecolor=blue,
	linkcolor=blue,
	urlcolor=purple
}


\def\subjclass#1{{\renewcommand{\thefootnote}{}%
		\footnote{\emph{Mathematics Subject Classification (2010):} #1}}}
\def\keywords#1{\par\medskip 	\noindent\textbf{Keywords.} #1}

\usepackage[all]{xy}
\usepackage{amsmath,amssymb,amsfonts,amsthm,graphicx}

\usepackage{multirow}
\usepackage{comment}
\theoremstyle{plain}

\tikzset{node distance=2cm, auto}

\newtheorem{thm}{Theorem}[section]

\newtheorem{cor}[thm]{Corollary}
\newtheorem{prop}[thm]{Proposition}

\newtheorem{lm}[thm]{Lemma}

\newtheorem{aim}[thm]{Aim}

\newtheorem*{theorem*}{Theorem}
\newtheorem*{aim*}{Aim}
\newtheorem*{initialaim*}{Initial Aim}
\newtheorem*{conj*}{Conjecture}
\newtheorem*{cor*}{Corollary}
\newtheorem*{prop*}{Proposition}
\newtheorem*{df*}{Definition}
\newtheorem*{lm*}{Lemma}
\newtheorem*{example*}{Example}
\newtheorem*{notation*}{Notation}
\newtheorem*{prob*}{Problem}

\theoremstyle{remark}
\newtheorem{rem}[thm]{Remark}
\newtheorem{notation}[thm]{Notation}

\newcommand{\F}{\mathbb{F}}

\newcommand{\GenGp}[1]{\langle#1\rangle}

\newcommand{\Norm}[2]{{\rm N}_{#1}({#2})}

\newcommand{\Sp}{\operatorname{Sp}}

\newcommand{\tA}{\mathrm A}
\newcommand{\tC}{\mathrm C}
\newcommand{\tB}{\mathrm B}
\newcommand{\tD}{\mathrm D}
\input{shortcuts}

\newcommand{\calV}{{\mathcal V}}

\newcommand{\wtbL}{\widetilde\bL}
\newcommand{\Symm}{\mathfrak{S}}
\newcommand{\sgnSymm}{\mathfrak{S}^{\pm}}

\newcommand{\GL}{\operatorname{GL}}

\newcommand{\CSp}{\operatorname{CSp}}

\newcommand{\Irr}{\operatorname{Irr}}

\newcommand{\Ind}{\operatorname{Ind}}
\newcommand{\Res}{\operatorname{Res}}

\newcommand{\ol}{\overline}
\newcommand{\ul}{\underline}

\newcommand{\wh}{\widehat}

\newcommand{\bG}{{{\mathbf G}}}

\numberwithin{equation}{section}

\title[Characters of normalisers of $d$-split Levi subgroups in ${\rm Sp}_{2n}(q)$]{Characters of normalisers of $d$-split Levi subgroups in ${\rm Sp}_{2n}(q)$}

\begin{document}

\date{\today}
\author{Julian Brough}
\thanks{The research of the author is funded through the DFG (Project: BR 6142/1-1)}
\address{School of Mathematics and Natural Sciences, University of Wuppertal, Gau\ss str. 20, 42119 Wuppertal, Germany}
	\email{brough@uni-wuppertal.de} 
\begin{abstract}
	
		In this paper characters of the normaliser of $d$-split Levi subgroups in ${\rm Sp}_{2n}(q)$ are parametrised with a particular focus on the Clifford theory between the Levi subgroup and its normaliser.
This forms a key step in verifying the inductive Alperin--McKay condition from Sp{\"a}th via a criterion given by  Sp{\"a}th and the author.

\end{abstract}

\date{\today}
\maketitle

\excludecomment{commentMe}

\vspace{-10mm}
\normalsize

\section{Introduction}
In the representation theory of finite groups there are  important conjectures which predict a strong relationship between the representations of a finite group and those of its $\ell$-local subgroups, for a prime $\ell$. The Alperin--McKay conjecture is one of these conjectures which refines the McKay conjecture using the notion of $\ell$-blocks introduced by Brauer. 

In \cite{IMNRedMcKay,AMSp} the McKay and the Alperin--McKay conjecture have been reduced to the verification of a so-called {\it inductive condition} for all finite simple groups and primes $\ell$. 
Moreover, to prove the Brauer height zero conjecture it remains to validate the inductive Alperin--McKay (iAM) condition for all simple groups \cite{NavSpBrHe0}.
Using the inductive McKay condition, Malle and Späth established the McKay conjecture for $\ell=2$ \cite{MalSpOddDeg}, while for odd primes due to the work of multiple authors the inductive McKay condition only remains open for simple groups of type ${\tD}$. 

For the iAM condition \cite{AMSp} dealt with simple groups of Lie type, when $\ell$ is the defining characteristic, and for alternating groups, when the prime $\ell$ is odd.
Additionally \cite{BAWMa} dealt with simple groups in certain low rank cases, while \cite{CabSpAMTypeA} dealt with the blocks of maximal defect in type $\tA$. 
In \cite{BrSpAM} Sp{\"a}th and the author provided a criterion to verify the iAM condition which is tailored towards groups of Lie type and was applied in type ${\tA}$.
In \cite{CabASFSp}, this criterion was successfully applied to the groups ${\rm Sp}_{2n}(q)$ for primes $\ell\mid q-1$. 
Furthermore in \cite{RuQIAM} Ruhstorfer used the results in type ${\tA}$ with a reduction to quasi-isolated blocks to prove the iAM condition for all blocks in type $\tA$ whenever $\ell\geq 5$.
Recently Ruhstorfer and the author have established the Alperin--McKay conjecture for all $2$-blocks of maximal defect \cite{BrRuAM2}.
There has also been a large amount of success in verifying the inductive (blockwise) Alperin weight condition via an analogous criterion given in \cite{BrSpAW}.

The labelling for blocks in groups of Lie type by $d$-cuspidal pairs, $(\bL,\lambda)$ as in \cite{CabEnRedGp}, provides strong evidence that the normaliser of $\bL$ should be a suitable local subgroup for the iAM condition (this was proven for type $\tA$ in \cite[Theorem 5.1]{BrSpAM}). 
The main result of this paper validates condition \cite[Theorem 2.4 (iv)]{BrSpAM} for finite groups of Lie type in type $\tC$ without the assumption that $\ell\mid q-1$:

\begin{thm}\label{thmA}
Let $p$ be a prime, $\bG:=\Sp_{2n}(\ol{\mathbb{F}}_p)$, $\wt\bG:=\CSp_{2n}(\ol{\mathbb{F}}_p)$, and 
		$F_p:\wbG\rightarrow \wbG$ a Frobenius endomorphism defining an $\FF_p$-structure.
For some positive integer $m$ let $F=F_p^m$ and $D:=\GenGp{\wh{F_p}}$ where $\wh{F}_p$ is the induced automorphism of $F_p$ on $\bG^F$.
Assume that, for some positive integer $d$, the subgroup $\bL$ is a $d$-split Levi subgroup of $({\bG}, F)$ and set
 $N:=\Norm{\GF}{{\bf L}}$, $\wt N:=\Norm{{\wt\bG^F}}{\bL}$, and $\wh{N}:=\Norm{\GF\rtimes D}{{\bf L}}$.
If $\psi\in \Irr ( N)$ lies above a $d$-cuspidal character of ${\bf L}^F$, then there exists a $\wt N$-conjugate $\psi_0$ of $\psi$ such that 
\begin{asslist}
\item \label{thm11a}$(\wt{N}\wh{N})_{\psi_0}=\wt{N}_{\psi_0}\wh{N}_{\psi_0}$, and
\item $\psi_0$ extends to $\wh{N}_{\psi_0}$.
\end{asslist}
\end{thm}

To prove Theorem~\ref{thmA}, the criterion proven in \cite[Theorem 4.1]{BrSpAM} is considered.
Given $\bL$ a $d$-split Levi subgroup, the main step is to paramterise the characters of its normaliser via $\bL^F$ and the corresponding relative Weyl group.
Then with this parametrisation investigate the action of automorphisms.
In particular, the following is proven:

\begin{thm}\label{StarConddSpLevi}
Assume the setup of Theorem~\ref{thmA} and set $W=N/{\bf L}^F$ and $\wh{W}:=\wh{N}/{\bf L}^F$.
\begin{thmlist}
\item \label{thm12a} There exists a $\wh{W}$-equivariant extension map with respect to ${\bf L}^F\lhd N$.
\item \label{thm12b}
Let $\lambda$ be a $d$-cuspidal character of ${\bf L}^F$ and $\wt{\lambda}\in \Irr (({\bf L}Z({\bf \wt{G}}))^F\mid \lambda)$ and $\xi_0\in \Irr(W_{\wt\lambda})$.
Then each character $\xi\in \Irr(W_{\lambda}\mid \xi_0)$ is $\NNN_{\wh{W}}(W_{\wt\lambda},W_{\lambda})_{\xi_0}$-invariant.
\end{thmlist}
\end{thm}

In addition to the main theorem contributing to the validation of the iAM condition for ${\rm PSp}_{2n}(q)$, there are other consequences.
For groups of Lie type with abelian Sylow $\ell$-subgroups, suitable bijections for the iAM condition and the blockwise Alperin weight condition were constructed by Malle \cite[Theorem 2.9]{BAWMa} assuming \ref{thm12a} for analogous local subgroups.
As a consequence of the proof of Theorem~\ref{thmA}, the Alperin--McKay conjecture holds via \cite[Theorem 2.9]{BAWMa} for ${\rm Sp}_{2n}(q)$ at primes $\ell\geq 5$ where the Sylow $\ell$-subgroups are abelian.
Beyond the scope of the Alperin--McKay conjecture, recently Rossi \cite[Theorem 8.3]{RoCTC} showed for a $d$-cuspidal pair $(\bL,\lambda)$ that condition \ref{thm12a} yields an equivariant bijection between the $d$-Harish-Chandra series $\mathcal{E}(\bG^F,(\bL,\lambda))$ and $\Irr(\NNN_{\bG^F}(\bL)\mid\lambda)$.
Moreover, condition \ref{thm11a} forms an ingredient for validating the inductive Dade condition for ${\rm PSp}_{2n}(q)$ \cite[Theorem 9.2]{RoCTC}.

This paper is organised in the following way: Section~\ref{Prelim} outlines some notation required for sets and signed permutations, and includes some results that will be required in the later sections.
Section~\ref{LabdLe} is used to parametrise $d$-split Levi subgroups in any classical type via sets. 
In Section~\ref{dSplitLevi} the restriction to type $\tC$ is made and an explicit description of $d$-split Levi subgroups and their normalisers is provided.
Moreover, some important properties of the irreducible characters of these subgroups are discussed.
Finally in Section~\ref{CliffTh} the Clifford theory arising in the relative Weyl group is studied in order to complete the proof of Theorem~\ref{StarConddSpLevi}.

{\bf Acknowledgement.} 
The author thanks Professor Britta Sp{\"a}th for useful discussions on the subject and Professor Gunter Malle for providing helpful comments on a previous draft.

\section{Preliminary results}\label{Prelim}

\subsection{Sets and permutations}
\indent

The structure of $d$-split Levi subgroups and their normalisers will be described explicitly via certain permutations and signed permutations.
In this section the required notation will be set up and some properties will be stated.

\subsubsection{Signed permutations}
\indent

Let $\Omega$ be a finite set, then the symmetric group on $\Omega$ will be denoted by $\Symm(\Omega)$.
In particular, if $|\Omega|=n$, then $\Symm(\Omega)\cong \Symm_n$, the symmetric group on $n$ objects.
For a finite subset $\mathcal{J}\subset \mathbb{N}_{>0}$  denote by $\pm \mathcal{J}:=\{\pm i\mid i\in\mathcal{J}\}$  and the group of signed permutations on $\mathcal{J}$ is defined as
$$
\begin{array}{ccl}
\sgnSymm(\mathcal{J}) & := & \{\sigma\in \Symm(\pm\mathcal{J})\mid \sigma(-i)=-\sigma(i) \text{ for all } i\in\mathcal{J}\}.\\
\end{array}
$$
As with symmetric groups, for $n\in \mathbb{N}_{>0}$, denote by $\sgnSymm_n$ the signed permutations on $\ul{n}:=\{1,\dots, n\}$.
If $\mathcal{J}\subset\mathbb{N}_{>0}$ with $|\mathcal{J}|=n$, then $\sgnSymm(\mathcal{J})\cong \sgnSymm_n$.

Given a finite subset $\mathcal{J}\subset \mathbb{N}_{>0}$ 
and a cycle $(a_1,\dots, a_r)\in \Symm(\pm\mathcal{J})$, define the cycle $(a_1,\dots, a_r)':=(-a_1,\dots, -a_r)$.
Then for $m$ disjoint cycles $\sigma_1,\dots,\sigma_m\in\Symm(\pm\mathcal{J})$ define $(\sigma_1\cdots\sigma_m)':=\sigma_1'\cdots\sigma_m'$.
This yields a natural embedding of $\Symm(\mathcal{J})$ into $\sgnSymm(\mathcal{J})$ by identifying an element $\tau\in \Symm(\mathcal{J})$ with the element $\tau\tau'\in\sgnSymm(\mathcal{J})$.
It follows from the definition that $\sigma\in\sgnSymm(\mathcal{J})$ if and only if it is a product of elements of the form $(a_1,\dots, a_r,-a_1,\dots,-a_r)$ or $(a_1,\dots, a_r)(-a_1,\dots,-a_r)$ for some $a_1,\dots,a_r\in\pm \mathcal{J}$.
Therefore the map which sends both $(a_1,\dots, a_r)(-a_1,\dots,-a_r)$ and $(a_1,\dots, a_r,-a_1,\dots,-a_r)$ to $(|a_1|,\dots, |a_r|)\in\Symm(\mathcal{J})$ extends to a homomorphism $\ol{\color{white} a }:\sgnSymm(\mathcal{J})  \rightarrow  \Symm(\mathcal{J})$ with ${\rm ker}(\ol{\color{white} a })= \GenGp{ (i,-i)\mid i\in\mathcal{J}}\cong C_2^{|\mathcal{J}|}$.
For any element $\sigma\in \sgnSymm(\mathcal{J})$, denote by $\ol{\sigma}$ its image in $\Symm(\mathcal{J})$.
For $\tau\in\Symm(\mathcal{J})$ it follows that $\tau=\ol{\tau\tau'}$ and thus 
$$\sgnSymm(\mathcal{J})=\GenGp{ (i,-i)\mid i\in\mathcal{J}}\rtimes \Symm(\mathcal{J})\cong C_2\wr \Symm(\mathcal{J}).$$

\begin{lm}\label{CentsgnSymm}
Let $n\in \mathbb{N}_{>0}$ and $\sigma\in\sgnSymm_n$.
Then $$
\Cent_{\sgnSymm_n}(\sigma)\cong \left( \prod_{i=1}^n C_{2i}\wr \Symm_{m_{1,i}}\right) \times \left(\prod_{i=1, \text{ odd}}^n C_{2i}\wr \Symm_{m_{2,i}}\right)\times \left( \prod_ {i=1, \text{ even}}^n (C_{2}\times C_i)\wr \Symm_{m_{2,i}}\right),
$$
where $m_{1,i}$ equals the number of disjoint cycles $\alpha$ of $\sigma$ with length $2i$ and $\alpha=\alpha'$ while $2m_{2,i}$ is the number of disjoint cycles $\beta$ of $\sigma$ with length $i$ and $\beta\ne \beta'$.  
\end{lm}
\begin{proof}
After conjugation any element $\sigma\in\sgnSymm_n$ can be written as $$\sigma=\left(\prod_{i=1}^n\alpha_{i,1}\cdots\alpha_{i,m_{1,i}}\right)\left(\prod_{i=1}^n(\beta_{i,1}\beta_{i,1}')\cdots (\beta_{i,m_{2,i}}\beta_{i,m_{2,i}}')\right)$$
where each $\beta_{i,j}=\ol{\beta_{i,j}}$ and $\alpha_{i,j}=(a_{j,1},\dots, a_{j,i},-a_{j,1},\dots, -a_{j,i})$ with $\ol{\alpha_{i,j}}=(a_{j,1},\dots, a_{j,i})$.
Set $\mathcal{I}_{\alpha_{i,j}}:=\{k\in \mathbb{N}_{>0}\mid \alpha_{i,j}(k)\ne k\}\subset \{1,\dots,n\}$ and $\delta_{\alpha_{i,j}}:=\prod_{k\in \mathcal{I}_{\alpha_{i,j}}}(k,-k)$ and similarly define $\mathcal{I}_{\beta_{i,j}}$ and $\delta_{\beta_{i,j}}$.
The kernel of the projection map from $\Cent_{\sgnSymm_n}(\sigma)$ to $\Cent_{\Symm_n}(\ol{\sigma})$ is the subgroup generated by the involutions $\delta_{\alpha_{i,j}}$ and $\delta_{\beta_{i,j}}$. 
Furthermore, for $\mathcal{I}_{1,i}:=\bigsqcup_{j=1}^{m_{1,i}}\mathcal{I}_{\alpha_{i,j}}$ and $\mathcal{I}_{1,i}:=\bigsqcup_{j=1}^{m_{2,i}}\mathcal{I}_{\beta_{i,j}}$ it follows that
$$
\Cent_{\sgnSymm_n}(\sigma)=\left( \prod_{i=1}^n\Cent_{\sgnSymm(\mathcal{I}_{1,i})}(\alpha_{i,1}\cdots\alpha_{i,m_{1,i}})\right)\times \left( \prod_{i=1}^n\Cent_{\sgnSymm(\mathcal{I}_{2,i})}((\beta_{i,1}\beta_{i,1}')\cdots (\beta_{i,m_{2,i}}\beta_{i,m_{2,i}}'))\right).
$$

For $1\leq i\leq n$ and $1\leq j\leq m_{1,i}$ set $\tau_{1,i,j}:=(a_{j,1},a_{j+1,1})\cdots (a_{j,i},a_{j+1,i})\in \Cent_{\Symm(\mathcal{I}_{1,i})}(\ol{\alpha}_{i,1}\cdots\ol{\alpha}_{i,m_{1,i}})$ which satisfies $\ol{\alpha}_{i,j}^{\tau_{1,i,j}}=\ol{\alpha}_{i,j+1}$ and $[\ol{\alpha}_{i,k},\tau_{1,i,j}]=1$ whenever $k\ne j,j+1$.
Direct calculation shows that
$$\Cent_{\sgnSymm(\mathcal{I}_{1,i})}(\alpha_{i,1}\cdots\alpha_{i,m_{1,i}})= \left( \prod_{j=1}^{m_{1,i}} \GenGp{\alpha_{i,j}}\right)\rtimes \GenGp{\tau_{1,i,j}\tau_{1,i,j}'\mid 1\leq j<m_{1,i}}
 $$
as this group surjects onto $\Cent_{\Symm(\mathcal{I}_{1,i})}(\ol{\alpha}_{i,1}\cdots\ol{\alpha}_{i,m_{1,i}})$ with kernel generated by the $\alpha_{i,j}^i=\delta_{\alpha_{i,j}}$. 

Similarly define involutions $\tau_{2,i,j}\in \Cent_{\Symm(\mathcal{I}_{2,i})}(\beta_{i,1}\cdots\beta_{i,m_{2,i}})$ such that $\beta_{i,j}^{\tau_{2,i,j}}=\beta_{i,j+1}$ and $[\beta_{i,k},\tau_{1,i,j}]=1$ whenever $k\ne j,j+1$.
As before direct calculation shows that 
$$\Cent_{\sgnSymm(\mathcal{I}_{2,i})}(\beta_{i,1}\beta_{i,1}'\cdots\beta_{i,m_{2,i}}\beta_{i,m_{2,i}}')= \left( \prod_{j=1}^{m_{2,i}} \GenGp{\beta_{i,j}\beta_{i,j}',\delta_{\beta_{i,j}}}\right)\rtimes \GenGp{\tau_{2,i,j}\tau_{2,i,j}'\mid 1\leq j<m_{2,i}}.
 $$
The result now follows by observing  
\begin{equation*}
\GenGp{\beta_{i,j}\beta_{i,j}',\delta_{\beta_{i,j}}}\cong \left\{
     \begin{array}{lr}
      C_{2i} & i \text{ odd}\\
       C_i\times C_2 & i \text{ even}\\
      \end{array}
   \right.  \qedhere
\end{equation*}
\end{proof}

\subsubsection{Set partitions and permutations}\label{NotSets}\label{PermPart}
\indent

Given a subset $\mathcal{J}\subseteq\mathbb{N}_{>0}$, with $|\mathcal{J}|=k$, write $\mathcal{J}=\{\mathcal{J}(1),\dots,\mathcal{J}(k)\}$ with $\mathcal{J}(i)<\mathcal{J}(i+1)$ and set 
$$w'_{\mathcal{J}}:=\big( \mathcal{J}(1),\mathcal{J}(2),\dots, \mathcal{J}(k),-\mathcal{J}(1),-\mathcal{J}(2),\dots, -\mathcal{J}(k)\big)\in \sgnSymm(\mathcal{J}).$$ 
If $\mathcal{J},\mathcal{J}'\subset \mathbb{N}_{>0}$ with $|\mathcal{J}|=|\mathcal{J}'|=k$ and $\mathcal{J}\cap\mathcal{J}'=\emptyset$ set 
$$\tau_{\mathcal{J},\mathcal{J}'}:=\prod_ {i=1}^k (\mathcal{J}(i),\mathcal{J}'(i))(-\mathcal{J}(i),-\mathcal{J}'(i))= \prod_{r=0}^{k-1}\big( (\mathcal{J}(1),\mathcal{J}'(1))(-\mathcal{J}(1),-\mathcal{J}'(1))\big)^{(w'_{\mathcal{J}}w'_{\mathcal{J}'})^r}.$$

For a collection of sets $\mathcal{I}:=\{\mathcal{J}\mid \mathcal{J}\subseteq \mathbb{N}_{>0}\}$ and a positive integer $i\in \mathbb{N}{>0}$ set $\mathcal{I}_i:=\{\mathcal{J}\in\mathcal{I}\mid |\mathcal{J}|=i\}$ and $\ul{\mathcal{I}}:=\bigcup_{\mathcal{J}\in\mathcal{I}} \mathcal{J}$.
A collection of sets $\mathcal{I}=\{\mathcal{J}\mid \mathcal{J}\subseteq \mathbb{N}_{>0}\}$ is called a partition of $n$, denoted by $\mathcal{I}\vdash \ul{n}$, if $\ul{\mathcal{I}}=\ul{n}$ and the sets $\mathcal{J}$ are disjoint.
For a partition $\mathcal{I}\vdash \ul{n}$ set $w'_\mathcal{I}:=\prod_{\mathcal{J}\in\mathcal{I}}w'_\mathcal{J}$ and $w'_{\mathcal{I}_s}:=\prod_{\mathcal{J}\in\mathcal{I}_s}w'_\mathcal{J}$.
By the proof of Lemma~\ref{CentsgnSymm}, it follows that $\Cent_{\sgnSymm_n}(w'_\mathcal{I})=\prod_{i=1}^n\Cent_{\sgnSymm(\ul{\mathcal{I}}_i)}(w'_{\mathcal{I}_s})$ with 
$$\Cent_{\sgnSymm(\ul{\mathcal{I}}_i)}(w'_{\mathcal{I}_s})=\left(\prod_{j=1}^{m_i}\GenGp{w'_{\mathcal{J}_{i,j}}}\right)\rtimes \GenGp{\tau_{\mathcal{J}_{i,j}}\mid 1\leq j\leq m_i-1}\cong \prod_{i=1}^n C_{2i}\wr\Symm_{m_i},$$
where $\mathcal{I}_i=\{\mathcal{J}_{i,1},\dots,\mathcal{J}_{i,m_i}\}$ and $\tau_{\mathcal{J}_{i,j}}:=\tau_{\mathcal{J}_{i,j},\mathcal{J}_{i,j+1}}$ for $1\leq j< m_i$. 
Note for $1\leq j<j'\leq m_i$ direct calculation shows that $\tau_{\mathcal{I}_{i,j},\mathcal{I}_{i,j'}}=(\tau_{\mathcal{I}_{i,j}})^{\tau_{\mathcal{I}_{i,j+1}}\dots\tau_{\mathcal{I}_{i,j'-1}}}$.

Finally a specific partition is provided which will help describe $d$-split Levi subgroups and their normalisers.
For any two positive integers $k,m\in\mathbb{N}_{>0}$ and $1\leq i\leq m$ set $$\mathcal{J}_i^{k,m}:=\{i+jm\mid 0\leq j\leq k-1\}\subset \ul{km},$$
and $\mathcal{J}^{k,m}:=\{\mathcal{J}_i^{k,m}\mid 1\leq i\leq m\}\vdash \ul{km}$. 
Note that $\mathcal{J}_i^{k,m}(j)=i+(j-1)m$.
Moreover $$C_{\sgnSymm_{km}}(w'_{\mathcal{J}^{k,m}})=\left( \prod_{i=1}^{m}\langle w'_{\mathcal{J}^{k,m}_i}\rangle \right)\rtimes \langle \tau_{\mathcal{J}^{k,m}_i}\mid 1\leq i\leq m-1\rangle\cong C_{2k}\wr \Symm_m,$$
and whenever $k$ is odd $\Cent_{\sgnSymm_{n}}((w'_{\mathcal{J}^{k,m}})^2)=\Cent_{\sgnSymm_{n}}(w'_{\mathcal{J}^{k,m}})$.

\subsubsection{Normalisers for certain subgroups in $\sgnSymm_n$}\label{NormsgnSymm}
\indent 

The next result considers normalisers in a wreath product of any abelian group by a symmetric group, which will be useful later. However it can also be immediately applied to $\sgnSymm_n$.

\begin{cor}\label{NormInWrOfWr}
Let $G=A\wr \Symm_n$ with $A$ an abelian group and $\mathcal{I}\vdash \ul{n}$.
To each $\mathcal{J}\in\mathcal{I}$ associate a subgroup $A_\mathcal{J}\leq A$ and let $H_\mathcal{I}=\prod_{\mathcal{J}\in\mathcal{I}} A_\mathcal{J}\wr\Symm(\mathcal{J})\leq G$.
Set $Z_\mathcal{I}:=\prod_{\mathcal{J}\in\mathcal{I}}\Delta_\mathcal{J}A$, where $$\Delta_{\mathcal{J}}A=\{(a_1,\dots,a_n)\in A^n\mid a_i=1 \text{ if } i\not\in\mathcal{J} \text{ and } a_i=a_j \text{ if } i,j\in\mathcal{J}\},$$ and $$S_{H_\mathcal{I}}=\langle \tau_{\mathcal{J},\mathcal{J}'}\mid |\mathcal{J}|=|\mathcal{J}'|\text{ and } A_\mathcal{J}=A_{\mathcal{J}'}\rangle.$$
Then $$\NNN_G(H_\mathcal{I})=\left(Z_\mathcal{I}H_\mathcal{I} \right)\rtimes S_{H_\mathcal{I}}.$$
\end{cor}
\begin{proof}
Under the projection homomorphism from $A^n\rtimes \Symm_n$ to $\Symm_n$ the group $\NNN_G(H_\mathcal{I})$ gets mapped to a subgroup of $\NNN_{\Symm_n}(H_\mathcal{I}\cap \Symm_n)$.
For the subgroup $S_\mathcal{I}:=\langle \ol{\tau}_{\mathcal{J},\mathcal{J}'}\mid |\mathcal{J}|=|\mathcal{J}'|\rangle$ it follows that $\NNN_{\Symm_n}(H_\mathcal{I}\cap \Symm_n)=\left(\prod_{\mathcal{J}\in\mathcal{I}}\Symm(\mathcal{J})\right)\rtimes S_\mathcal{I}$.
The group $S_\mathcal{I}$ permutes the factors of $A^n$ and thus $\NNN_{S_\mathcal{I}}(H_\mathcal{I}\cap A^n)=S_{H_\mathcal{I}}$.
Moreover, by definition $S_{H_\mathcal{I}}\leq \NNN_G(H_\mathcal{I})$ and hence
$$
\NNN_G(H_\mathcal{I})=\NNN_{A^n}(H_\mathcal{I})H_\mathcal{I}\rtimes S_{H_\mathcal{I}}.
$$

Let $\sigma\in \prod_{\mathcal{J}\in\mathcal{I}}\Symm(\mathcal{J})$ and take $g=(g_1,\dots, g_n)\in\NNN_{A^n}(H_\mathcal{I})$ so that $[g,\sigma]=g^{-1}g^{\sigma^{-1}}\in H_\mathcal{I}\cap A^n$. 
Hence for each $i\in \mathcal{J}$ it follows that $g_{\sigma(i)}\in g_iA_\mathcal{J}$.
In particular, for $g_\mathcal{J}:=(a_1,\dots, a_n)$ with $a_i=g_i$ whenever $i\in\mathcal{J}$ and $a_i=1$ otherwise, there is an $h_\mathcal{J}=(h_1,\dots, h_n)$ with $h_i\in A_\mathcal{J}$ whenever $i\in\mathcal{J}$ and $h_i=1$ otherwise, such that $g_\mathcal{J}h_\mathcal{J}\in \Delta_\mathcal{J}A$. 
Thus for $h=\prod_{\mathcal{J}\in\mathcal{I}}h_\mathcal{J}$ then $gh\in Z_\mathcal{I}$ and so $\NNN_{A^n}(H_\mathcal{I})\leq Z_\mathcal{I}H_\mathcal{I}$.
Moreover $[Z_\mathcal{I},\sigma]=1$ for any $\sigma\in \prod_{\mathcal{J}\in\mathcal{I}}\Symm(\mathcal{J})$, as for any $g=(g_1,\dots, g_n)\in Z_\mathcal{I}$ and $i,j\in\mathcal{J}$ then $g_i=g_j$.
Hence $Z_\mathcal{I}\leq \NNN_{A^n}(H_\mathcal{I})$.

\end{proof}

\begin{rem}\label{NormProdsgnSymmSym}
Let $n\in \mathbb{N}_{>0}$ and $\mathcal{I}\vdash n$.
Assume that $\mathcal{I}=\mathcal{I}_1\sqcup\mathcal{I}_2$ and set  
$$H_1=\prod_{\mathcal{J}\in \mathcal{I}_1}\Symm (\mathcal{J})
\text{ and } H_2=\prod_{\mathcal{J}\in \mathcal{I}_2}\sgnSymm (\mathcal{J}).$$
Then $H:=H_1\times H_2$ is a subgroup of $\sgnSymm_n\cong C_2\wr \Symm_n$ for which, in the notation of Corollary~\ref{NormInWrOfWr}, $H=H_\mathcal{I}$ where $A_\mathcal{J}=C_2$ whenever $\mathcal{J}\in \mathcal{I}_2$ and $A_\mathcal{J}$ is trivial otherwise.
Therefore  $\NNN_{\sgnSymm_n}(H)=\left(Z_\mathcal{I}H_\mathcal{I} \right)\rtimes S_{H_\mathcal{I}}$ with $Z_\mathcal{I}=\langle \iota_{\mathcal{J}}\mid \mathcal{J}\in\mathcal{I}\rangle$, where $\iota_{\mathcal{J}}:=\prod_{i\in\mathcal{J}} (i,-i)$.
Moreover $$
\NNN_{\sgnSymm_n}(H)/H  =  
       \GenGp{\iota_\mathcal{J},\tau_{\mathcal{J},\mathcal{J}'}\mid \mathcal{J},\mathcal{J'}\in\mathcal{I}_1 \text{ and } |\mathcal{J}|=|\mathcal{J}'|}\times  \GenGp{\tau_{\mathcal{J},\mathcal{J}'}\mid \mathcal{J},\mathcal{J'}\in\mathcal{I}_2 \text{ and } |\mathcal{J}|=|\mathcal{J}'|}.
$$
In particular, $\NNN_{\sgnSymm_n}(H)/H$ is a product of symmetric groups acting on $\mathcal{I}_2$ and a product of signed permutation groups acting on $\mathcal{I}_1$.
\end{rem}

\subsection{Characters of ${\rm GL}_n(q)$ and ${\rm GU}_n(q)$}
\indent 

Consider first the following result about wreath products.

\begin{lm}\label{ExtensionwithWreath}
Let $X\lhd Y$, $E\leq Y$ be finite groups with $Y=XE$.
Take $n\in \mathbb{N}_{>0}$ and assume $\theta_1,\dots,\theta_n\in \Irr(X)$ extend to their inertia subgroup in $Y$.
Then $\theta:=\theta_1\times \dots\times \theta_n\in \Irr(X^n)$ extends to $(Y\wr \Symm_n)_\theta$.
\begin{proof}
As any $E$-conjugate of $\theta_i$ extends to its inertia subgroup in $XE$, after $E^n$-conjugation, it can be assumed for all $1\leq i\ne j\leq n$ that either $\theta_i=\theta_j$ or $\theta_i$ and $\theta_j$ are not $E$-conjugate characters of $X$.  

The character $\theta$ yields a partition $\mathcal{I}\vdash \ul{n}$ where for $\mathcal{J}\in \mathcal{I}$ two numbers $i,j\in \mathcal{J}$ if and only if $\theta_i=\theta_j$ as characters in $\Irr (X)$.
In particular, for $\mathcal{J}\in\mathcal{I}$ denote by $\theta_{0,\mathcal{J}}\in \Irr(X)$ the character such that $\theta_i=\theta_{0,\mathcal{J}}$ for each $i\in\mathcal{J}$.
Set $$X_{\mathcal{J}}=\{(x_1,\dots,x_n)\in X^n\mid x_i=1 \text{ if } i\not\in \mathcal{J}\}\cong X^{|\mathcal{J}|}$$ and define $\theta_\mathcal{J}\in\Irr(X_\mathcal{J})$ by $\theta_\mathcal{J}(x_1,\dots,x_n)=\prod_{i\in\mathcal{J}}\theta_{0,\mathcal{J}}(x_i)$.
Then
$$(Y\wr\Symm_n)_\theta=\prod_{\mathcal{J}\in\mathcal{I}} (XE\wr \mathfrak{S}(\mathcal{J}))_{\theta_{\mathcal{J}}}=\prod_{\mathcal{J}\in\mathcal{I}} (XE_{\theta_{0,\mathcal{J}}}\wr \mathfrak{S}(\mathcal{J}))$$
and it suffices to extend $\theta_\mathcal{J}$ to $ (XE_{\theta_{0,\mathcal{J}}}\wr \mathfrak{S}(\mathcal{J}))$.

In an analogous fashion define $(XE_{\theta_{0,\mathcal{J}}})_{\mathcal{J}}\leq (XE)^n$ and a fixed extension $\wt{\theta}_{0,\mathcal{J}}$ of $\theta_{0,\mathcal{J}}$ to $XE_{\theta_{0,\mathcal{J}}}$ yields an extension $\wt{\theta}_{\mathcal{J}}$ of $\theta_\mathcal{J}$ to $(XE_{\theta_{0,\mathcal{J}}})_\mathcal{J}$.
By construction, $\mathfrak{S}(\mathcal{J})_{\wt{\theta}_\mathcal{J}}=\mathfrak{S}(\mathcal{J})$ and by [Huppert, 25.5] $\wt{\theta}_\mathcal{J}$ extends to $XE_{\theta_{0,\mathcal{J}}}\wr \mathfrak{S}(\mathcal{J})$, which must also be an extension of $\theta_\mathcal{J}$.
\end{proof}
\end{lm}

\begin{notation}
Let $\epsilon\in\{\pm 1\}$, $n$ a positive integer and $q$ a prime power.
In the following denote by ${\rm GL}_n(\epsilon q)$ the group ${\rm GL}_n(q)$ when $\epsilon=1$ and the group ${\rm GU}_n(q)$ when $\epsilon=-1$.
Furthermore, ${\rm SL}_n(\epsilon q):=[{\rm GL}_n(\epsilon q),{\rm GL}_n(\epsilon q)]$.
\end{notation}

\begin{cor}\label{ExtGnSymm}
Let $G={\rm GL}_n(\epsilon q)$ and $E\leq {\rm Aut}(G)$ a subgroup of the so called {\it field} and {\it graph} automorphisms of $G$ as described in \cite[2.5.1]{GLS3}.
Then each $\chi\in \Irr(G)$ extends to $G\rtimes E_\chi$.
Moreover for any $m\in \mathbb{N}_{>0}$ each character of $G^m$ extends to its inertia subgroup in $(GE)\wr \Symm_m$.
\end{cor}
\begin{proof}
By \cite[Theorem 4.1]{CabSpAMTypeA} there is a constituent $\chi_0\in \Irr ({\rm SL}_n(\epsilon q)\mid \chi)$ such that $\chi_0$ extends to ${\rm SL}_n(\epsilon q)\rtimes E_{\chi_0}$ and $(G\rtimes E)_{\chi_0}=G_{\chi_0}\rtimes E_{\chi_0}$.
Then $E_\chi\leq E_{\chi_0}$ and for $\wt{\chi}_0\in \Irr(G_{\chi_0})$   the Clifford correspondent to $\chi$, it follows that $E_\chi=E_{\wt{\chi}_0}$, that is $G\rtimes E_\chi=G\rtimes E_{\wt{\chi}_0}$.
Thus \cite[Theorem 5.8(a)]{CabSpAMTypeA} implies $\wt{\chi}_0$ will extend to its inertia subgroup in $G_{\chi_0}\rtimes E_{\chi}$ and hence the Mackey formula shows $\chi$ extends to $G\rtimes E_{\chi}$.
\end{proof}

The following result is adapted from \cite[Proposition 4.1]{SpTyDI} which dealt with $1$-cuspidal characters of ${\rm GL}_n(q)$.

\begin{prop}\label{CuspTypeA}
Let $n\geq 2$, $p$ a prime and $F$ a Steinberg endomorphism of ${\bf K}:={\rm GL}_n(\ol{\mathbb{F}}_q)$ with ${\bf K}^F\cong {\rm GL}_n(\epsilon q)$ for $q=p^f$ for some positive integer $f$.
Set $\gamma\in{\rm Aut}({\bf K}^F)$ to be given by transpose-inverse and let $\chi\in \Irr({\bf K}^F)$ with $\chi^\gamma=\chi$.
Assume that $ \chi$ is a $1$-cuspidal character when $\epsilon=1$ and $2$-cuspidal when $\epsilon=-1$.
Then $Z({\bf K}^F)\leq {\rm ker}(\chi)$.
\end{prop}
\begin{proof}
Denote by ${\bf K}^*:={\rm GL}_n(\ol{\mathbb{F}}_q)$ the dual group to ${\bf K}$.
For $q=p^f$ with $p$ a prime let $F_q=F_p^f$ with $F_p$ the standard Frobenius endomorphism on ${\bf K}^*$ which raises each matrix entry to the $p^{\rm th}$ power.
For $\epsilon\in \{\pm 1\}$ and $F^*:=\gamma^{\frac{1-\epsilon}{2}}\circ F_q$, up to some inner automorphism, it can be assumed that $({\bf K},F)$ is dual to the pair $({\bf K}^*,F^*)$.
In particular, $({\bf K}^*)^{F^*}\cong {\rm GL}_n(\epsilon q)$.
Set ${\bf T}_0^*$ to be the maximal torus consisting of diagonal matrices, which is $F^*$-stable, and let $W^*$ denote its associated Weyl group.
Observe that $W^*$ can be taken to consist of permutation matrices with entries are either zero or one, and on such matricies the actions of tranposition and inversion coincide.
Hence $F^*$ acts trivially this $W^*$.

The assumption of the proposition implies that for some semisimple element $s'\in({\bf K}^*)^{F^*}$ the Lusztig series $ \mathcal{E}({\bf K}^F,s')$ contains $\chi$ an $e$-cuspidal character, for $e=1$ when $\epsilon=1$ and $e=2$ when $\epsilon=-1$.
As ${\bf K}^*={\rm GL}_n(\ol{\mathbb{F}}_q)$ it follows that $C_{\bf K^*}(s')$ is connected.
Therefore \cite[Proposition 3.5.23]{GeckMal} implies that $\chi$ corresponds to a unipotent $e$-cuspidal character $\psi$ of $C_{\bf K^*}(s')^{F^*}$ and the Sylow $e$-torus of $Z(C_{\bf K^*}(s'))$ lies in $Z({\bf K}^*)$.

The element $s'$ lies in an $F^*$-stable maximal torus ${\bf T}^*$ \cite[Proposition 26.6]{TestMal}.
Moreover, by \cite[Remark 2.3.21]{GeckMal}, the bijection between $({\bf K}^*)^{F^*}$-conjugacy classes of $F^*$-stable maximal tori of ${\bf K}^*$ and the $F^*$-conjugacy classes of $W^*$ implies that there is an element $h$ such that $h^{-1}F^*(h)=w\in W^*$ and ${\bf T}^*=({\bf T}_0^*)^h$.
For this $h$ take $s\in{\bf T}_0^*$ such that $s^h=s'$ and note that 
$$(\Cent_{\bf K^*}(s'))^{F^*}=((\Cent_{\bf K^*}(s))^h)^{F^*}=(\Cent_{\bf K^*}(s))^{wF^*}.$$
Let $e_i$ denote the standard orthonormal basis of $\mathbb{R}^n$ and identify the root system arising from ${\bf T}_0^*$ with $\{e_i-e_j\mid 1\leq i,j\leq n,i\ne j\}$.
For a root $\alpha$ set $X_\alpha:=\langle x_{\alpha}(t)\mid t\in\ol{\mathbb{F}}_q \rangle$, where $x_{e_i-e_j}(t)-Id_n$ is the elementary matrix with entry $t$ at position $(i,j)$.
For the element $s={\rm diag}(s_1,\dots, s_n)\in {\bf T}_0^*$ with $s_i\in\ol{\mathbb{F}}_q^\times$, it follows that $$\Cent_{\bf K^*}(s)=\langle {\bf T}_0^*,X_{e_i-e_j}\mid s_i=s_j\rangle\cong\prod_{i=1}^r {\rm GL}_{n_i}(\ol{\mathbb{F}}_q),$$
where $s$ has $r$ distinct eigenvalues and each one occurs with multiplicity $n_i$.
Moreover, note that $w$ acts on this centraliser by permuting the copies on ${\rm GL}_{n_i}(\ol{\mathbb{F}}_q)$ as $s$ is $wF^*$-stable.
In particular, if $\lambda_1,\dots, \lambda_r$ denote the distinct eigenvalues of $s$, then the sets $\mathcal{J}_{i}:=\{j\mid s_j=\lambda_i\}$ are permuted by $w$ and 
$$\Cent_{\bf K^*}(s)^{wF^*}\cong \prod_{\mathcal{O}} {\rm GL}_{n_\mathcal{O}}(\epsilon_{\mathcal{O}}q^{d_\mathcal{O}})$$
where the product runs over the $w$-orbits on $\left\{ \mathcal{J}_{i}\mid 1\leq i\leq r\right\}$, $d_\mathcal{O}=|\mathcal{O}|$, $\epsilon_\mathcal{O}=\epsilon^{d_\mathcal{O}}$ and $n_\mathcal{O}=n_i$ with $\mathcal{J}_{i}\in\mathcal{O}$.
Thus the unipotent $e$-cuspidal character $\psi$ of $\Cent_{\bf K^*}(s)^{wF^*}$ decomposes as $\psi=\prod_{\mathcal{O}}\psi_{\mathcal{O}}$ with each $\psi_{\mathcal{O}}$ a unipotent $e$-cuspidal character of ${\rm GL}_{n_\mathcal{O}}(\epsilon_{\mathcal{O}}q^{d_\mathcal{O}})$.
In particular, as in \cite[Example 3.5.29]{GeckMal}, each $\psi_{\mathcal{O}}$ arises as either a unipotent $1$-cuspidal character of ${\rm GL}_{n_{\mathcal{O}}}(q^{m})$ or a unipotent $2$-cuspidal character of ${\rm GU}_{n_{\mathcal{O}}}(q^m)$ for some $m\in\mathbb{N}_{>0}$.
Then \cite[Corollary 4.6.5 and 4.6.7]{GeckMal} implies $\phi_\mathcal{O}$ corresponds to a partition of $n_\mathcal{O}$ which is a $1$-core which means each $n_\mathcal{O}=1$.
Hence $s$ has $n$ distinct eigenvalues and $\Cent_{{\bf K}^*}(s)$ is a torus, that is $\Cent_{\bf K^*}(s)={\bf T}_0^*$.
Expressing $w=\prod_{i=1}^ {k}\sigma_i$ as a product of disjoint cycles in $\mathfrak{S}_n$, it follows for $m_i=o(\sigma_i)$ that $$|({\bf T}_0^*)^{wF^*}|=\prod_{i=1}^{k}q^{m_i}-\epsilon^{m_i}.$$
However noting that $|Z({\bf K}^*)^{F^*}|=q-\epsilon\mid q^{m_i}-\epsilon^{m_i}$ and that the Sylow $e$-torus of $Z(C_{\bf K^*}(s'))$ lies in $Z({\bf K}^*)$ implies that $w$ must be an $n$-cycle.
Therefore as $s$ is fixed by  $wF^*$, this forces the $n$ distinct eigenvalues of $s$ to form an orbit of size $n$ under the action of $F^*$.
Now the argument as in \cite[Proposition 4.1]{SpTyDI} can be applied and finishes the proof.
\end{proof}

\section{Labelling $d$-split Levi subgroups in classical types}\label{LabdLe}

The main aim of this article is to study the representation theory of $d$-split Levi subgroups, which are characterised as the centralisers of $d$-tori.
In this section all classical types ${\tA}_n,$ $\tB_n$, $\tC_n$ and $\tD_n$ will be considered.

\subsection{Groups of Lie type}
\indent

Let $p$ be a prime and $\bG$ a simply connected simple algebraic group over an algebraic closure $\ol{\F}_p$ of $\F_p$.
Take ${\bf B}$ to be a Borel subgroup of $\bG$ with maximal torus ${\bf T}$ which in turn determines $\Phi$, $\Phi^+$ and $\Pi$ a set of roots, positive roots and simple roots respectively.
Set ${\bf N}:=\NNN_{\bG}({\bf T})$, and $W_{\Phi}:={\bf N}/{\bf T}$ the Weyl group of $\bG$ with corresponding homomorphism $\rho:{\bf N}\rightarrow W$.
Throughout, the Chevalley generators as given in \cite[1.12.1]{GLS3} are considered subject to the Chevalley relations.

\subsubsection{Chevalley relations}\label{ChevRel}
\indent

Every $\alpha\in \Phi$ determines a root subgroup $X_\alpha:=\GenGp{x_\alpha(t)\mid t\in \ol{\mathbb{F}}_q}$ of $\bG$ such that for two roots $\alpha,\beta\in \Phi$ and  $t,u\in \ol{\mathbb{F}}_q$, the commutator $[x_{\alpha}(t),x_{\beta}(u)]$ is a product of elements $x_{i\alpha+j\beta}(c_{i,j,\alpha,\beta}t^iu^j)$ for certain $c_{i,j,\alpha,\beta}\in \ol{\F}_q$ such that $i\alpha+j\beta\in \Phi$ for a pair of positive integers $i,j$. In particular, if there is no such pair of positive integers $i,j$ then $[X_\alpha,X_\beta]=1$. 

For each $t\in \ol{\mathbb{F}}_p^\times$ set ${\bf n}_\alpha(t):={\bf x}_\alpha(t){\bf x}_{-\alpha}(-t^{-1}){\bf x}_{\alpha}(t)$ and $ h_{\alpha}(t):={\bf n}_\alpha(1)^{-1}{\bf n}_{\alpha}(t)$ so that the maximal torus ${\bf T}=\GenGp{h_\alpha(t)\mid \alpha\in\Phi\text{ and }t\in  \ol{\mathbb{F}}_p^\times}$.
Given $\Pi=\{\alpha_1,\dots, \alpha_m\}$ it follows that $h_\alpha(t)=\prod\limits_{i=1}^m h_{\alpha_i}(t^{c_i})$ where $\alpha^\vee=\sum\limits_{i=1}^mc_i\alpha_i^\vee$ for $\alpha^\vee:=\frac{2}{(\alpha,\alpha)}\alpha$.
Furthermore, ${\bf n}_{\alpha}(1)^2=h_{\alpha}(-1)$ and $\prod_{i=1}^mh_{\alpha_i}(t_i)=1$ if and only if $t_i=1$ for all $i$.
As an immediate consequence, it follows that $h_{-\alpha}(t)=h_{\alpha}(t)^{-1}=h_{\alpha}(t^{-1})$ and ${\bf n}_{\alpha}(-1)={\bf n}_{\alpha}(1)^3$, that is $\GenGp{{\bf n}_{\alpha}(-1)}=\GenGp{{\bf n}_{\alpha}(1)}$.

For each root $\alpha\in \Phi$ let $r_\alpha$ denote the orthogonal reflection in the hyperplane $\alpha^\perp$.
Then $({\bf n}_{\alpha}(t))^{{\bf n}_\beta(1)}={\bf n}_{r_{\beta}(\alpha)}(c_{\alpha,\beta}t)$ for any two roots $\alpha,\beta\in\Phi$, $t\in \ol{\mathbb{F}}_p^\times$ and some $c_{\alpha,\beta}=\pm 1$. 
Thus $ {\bf n}_\alpha(t)^{{\bf n}_{\beta}(-1)}= {\bf n}_{r_\beta(\alpha)}(\pm t)$ as $r_{\beta}^3=r_{\beta}$ and so $\GenGp{ {\bf n}_\alpha(1)}^{{\bf n}_\beta(\pm1)}=\GenGp{  {\bf n}_{r_\beta(\alpha)}(1)}$.
Additionally $h_{\alpha}(t)^{{\bf n}_\beta(1)}=h_{r_\beta(\alpha)}(t)$.

For ease of notation ${\bf n}_\alpha:={\bf n}_\alpha (1)$ and $h_\alpha:=h_\alpha(-1)={\bf n}_\alpha^2$ will be adopted.
The {\it extended Weyl group} is defined to be $V_\Phi:=\GenGp{{\bf n}_\alpha\mid \alpha\in \Phi}$ and by construction  $\NNN_{\bG}({\bf T})={\bf T}V_\Phi$.
For $H_\Phi:=\GenGp{h_\alpha\mid \alpha\in\Phi}$ it follows that $|H_\Phi|=(2,q-1)^{|\Pi|}$, $H_\Phi={\bf T}\cap V_\Phi$ and $V_\Phi/H_\Phi\cong W_\Phi$.
In the following $\rho$ will denote the morphism $V_\Phi\rightarrow W_\Phi$.
\subsubsection{Endomorphisms of ${\bG}$}\label{endo}
\indent

Let $F_p$ be the field automorphism of $\bG$ which can be defined via the root subgroups $X_\alpha$, where $F_p(x_\alpha(t))=x_\alpha(t^p)$ for $t\in \ol{\F}_p$ and $\alpha\in \Phi$.
Any length-preserving automorphism of $\Phi$ stabilising $\Pi$, induces a graph automorphism $\gamma$ of $\bG$ satisfying $\gamma(x_\alpha(t))=x_{\gamma(\alpha)}(t)$ for $t\in \ol{\F}_p$ and $\alpha\in \pm\Pi$. 
Set $E_0\leq {\rm Aut}(\bG)$ to be the subgroup corresponding to the set of length-preserving automorphisms of $\Pi$ (denoted by $\Gamma_{\ol{K}}$ in \cite[1.15.5]{GLS3}).
Note that by construction the morphisms $F_p$ and $\gamma$ commute, where $\gamma$ is induced from any length-preserving automorphism of $\Pi$.

To construct the so called {\it diagonal automorphisms}  the notation of a regular embedding is required. The following is taken from \cite[4.1]{CabSpMcKTypC}.
Let $r$ denote the rank of $Z({\bf G})$ as an abelian group and ${\bf Z}\cong (\ol{\mathbb{F}}_q^\times)^r$ a torus of the same rank with an embedding of $Z(\bG)$. Set $\wt{\bG}:=\bG\circ_{Z({\bf G})}{\bf Z}$ to be the central product of $\bG$ and ${\bf Z}$ over $Z(\bG)$. Then the group $\wt{\bG}$ is a connected reductive group with connected centre and the embedding $\bG\rightarrow \wt{\bG}$ is a regular embedding in the sense of  \cite[15.1]{CabEnRedGp}.
Note that $\wt{\bf B}:={\bf B}{\bf Z}$ is a Borel subgroup of $\wt{\bG}$, $\wt{\bT}:=\bT{\bf Z}$ is a maximal torus and $\wt{\bf N}:=\NNN_{\wt{\bG}}(\wt{\bf T})={\bf N}{\bf Z}$. 
Moreover, as explained in \cite[Section 2.B]{MalSpOddDeg} the field and graph automorphisms introduced above can be extended to endomorphisms of $\wt{\bG}$ stabilising both $\wt{\bf B}$ and $\wt{\bf T}$, also denoted by $F_p$ and $\gamma$.

Throughout a Steinberg endomorphism $F:=F_p^m\gamma$ for some positive integer $m$ and $\gamma\in E_0$ will be considered.
Such an endomorphism $F$ determines an $\F_q$-structure on $\wt{\bG}$ where $q=p^m$. 
Let $G:=\bG^F$ and $\wt{G}:=\wt{\bG}^F$.
The order of the automorphisms $F_p$ and $\gamma$ on both groups $G$ and $\wt{G}$ coincides.
Moreover, the group $\wt{G}\rtimes \GenGp{F_p,E_0}$ is well defined and induces all automorphisms of $G$.
Let $e$ be a positive integer such that $(vF)^e=F_p^{em}$ for any $v\in V_\Phi$, set $E_1:=\Cent_{E_0}(\gamma)$ and $E:=C_{em}\times E_1$.
In particular, if $\phi$ denotes the automorphism of $\mathbb{Z}\Phi$ induced by $F$ then one could take $e=o(\phi){\rm exp}(E_1)|V_\Phi|$. 
An action of $E$ on $\bG^{F_p^{em}}$ is defined via the generating element $\wh{F}_p$ of $C_{em}$ acting as $F_p$ and the second factor acting as the standard graph automorphisms.
Then $\wh{F}$ denotes the element $\wh{F}_p^m\gamma$ of $E$ inducing the automorphism $F$ on $\bG^{F_p^{em}}$.

For any $F$-stable subgroup ${\bf K}\leq \bG$ the normaliser in $\bG^F\rtimes E$ is well defined and will be denoted by $(\bG^F\rtimes E)_{\bf K}$.

\subsubsection{Relative Weyl groups}\label{relWeylGp}
\indent

Let ${\bf L}$ be a Levi subgroup of $\bG$, that is, the centraliser of some torus in $\bG$.
The relative Weyl group of ${\bf L}$ in $\bG$ (see \cite[Section 4]{BMGenHC}) is defined to be 
$$W_{\bG}({\bf L}):=\NNN_{\bG}({\bf L})/{\bf L}\cong \NNN_{W_{\bG}}(W_{\bf L})/W_{\bf L}.$$
For $F$ a Steinberg endomorphism of $\bG$, if $\bL$ is $F$-stable, it follows that $\NNN_{\bG}({\bf L})$ is also $F$-stable.
Hence $F$ induces an automorphism $\sigma$ on $W_{\bG}(\bL)$ and by applying the Lang--Steinberg theorem, it follows that $$
\NNN_{\bG}({\bf L})^F/\bL^F\cong \Cent_{W_\bG(\bL)}(\sigma).
$$

\subsection{Weyl groups and extended Weyl groups}

\subsubsection{Weyl groups in classical type}\label{WeylGp}
\indent

Fix $\{e_i\}$ an orthonormal basis of $\mathbb{R}^n$, then as in \cite[1.8.8]{GLS3}, the root systems for classical types can be taken to be:
$$
\begin{array}{ccc}
\Phi_{A_{n-1}} & = & \{ e_i-e_j\mid 1\leq i\ne j\leq n\}\\
\Phi_{B_n} & = & \{ e_i\pm e_j, \pm e_i \mid 1\leq i\ne j\leq n\}\\
\Phi_{C_n} & = & \{ e_i\pm e_j, \pm 2e_i \mid 1\leq i\ne j\leq n\}\\
\Phi_{D_n} & = & \{ e_i\pm e_j\mid 1\leq i\ne j\leq n\}\\
\end{array}
$$

For systems of type $B_n$ and $C_n$ it will be assumed $n\geq 2$, while for those of type $D_n$ it will be assumed $n \geq 4$.

\begin{notation}
For $\mathcal{J}\subseteq \ul{n}$ and a root system $\Phi\subseteq \langle e_i\mid 1\leq i\leq n\rangle$ set $\Phi_\mathcal{J}=\Phi\cap \langle e_i\mid i\in\mathcal{J}\rangle$.
In particular, let $\Phi_{A_{\mathcal{J}}}:=\Phi_{A_{n-1}}\cap\langle e_i\mid i\in\mathcal{J}\rangle$ and similarly define $\Phi_{B_\mathcal{J}}$, $\Phi_{C_\mathcal{J}}$ and $\Phi_{D_\mathcal{J}}$.
Additionally for a partition $\mathcal{I}=\{\mathcal{J}_i\mid 1\leq i\leq r\}\vdash\ul{\mathcal{I}}$, then $\Phi_{A_{\mathcal{I}}}:=\sqcup_{i=1}^r \Phi_{A_{\mathcal{J}_i}}$.
\end{notation}

\begin{rem}\label{CanWeylGpHom}
Recall that if $\Phi$ is a root system of type $A_{n-1}$ ($n\geq 3$) or $D_n$ there is a length preserving graph automorphism $\gamma$ stabilising the set of simple roots $\Pi$, which has order $2$.
By defining $e_{-i}:=-e_i$, as in \cite[Notation 3.1]{SpSylowTori2}, for $\Phi$ of classical type as above there is a canonical homomorphism $f:\langle W_\Phi,\gamma\rangle \rightarrow \Symm_{\pm n}$ such that $x\cdot e_i=e_{f(x)(i)}$  for any $x\in \langle W_\Phi,\gamma\rangle$ .
Moreover, for $\ol{f}: \langle W_\Phi,\gamma\rangle \rightarrow \Symm_n$ given by composing $f$ with the projection map, each
$x\in \langle W,\gamma\rangle$ determines a sign function $\epsilon_{f(x)}:\ul{n}\rightarrow \{\pm 1\}$ via the equation $x\cdot e_i=\epsilon_{f(x)}(i) e_{\ol{f}(x)(i)}$.  

For $r_\alpha$, the reflection corresponding to a root $\alpha\in \Phi$, direct calculation yields $f(r_{e_i-e_j})=(i,j)(-i,-j)$, $f(r_{e_i+ e_j})=(i,-j)(-i,j)$ and $f(r_{2e_i})=f(e_i)=(i,-i)$.
Hence it follows that $f(W_{C_n})=f(W_{B_n})=\sgnSymm_n$ and $f(W_{A_{n-1}})=\Symm_n\leq \sgnSymm_n$.
Furthermore noting that $f(r_{e_i-e_j})f(r_{e_i+e_j})=(i,-i)(j,-j)$ implies $f(W_{D_n})=\langle (i,-i)(j,-j)\mid 1\leq i\ne j\leq n\rangle\rtimes \Symm_n$, which has index $2$ in $\sgnSymm_n$. 
\end{rem}

Up to $W_\Phi$-conjugation, the following set of simple roots $\Pi$ and thus also the corresponding set of positive roots $\Phi^+$ are fixed:
$$
\begin{array}{ccc}
\text{ Type } & \Pi & \Phi^+ \\
A_{n-1}& \{e_i-e_{i+1}\mid 1\leq i\leq n-1\} &  \{e_i-e_j\mid 1\leq i<j\leq n\}\\
{B_n} & \{e_i - e_{i+1}, e_n \mid 1\leq i\leq n-1\} &  \{e_i\pm e_j, e_k \mid 1\leq i<j\leq n, 1\leq k\leq n\}\\ 
{C_n} & \{e_i - e_{i+1}, 2e_n \mid 1\leq i\leq n-1\} & \{e_i\pm e_j, 2e_k \mid 1\leq i<j\leq n, 1\leq k\leq n\}\\ 
{D_n} & \{e_i - e_{i+1}, e_{n-1}+e_n \mid 1\leq i\leq n-1\} & \{e_i\pm e_j \mid 1\leq i<j\leq n\}\\ 
\end{array}
$$ 

For the set of simple roots $\Pi_{D_n}$ ($n\geq 4$) given above, the graph automorphism $\gamma$ of order $2$ swaps the vertices $e_{n-1}+e_{n}$ and $e_{n-1}-e_n$ and fixes the roots $e_i-e_{i+1}$ for $1\leq i\leq n-2$.
Hence by considering $W_{D_n}\leq W_{B_n}$, this graph automorphism is induced by the reflection $r_{e_n}$.
In particular, for this choice, $f(\gamma)=f(r_{e_n})=(n,-n)$.

\subsubsection{Extended Weyl group}\label{ExtWeylGp}
\indent

Recall from Section~\ref{ChevRel} that the extended Weyl group was defined to be $V_\Phi:=\GenGp{{\bf n}_\alpha\mid \alpha\in \Phi}$ and contained the normal subgroup $H_\Phi:=\GenGp{h_\alpha\mid \alpha\in\Phi}$ of order $(2,q-1)^{|\Pi|}$ such that $V_\Phi/H_\Phi\cong W_\Phi$.
If $\Phi_1,\Phi_2$ are orthogonal root subsystems of $\Phi$ and $\Pi_i$ is a set of  simple roots in $\Phi_i$, then $\Pi_1\sqcup\Pi_2$ forms a set of simple roots in $\Phi_1\sqcup\Phi_2$.
Moreover, $W_{\Phi_1\sqcup \Phi_2}\cong W_{\Phi_1}\times W_{\Phi_2}$ and by comparing group orders $H_{\Phi_1\sqcup \Phi_2}=H_{\Phi_1}\times H_{\Phi_2}$.
Hence $V_{\Phi_1}\cap V_{\Phi_2}\leq H_\Phi\cap V_{\Phi_i}=H_{\Phi_i}$ for $i=1,2$.
However $H_{\Phi_1}\cap H_{\Phi_2}$ is trivial and thus $V_{\Phi_1}\cap V_{\Phi_2}=1$.

\begin{cor}\label{IDisjVDirProd}
Let $\Phi$ be a root system of type $A_{n-1},C_n$ or $D_n$ and  $\mathcal{J}_1,\dots, \mathcal{J}_k$ disjoint subsets of $\ul{n}$.
Then $\langle V_{\Phi_{\mathcal{J}_i}}\mid 1\leq i\leq k\rangle= V_{\Phi_{\mathcal{J}_1}}\times \dots \times V_{\Phi_{\mathcal{J}_k}}$.
If furthermore there is some $\mathcal{J}\subset \ul{n}$ and elements $g_1,\dots, g_k\in V_{\Phi}$ such that $(V_{\Phi_\mathcal{J}})^{g_i}\leq V_{\Phi_{\mathcal{J}_i}}$, then the map 
$$
\begin{array}{ccc}
V_{\Phi_{\mathcal{J}}} & \rightarrow & V_{\Phi} \\
 x & \mapsto & \prod_{i=1}^k x^{g_i}\\
\end{array}
$$
is an injective group homomorphism.
\end{cor}

\begin{proof}
As $\Phi$ is a root system of type $A_{n-1}$, $C_n$ or $D_n$, the Chevalley relations imply $[x_\alpha(t),x_\beta(u)]=1$ for any pair of orthogonal roots $\alpha$ and $\beta$ in $\Phi$.
Thus for any $l\leq k-1$, it follows that $[V_{\Phi_{\mathcal{J}_1\cup\dots \cup \mathcal{J}_l}} , V_{\Phi_{\mathcal{J}_{l+1}}}]=1$.
Moreover, as the root subsystems $\Phi_{\mathcal{J}_1\cup\dots \cup \mathcal{J}_l}$ and $\Phi_{\mathcal{J}_{l+1}}$ are orthogonal, the intersection $V_{\Phi_{\mathcal{J}_1\cup\dots \cup \mathcal{J}_l}} \cap V_{\Phi_{\mathcal{J}_{l+1}}}=1$.
Hence the first statement follows by an inductive argument after observing that $\langle V_{\Phi_{\mathcal{J}_i}}\mid 1\leq i \leq l\rangle \leq V_{\Phi_{\mathcal{J}_1\cup\dots\cup \mathcal{J}_l}}$.

The defined map is a homomorphism of groups as the commutators $[V_{\Phi_{\mathcal{J}_i}},V_{\Phi_{\mathcal{J}_j}}]=1$ whenever $i\ne j$.
Moreover, $\langle V_{\Phi_{\mathcal{J}_i}}\mid 1\leq i\leq k-1 \rangle\cap V_{\Phi_{\mathcal{J}_k}}=1$, implies that the map must be injective.
\end{proof}

\begin{rem}
Classical type $B_n$ has to be excluded from Corollary~\ref{IDisjVDirProd} as $e_1$ and $e_2$ are orthogonal but $e_1+e_2\in\Phi_{B_n}$.
Thus for two orthogonal root subsystems $\Phi_1$ and $\Phi_2$ in $\Phi_{B_n}$, the groups $V_{\Phi_1}$ and $V_{\Phi_2}$ need not commute.
\end{rem}

\subsubsection{Parabolic root subsystems}
\indent
Let $\Phi$ be a root system.
A root subsystem $\Gamma$ of $\Phi$ is called parabolic if there is a set of simple roots $\Pi$ of $\Phi$ such that some subset $I\subset \Pi$ forms a set of simple roots of $\Gamma$.

\begin{lm}\label{ParaSubSys}
Let $\Phi\in \{A_{n-1}, B_n, C_n,D_n\}$ and $\Phi'$ a parabolic root subsystem.
Then up to $W_\Phi$-conjugation, there is a integer $ 0\leq m\leq n$ and a partition $\mathcal{I}\vdash \ul{m}$ such that
 $$\Phi'=\Phi_{\ul{n}\setminus \ul{m}}\sqcup \Phi_{A_{\mathcal{I}}},$$
except in type $D_n$ where there is also an additional case $\Phi'=r_{e_n}(\Phi_{A_{\ul{n}\setminus\ul{m}}})\sqcup \Phi_{A_{\mathcal{I}}}$. 
\begin{proof}
Up to $W_\Phi$-conjugation it can be assumed that $\Phi'$ is generated by a subset $I\subseteq \Pi$, for the $\Pi$ given in Section~\ref{WeylGp}.

If $I\subset \Pi_{A_{n-1}}$, then there exists a partition on $\ul{n}$ defined via $i\sim j$ if and only if $e_i-e_j\in I$.
Thus there exists a partition $\mathcal{I}$ of $\ul{n}$ such that $\Phi'=\Phi_{A_{\mathcal{I}}}$.

For types $B_n$ and $C_n$, if $e_n\not\in I$ (respectively $2e_n$), then $I\subseteq \Pi_{A_{n-1}}$ and the result follows.
Otherwise, there exists a maximal $m\leq n$ such that $e_{m-1}-e_m\not\in I$.
Then it follows that $\Phi'=(\Phi\cap \langle e_i\mid m\leq i\leq n\rangle)\sqcup \Delta$, with $\Delta$ a subsystem of $\Phi_{A_{m-2}}$.

Similarly in type $D_n$ it can be assumed that $e_{n-1}+e_n\in I$.
Let $m$ be maximal such that $e_{m-1}-e_m\not\in I$.
If $m<n$, it follows that $e_{i}-e_{j}\in \Phi'$ for $i,j\geq m$ and as $\Phi'$ is a subsystem $r_{e_{n-1}-e_j}(r_{e_n-e_i}(e_n+e_{n-1})=e_i+e_j\in \Phi'$.
Thus $\Phi'=(\Phi\cap \langle e_i\mid m\leq i\leq n\rangle)\sqcup \Delta$, with $\Delta$ a subsystem of $\Phi_{A_{m-2}}$.
If $m=n$, then $\{ e_n+e_{n-1},e_{n-2}-e_{n-1},\dots, e_{k}-e_{k+1}\}\subset I$ for some $k$ and by the same argument as before $\{e_n+e_i, e_i-e_j\mid k\leq i,j<n\}\subset \Phi'$.
Hence $\Phi'=\{e_n+e_i, e_i-e_j\mid k\leq i,j<n\}\sqcup \Delta$, with $\Delta$ a subsystem of $A_{k-2}$.
\end{proof}
\end{lm}

\begin{rem}
If $n-m\geq 2$ in type $B_n,C_n$ (respectively  $n-m \geq 4$ type $D_n$), then $\Phi_{\ul{n}\setminus\ul{m}}=\Phi_{B_{\ul{n}\setminus\ul{m}}},\Phi_{C_{\ul{n}\setminus\ul{m}}}$ (respectively $\Phi_{D_{\ul{n}\setminus\ul{m}}}$).
The following table explicitly describes the structure of $\Phi_{\ul{n}\setminus\ul{m}}$ in the remaining cases:
$$
     \begin{array}{c|c|c|c}
     \text{ Type: } & {B_n, C_n} &  \multicolumn{2}{|c}{D_n}  \\
     n-m & 1 &  2  & 3 \\
     \Phi_{\ul{n}\setminus\ul{m}} & \{ e_n\}, \{2e_n\} &  \Phi_{A_{\mathcal{I}_{-1}}} \sqcup r_{e_n}(\Phi_{A_{\mathcal{I}_{-1}}}) & \{\pm e_i \pm e_j\} \\
\text{ Isom. Type}  & \Phi_{A_1}  & \Phi_{A_1}\times \Phi_{A_1}  & \Phi_{A_3} \\
     \end{array}
$$
\end{rem}

\begin{notation}\label{ParamI}
Let $\Phi$ a root system of type $\tA_n$, $\tB_n$, $\tC_n$ or $\tD_n$. 
For an integer $0\leq m\leq n$ and $\mathcal{I}\vdash \ul{m}$ set $\Phi_{m,\mathcal{I}}:=\Phi_{\ul{n}\setminus\ul{m}}\sqcup\Phi_{A_\mathcal{I}}$.
The case $\Phi$ is of type $\tD_n$ and $m=n-1$ is omitted as in this case $\Phi_{n-1,\mathcal{I}}=\Phi_{n,\mathcal{I}\sqcup\{n\}}$.
When $\Phi$ is of type $\tA_{n-1}$ then $\Phi_{m,\mathcal{I}}=\Phi_{n,\mathcal{I}\sqcup\{\ul{n}\setminus\ul{m}\}}$, hence it shall be assumed in this case $m=n$.
\end{notation}

\begin{cor}\label{xStabPara}
Let $\Phi$ be a root system of classical type, $\Phi'=\Phi_{m,\mathcal{I}}$ be a parabolic subsystem as in Notation~\ref{ParamI} and $x\in \langle W_\Phi,\gamma\rangle$. 
Then $\Phi'$ is $x$-stable if and only if $\mathcal{I}$ is $\ol{f}(x)$-stable and $\epsilon_{f(x)}\mid_{\mathcal{J}}$ is constant for all $\mathcal{J}\in \mathcal{I}$, for the functions $\ol{f}$ and $\epsilon_{f(x)}$ defined in Remark~\ref{CanWeylGpHom}.
\begin{proof}
This follows by noting that $\Phi_{\ul{n}\setminus\ul{m}}$ contains roots of the form $e_i$ or $\pm(e_i\pm e_j)$ and $x\cdot (e_i-e_j)=\epsilon_{f(x)}(i)e_{\ol{f}(x)(i)}-\epsilon_{f(x)}(j)e_{\ol{f}(x)(j)}$.
\end{proof}
\end{cor}

\begin{rem}
Let $\Gamma$ be a parabolic subsystem of $\Phi$ and let $w\in W_{\Phi}$ such that $\Gamma:=w\cdot \Phi_{m,\mathcal{I}}$ for $\Phi_{m,\mathcal{I}}$ as in Notation~\ref{ParamI}.
Then for $x\in \langle W_{\Phi},\gamma\rangle$, the system $\Gamma$ is $x$-stable if and only if $\Phi_{m,\mathcal{I}}$ is $x^w$-stable.
In particular the parabolic subsystem $r_{e_n}(\Phi_{A_{\ul{n}\setminus\ul{m}\sqcup \mathcal{I}}})$ in type $D_n$ can be dealt with by considering when $\Phi_{A_{\ul{n}\setminus\ul{m}\sqcup \mathcal{I}}}$ is $x^{(n,-n)}$-stable, 
 i.e. $\ol{f}(x^{(n,-n)})\cdot\mathcal{J}\in \mathcal{I}\sqcup \{\ul{n}\setminus\ul{m}\}$ and $\epsilon_{x^{(n,-n)}}|_\mathcal{J}$ is constant for each $\mathcal{J}\in\mathcal{I}\sqcup  \{\ul{n}\setminus\ul{m}\}$.
\end{rem}

\subsection{$d$-split Levi subgroups}
\indent

In the following the triality group arising in type $D_4$ will be excluded. A $d$-split Levi subgroup of ${\bG}$ is defined to be the centraliser of some $d$-torus of ${\bG}$ \cite[3.5.1]{GeckMal}.
Any $d$-torus lies within a maximal $d$-torus, and therefore any $d$-split Levi subgroup contains a minimal $d$-split Levi subgroup.
Moreover, for $({\bf G},F)$ a connected reductive group of Lie type, a Sylow $d$-torus is determined by a so called Sylow twist $vF$ (see \cite[Section 2]{SpSylowTori2}), for some element $v\in \langle V, {\bf T}\rangle$ with $V$ the extended Weyl group in ${\bf G}$.

\subsubsection{Sylow twists}
\indent

Assume that ${\bf G}^F$ is a classical group and ${\bf n}\in \NNN_{\bf G}({\bf T})$.
The element ${\bf n}F$ acts on ${\bf G}$ via conjugation and induces an element $\rho({\bf n})\phi\in\langle W,\gamma\rangle$ (as the triality group in type $D_4$ has been excluded).

\begin{rem}\label{SylowTwist}\label{d0Defn}
In \cite{SpSylowTori2} suitable Sylow twists have been provided, here only the required notation is recalled.

Let $d$ be a positive integer.
Set 
$$
\begin{array}{c|c|c} 
A_{n-1} &  {}^2A_{n-1} & B_n,C_n, D_n,{}^2D_n\\
\hline
d_0:=d 
&
 d_0:= \left\{
     \begin{array}{lr}
       d & 4\mid d\\
       d/2 & d\equiv 2 \text{ mod }4\\
       2d & d \text{ odd }\\
      \end{array}
   \right.
& 
 d_0:= \left\{
     \begin{array}{lr}
       d & d \text{ odd }\\
       d/2 & d \text{ even }\\
      \end{array}
   \right.
\\
\end{array}
$$
It can be determined whether ${\bf n}F$ is a Sylow $d$-twist of $({\bf G},F)$ by understanding the orbits of $\pi:=f(\rho({\bf n})\phi)$ on the set $\{\pm 1,\dots, \pm n\}$, for $f$ as defined in Remark~\ref{CanWeylGpHom} (see \cite[Remark 3.2]{SpSylowTori2}).
In particular, each orbit must have length $d_0$. 
\end{rem}

In the next section a condition upon the partition $\mathcal{I}$ will be provided to determine when the parabolic subsystem $\Phi_{m,\mathcal{I}}$ is $\ol{\pi}$-stable.
To achieve this the following result concerning powers of a primitive $d^{\rm th}$ root of unity and the sign function $\epsilon_\pi$ from Remark~\ref{CanWeylGpHom} is required. 

\begin{cor}\label{Znepm1}
Let $\zeta$ be a primitive $d^{\text{th}}$ root of unity and ${\bf n}F$ be a Sylow twist with $\pi:=f(\rho({\bf n})\phi)$.
If $\bG$ is of type $A_{n-1}$ then $\zeta^k\ne \epsilon_{\pi^k}(i)$ for any $0<k<d_0$.
Meanwhile if $\bG$ is of type $B_n,C_n,D_n$, then $\zeta^k\ne\pm \epsilon_{\pi^k}(i)$ for any $0<k<d_0$.
\end{cor}
\begin{proof}
First consider types $B_n,C_n$ and $D_n$.
If $d$ is odd, then $d=d_0$ and hence $\zeta^k=\pm 1$ implies that $d$ divides $k$.
For $d$ even, $d_0=\frac{d}{2}$ and $\zeta^k=\pm 1$ implies that $\frac{d}{2}$ divides $k$.

In type $A_{n-1}$, if $\phi=1$, then $\pi\in\Symm_n\leq \sgnSymm_n$ but $\epsilon_{\pi^k}(i)=1$ and so $\zeta^k=1$ implies $d$ divides $k$.
Thus assume $\phi$ to be non-trivial so that by \cite[6.9]{SpringReg} $\epsilon_{\pi^k}(i)=(-1)^k$.
If $d$ is odd, then $d_0=2d$ and $\zeta^k=\pm 1$ implies $d$ divides $k$. However $\zeta^d=1\ne -1=(-1)^d=\epsilon_{\pi^d}(i)$.
If $d$ is even then $\zeta^k=\pm 1$ implies that $d/2$ divides $k$.
As $d_0=d/2$ when $d\equiv 2$ mod $4$, it remains to consider the case $4\mid d$, with $d_0=d$.
If $k=\frac{d}{2}$, then $\zeta^k=-1$ but $\epsilon_{\pi^k}(i)=(-1)^k=1$.
\end{proof}

\subsubsection{Root systems for $d$-split Levi subgroups}\label{RootSysdLevi}
\indent

Take $\{e_1,\dots, e_n\}$ the standard basis of $\mathbb{R}^n$ that is orthonormal with respect to $\langle \indent ,\indent \rangle$.
Then in Section~\ref{WeylGp} the root system of $\bG$ with respect to the fixed maximal torus $\bT$ was identified with $\Phi\subset \{e_i\pm e_j,\pm2e_i,\pm e_i\mid 1\leq i\ne j\leq n\}$.
Let $\Phi^\vee:=\{\alpha^\vee\mid \alpha\in\Phi\}$ denote the set of co-roots as defined in Section~\ref{ChevRel}.
Denote by $\calV$ the $\mathbb{C}$-vector space defined by $\calV:=\Phi^\vee\otimes_{\mathbb{Z}} \mathbb{C}$, which by identifying $\alpha^\vee$ with the co-character $h_\alpha$, coincides with $Y(\bT)\otimes_\mathbb{Z}\mathbb{C}$ for $Y(\bT)$ the group of all co-characters of $\bT$. 

Let  ${\bf n}F$ be a Sylow $d$-twist of $({\bf G},F)$ and $\pi:=f(\rho({\bf{n}})\phi)\in \sgnSymm_n$. 
For $\zeta\in \mathbb{C}$ a primitive $d^{\rm th}$-root of unity set ${\calV}(\pi,\zeta)$ to be the eigenspace of $\pi$ on ${\calV}$ with eigenvalue $\zeta$.
By Remark~\ref{CanWeylGpHom}, $\pi$ defines a function $\epsilon_{\pi}$ and $\ol{\pi}:=\ol{f}(\rho({\bf{n}})\phi)\in \Symm_n$. Hence
$$ \calV(\pi,\zeta)\subseteq\{ \sum_{i=1}^n a_ie_i\mid  a_i\in \mathbb{C} \text{ with } a_{\ol{\pi}^k(i)}=\epsilon_{\pi^k}(i) \zeta^{-k} a_i \text{ for each } 1\leq i\leq n \text{ and } 1\leq k\leq d_0\}.$$

Assume that ${\bL}$ is a $d$-split Levi subgroup in ${\bG}$ with respect to $(\bG, {\bf n}F)$.
As $\bL$ is the centraliser of some $d$-torus of $\bG$, after conjugation in $\bG^{{\bf n}F}$, it can be assumed that $\bT\leq \bL$.
Let  $\Phi_{{\bL}}$ be the root system of $\bL$ with respect to $\bT$.
Then $\Phi_{\bL}$ is a $\pi$-stable parabolic root subsystem of $\Phi$ with ${\bf L}=\GenGp{{\bf T}, X_\alpha\mid \alpha\in\Phi_{{\bf L}}}$.
According to \cite[Proposition 3.5.5]{GeckMal}, ${\bf L}=\Cent_{{\bf G}}(Z({\bf L})^\circ_d)$, where $Z(\wt{\bf L})^\circ_d$ denotes the Sylow $d$-torus of $Z({\bf L})^\circ$.
As in \cite[Section 3B]{BrSpAM}, this translates to  
\begin{equation}\label{eq:wStabRoot}
\left( {\calV}(\pi,\zeta)\cap \Phi_{\bL}^{\perp} \right)^{\perp}\cap \Phi=\Phi_{\bL}.
\end{equation}

\begin{rem}
Lemma~\ref{ParaSubSys} shows it suffices to focus on the $\pi$-stable parabolic subsystems of the form $\Phi_{m,\mathcal{I}}$ for some $m\leq n$ and $\mathcal{I}\vdash \ul{m}$ as defined in Notation~\ref{ParamI}. 
As it is assumed $m\ne n-1$ in type $\tD_n$ and $m=n$ in type $\tA_{n-1}$, it follows that
$$\Phi_{m,\mathcal{I}}^\perp\subseteq\{\sum_{i=1}^n a_ie_i\mid a_i=0 \text{ for all } i\in\ul{n}\setminus\ul{m} \text{ and } a_i=a_j\in\mathbb{C} \text{ whenever } \{i,j\}\subset\mathcal{J}\in\mathcal{I}\}. 
$$
By Corollary~\ref{xStabPara}, $\mathcal{I}$ is $\ol{\pi}$-stable and $\epsilon_\pi\mid_{\mathcal{J}}$ is constant for all $\mathcal{J}\in\mathcal{I}$.
The description of $\calV (\pi,\zeta)$ implies
$$\mathcal{I}_{-1}:=\{i\in\ul{n}\mid e_i\in \left( {\calV}(\pi,\zeta)\cap \Phi_{m,\mathcal{I}}^\perp \right)^{\perp}\}$$
is a $\ol{\pi}$-stable set containing $\ul{n}\setminus\ul{m}$.  
\end{rem}

\begin{lm}\label{PartwStableRootSubsys}
Let ${\bf n}\in\NNN_{\bf G}({\bf T})$ such that ${\bf n}F$ is a Sylow $d$-twist with respect to $\bT\leq \bG$.
Set $\pi=f(\rho({\bf{n}})\phi)$ and $\ol{\pi}=\ol{f}(\rho({\bf{n}})\phi)$.
Take $\Phi':=\Phi_{m,\mathcal{I}}$ a parabolic subsystem of $\Phi$, as defined in Notation~\ref{ParamI}, which is $\pi$-stable.
Then $\Phi'$ satisfies Equation {\rm (\ref{eq:wStabRoot})} if and only if $\mathcal{I}_{-1}\in\mathcal{I}\sqcup \{\ul{n}\setminus\ul{m}\}$ and each $\ol{\pi}$-orbit on $\mathcal{I}\setminus\mathcal{I}_{-1}$ has length $d_0$.
\begin{proof}
Set $\mathcal{U}:= {\calV}(\pi,\zeta)\cap \Phi_{m,\mathcal{I}}^\perp$.
Any element $a\in \mathcal{U}$ can be expressed as $a=\sum_{i=1}^n \langle a,e_i\rangle e_i$, thus set $a_i:=\langle a,e_i\rangle\in \mathbb{C}$.  

First assume that $\mathcal{U}^\perp\cap \Phi=\Phi'$.
If $i,j\in \mathcal{I}_{-1}$, then $e_i$, $2e_i$ and $e_i\pm e_j$ are in the vector space $\mathcal{U}^\perp$.
In particular $e_i-e_j\in \mathcal{U}^\perp\cap \Phi=\Phi_{\ul{n}\setminus\ul{m}}\sqcup\Phi_{A_{\mathcal{I}}}$.
Therefore $\mathcal{I}_{-1}\in \mathcal{I}\sqcup \{\ul{n}\setminus\ul{m}\}$.
Let $\mathcal{J}\in \mathcal{I}\setminus\mathcal{I}_{-1}$ and assume that $\ol{\pi}^k\cdot \mathcal{J}=\mathcal{J}$ for some $1\leq k\leq d_0$.
For any $i\in\mathcal{J}$, there is an $a\in \mathcal{U}$ such that $a_i\ne 0$ and $a_i=\epsilon_{\pi^k}(i)\zeta^{-k}a_{i}$.
Hence by Corollary~\ref{Znepm1}, $k=0$ or $d_0$.

Now consider the converse.
It suffices to show that $\mathcal{U}^\perp\cap \Phi\subseteq \Phi'=\Phi_{\mathcal{I}_{-1}}\sqcup \Phi_{A_{\mathcal{I}\setminus\mathcal{I}_{-1}}}$.
First take an $i\in \mathcal{I}_{-1}$ and $j\in\ul{n}\setminus\mathcal{I}_{-1}$.
If there is some $e_i\pm e_j\in \mathcal{U}^\perp$, then the coefficient $a_j=0$ for all $a\in \mathcal{U}$ and so $j\in \mathcal{I}_{-1}$.
Hence $$\mathcal{U}^\perp\cap\Phi\subseteq \Phi_{\mathcal{I}_{-1}} \sqcup \Phi_{\ul{n}\setminus\mathcal{I}_{-1}}.$$
The assumption about the action of $\ol{\pi}$ on $\mathcal{I}$ implies that for any $\mathcal{J}\in\mathcal{I}\setminus \mathcal{I}_{-1}$ the element $$a_\mathcal{J}:=\sum\limits_{i\in\mathcal{J}} e_i+\epsilon_{\pi}(i)\zeta^{-1} e_{\ol{\pi}(i)}+\dots + \epsilon_{\pi^{d_0-1}}(i)\zeta^{-(d_0-1)} e_{\ol{\pi}^{d_0-1}(i)}\in \mathcal{U}.$$
Thus for any $i\in \mathcal{J}$ neither $e_i$ or $2e_i$ lies in $\mathcal{U}^\perp\cap \Phi$.
If moreover there is a root $e_i- e_j\in \mathcal{U}^\perp\cap \Phi$ (respectively $e_i+e_j$) for some $j\in\ul{n}$, then $a_i= a_j$ (respectively $a_i=-a_j$) for all $a\in \mathcal{U}$.
The definition of $a_\mathcal{J}$ implies there exists some $0\leq k<d_0$ such that  $j\in \ol{\pi}^k(\mathcal{J})$  and thus $\zeta^k=\pm \epsilon_{\pi^k}(i)$ (in type $\tA$ the equation will be $\zeta^k= \epsilon_{\pi^k}(i)$).
Thus by Corollary~\ref{Znepm1}, $k=0$ and $j\in\mathcal{J}$.
The root $e_i-e_j$ is in $\Phi'$ and as $e_i\not\in \mathcal{U}^\perp$, the root $e_i+e_j\not\in \mathcal{U}^\perp\cap \Phi$.
Hence $\mathcal{U}^\perp\cap \Phi\subseteq \Phi_{\mathcal{I}_{-1}}\sqcup \Phi_{A_{\mathcal{I}\setminus\mathcal{I}_{-1}}}$.
\end{proof}
\end{lm}

\begin{rem}\label{PartLabLevi}
Observe that in types $\tB_n$, $\tC_n$ and $\tD_n$, if $\Phi_{m,\mathcal{I}}$ satisfies  Equation {\rm (\ref{eq:wStabRoot})}, then $\mathcal{I}_{-1}=\ul{n}\setminus\ul{m}$.
In type $\tA_{n-1}$ there is at most one set in $\mathcal{I}\vdash n$ which is $\ol{\pi}$-stable (this coincides with \cite[Section 3.4]{BrSpAM} where $0$ was used as an index instead of $-1$).

Hence given a pair $(\mathcal{I}_{-1},\mathcal{I})$, such that $\mathcal{I}_{-1}\subset \ul{n}$ is $\ol{\pi}$-stable and each $\ol{\pi}$-orbit on $\mathcal{I}\vdash \ul{n}\setminus\mathcal{I}_{-1}$ has length $d_0$, define $\Phi_{\mathcal{I}_{-1},\mathcal{I}}:=\Phi_{\mathcal{I}_{-1}}\sqcup \Phi_{A_\mathcal{I}}$.
Then the root system of a $d$-split Levi subgroup in $(\bG,{\bf n}F)$ is given by such a $\ol{\pi}$-stable parabolic root subsystem $\Phi_{\mathcal{I}_{-1},\mathcal{I}}$. 
\end{rem}

\section{Structures associated to $d$-split Levi subgroups in type $\mathrm{C}_{n}$}
\label{dSplitLevi}

\subsection{Setup}

\subsubsection{Notation}\label{CnSetup}
\indent

In type $\tC_n$ the simply connected simple group $\bG\cong {\rm Sp}_{2n}(\overline{\F}_p)$ with a regular embedding corresponding to $\wt{\bG}\cong{\rm CSp}_{2n}(\overline{\F}_p)$.
Assume throughout that the prime $p\ne 2$.
In particular, there are no graph automorphisms and thus any Steinberg endomorphism $F$ defining an $\mathbb{F}_q$ structure on $\bG$ is a power of $F_p$ and induces the trivial action on the Weyl group $W$.
The set of simple roots was identified with $\{e_1-e_2,e_2-e_3,\dots,e_{n-1}-e_n,2e_n\}$ (see Section~\ref{WeylGp}).
The Chevalley relations, Section~\ref{ChevRel}, imply $h_{2e_{n-1}}(t)=h_{e_{n-1}-e_{n}}(t)h_{2e_n}(t)$ and inductively for $1\leq i\leq n-2$ that $h_{{2e_i}}(t)= h_{e_{i}-e_{i+1}}(t)h_{2e_{i+1}}(t)$.
Hence the maximal torus ${\bf T}$ of ${\bG}$ is given by 
$$\bT=\GenGp{h_{2e_i}(t)\mid 1\leq i\leq n \text{ and } t\in \ol{\mathbb{F}}_p^\times}.$$
For any subset $\mathcal{J}\subseteq \ul{n}$, define $${\bf T}_{\mathcal{J}}=\GenGp{ h_{2e_i}(t)\mid i\in \mathcal{J} \text{ and } t\in \ol{\mathbb{F}}_p^\times}.$$
For $\mathcal{I}\vdash \ul{n}$, it follows that ${\bf T}=\prod_{\mathcal{J}\in \mathcal{I}}{\bf T}_{\mathcal{J}}$.
Note that a maximal torus in ${\wt{\bf G}}$ will not decompose across a partition in a similar way.
By \cite[Table 1.12.6]{GLS3}, the centre $Z(\bG)$ is generated by the involution $z=h_{2e_1}(-1)h_{2e_2}(-1)\cdots h_{2e_n}(-1)$. 

Recall that $\rho$ denoted the homomorphism $V_{\Phi_{{\tC}_n}}\rightarrow  W_{\Phi_{{\tC}_n}}$ and as with the maximal torus $\bT$ the kernel $H_{\Phi_{\tC_n}}=\langle h_{2e_i}(-1)\mid 1\leq i\leq n\rangle$.
For ease of notation $\rho$ will be identified with $f\circ \rho$, for $f$ from Remark~\ref{CanWeylGpHom}, so that $\rho: V_{\Phi_{{\tC}_n}}\rightarrow \sgnSymm_n$.
Similarly, set $\ol{\rho}:V_{\Phi_{{\tC}_n}}\rightarrow \Symm_n$.
For simplicity, denote by $V_{\tC_n}=V_{\Phi_{{\tC}_n}}$ and $W_{\tC_n}=W_{\Phi_{{\tC}_n}}$.

\subsubsection{The choice of Sylow twist}\label{vwSylTwist}
\indent

Fix a positive integer $d$ and take $d_0$ as in Remark~\ref{d0Defn}.
Let $l\leq n$ be maximal such that $d_0\mid l$ and set $a:=l/d_0$.
Note that this is equivalent to $l$ being maximal such that $d\mid 2l$.
Set ${\bf n}_i:={\bf n}_{e_i-e_{i+1}}={\bf n}_{e_i-e_{i+1}}(1)$ as in Section~\ref{ChevRel}.
Then for $v_0:= ({\bf n}_1\cdots {\bf n}_{l-1}{\bf n}_{2e_{l}}(-1))$ the element $v:=v_0^{2l/d}$ yields a Sylow $d$-twist \cite[5.A]{CabSpMcKTypC}.
It follows from Remark~\ref{CanWeylGpHom} that $w_0:=\rho(v_0)=(1,2,\dots,l,-1,-2,\dots, -l)$ and $w':=w_0^a=\prod_{i=1}^a w'_{\mathcal{J}_i^{d_0,a}}$, where $w'_{\mathcal{J}_i^{d_0,a}}$ are as defined in Section~\ref{PermPart}.
Hence for
$$w_{\mathcal{J}_i^{d_0,a}}:= \left\{
     \begin{array}{lr}
      w'_{\mathcal{J}_i^{d_0,a}} & d \text{ even; }\\
       (w'_{\mathcal{J}_i^{d_0,a}})^2 & d \text{ odd, }\\
      \end{array}
   \right.$$
it follows that $w:=\rho(v)=\prod_{i=1}^a w_{\mathcal{J}_i^{d_0,a}}$.
In particular, $w$ is a product of $d$-cycles and the image $\ol w$ in $\Symm_n$ is a product of $l/d_0$ cycles all of length $d_0$.
Moreover, by Section~\ref{PermPart}, the elements $w$ and $w'$ have the same centraliser.

To prove Theorem~\ref{thmA} the following criterion provided in \cite{BrSpAM} will be considered:

\begin{thm}\cite[Theorem 4.1]{BrSpAM} \label{MainThmReq}
Let $d$ be a positive integer and $\wtbL$ a $vF$-stable $d$-split Levi subgroup of $\wt{\bG}$.
Set $\bL=\wt{\bL}\cap \bG$, a $vF$-stable $d$-split Levi subgroup of $\bG$.
Assume for the subgroups $N:=\NNN_{{\bG}}(\bL)^{vF}$, $\wt{N}:=\NNN_{\wt{\bG}}(\bL)^{vF}$ and $\wh{N}:=(\Cent_{{\bG}^{F_0^{em}}E}(v\wh{F}))_\bL$ the following conditions are satisfied:
\begin{asslist}
\item\label{thm41i} There exists some set $\mathcal{T}\subseteq \Irr (L)$, such that 
\begin{enumerate}
\item $\wt{N}_{\xi}=\wt{L}_{\xi}N_{\xi}$ for every $\xi\in\mathcal{T}$,
\item $(\wt{N}\wh{N})_{\Ind_L^N(\xi)}=\wt{N}_{\Ind_L^N(\xi)}\wh{N}_{\Ind_L^N(\xi)}$ for every $\xi\in \mathcal{T}$, and 
\item $\mathcal{T}$ contains some $\wh N$-stable $\wt{L}$-transversal of $\Irr (L)$.
\end{enumerate}

\item \label{thm41ii}There exists an extension map $\Lambda$ with respect to $L\lhd N$ such that 
\begin{enumerate}
\item $\Lambda$ is $\wh{N}$-equivariant.
\item Every character $\xi\in \mathcal{T}$ has an extension $\wh{\xi}\in\Irr(\wh{N}_{\xi})$ with $\Res_{\NNN_{\xi}}^{\wh{N}_{\xi}}(\wh{\xi})=\Lambda(\xi)$ and $v\wh{F}\in \ker(\wh{\xi})$.
\end{enumerate}

\item \label{thm41iii}Let $W_d:=N/L$ and $\wh{W}_d:=\wh{N}/L$.
For $\xi \in \Irr (L)$ and $\wt{\xi}\in \Irr (\wt{L}_\xi\mid \xi)$ define $W_{\wt{\xi}}:=N_{\wt{\xi}}/L$, $W_\xi:=N_{\xi}/L$, $K:=\NNN_{W_d}(W_\xi,W_{\wt{\xi}})$ and $\wh{K}:=\NNN_{\wh{W}_d}(W_\xi,W_{\wt{\xi}})$.
Then for every $\eta_0\in\Irr (W_{\wt{\xi}})$ there exists  some $\eta\in \Irr (W_{\xi}\mid \eta_0)$ which is $\wh{K}_{\eta_0}$-invariant.
\end{asslist}
Then:
\begin{thmlist}
\item For every $\chi\in \Irr (\wt{N})$ there exists some $\chi_0\in \Irr (N\mid \chi)$ such that
\begin{enumerate}
\item $(\wt{N}\wh{N})_{\chi_0}=\wt{N}_{\chi_0}\wh{N}_{\chi_0}$, and
\item $\chi_0$ has an extension $\wt{\chi}_0$ to $\wh{N}_{\chi_0}$ with $v\wh{F}\in \ker (\wt{\chi}_0)$. 
\end{enumerate}
\item Moreover, there exists some $\wh N$-equivariant extension map with respect to $\wt L \lhd \wt N$ that is compatible with $\Irr(\wh N/N)$.
\end{thmlist}
\end{thm}

Note that for condition \ref{thm41iii}, has been simplified as in this case the groups $E$ and $(W_d)_{\xi}/(W_d)_{\wt\xi}\cong {\rm Lin}(\tilde{L}_\eta/\tilde{L}_{\Lambda(\eta)})$ are cyclic groups.

To consider the conditions of the above Theorem, it will be helpful in the later sections if $v$ can be replaced by some other element in the extended Weyl group $V_{{\tC_n}}$.
The following lemma considers when this is possible.

\begin{lm}\label{TwistTwist}
Let $d$ be a positive integer, $vF$ the Sylow $d$-twist constructed above and $\bL$ a $vF$-stable $d$-split Levi subgroup of $({\bG},vF)$.
Assume there exists an ${\bf n}\in V_{{\tC_n}}$ such that ${\bf n}v^{-1}\in\bL$.
Then $\bL$ is an ${\bf n}F$-stable $d$-split Levi subgroup of $({\bG},{\bf n}F)$.
Moreover for $g\in\bL$ with ${\bf n}v^{-1}=g(vF)(g)^{-1}:=g ({}^vF(g))^{-1}$, the map 
$$
\begin{array}{ccc}
\iota: \wt{\bG}^{F_p^{em}}\rtimes E & \rightarrow &  \wt{\bG}^{F_p^{em}}\rtimes E\\
x & \mapsto & x^{g^{-1}}\\
\end{array}
$$
is an isomorphism which maps ${\bG}^{vF}\rtimes E$ onto $ {\bG}^{{\bf n}F}\rtimes E$ with $\iota(v\wh{F})={\bf n}\wh{F}$ and $\iota({\bL}^{vF})={\bL}^{{\bf n}F}$.
In particular, the conclusion of Theorem~\ref{MainThmReq} holds with respect to $vF$ if and only if it holds with respect to ${\bf n}F$. 
\end{lm}
\begin{proof}
As ${\bf n}=lv$ for some $l\in \bL$, it follows that $\bL$ is ${\bf n}F$-stable.
The observation that $(lvF)(z)=(vF)(z)$, for any $z\in Z(\bL)$, implies $\bL$ is a $d$-split Levi subgroup of $({\bG},{\bf n}F)$. 

Note that $(vF)^e=F_p^{em}$ and $(vF)^i({\bf n}v^{-1})={}^{v^i}({\bf n}v^{-1})=v^i{\bf n}v^{-(i+1)}$.
Therefore $g\in{\bG}^{F_p^{em}}$ (and also $F_p(g)$) as $$gF_p^{em}(g)^{-1}=\prod_{i=0}^{e-1}((vF)^i(g)\cdot (vF)^{i+1}(g)^{-1})=\prod_{i=0}^{e-1}(vF)^i({\bf n}v^{-1})={\bf n}^ev^{-e}=1.$$ 
Recall that $E=\langle \wh{F}_p\rangle$, which acts trivially on the extended Weyl group $V$.
Hence, as $\wh{F}_p^{g^{-1}}=gF_p(g)^{-1}\wh{F}_p\in \wt{\bG}^{F_p^{em}}\rtimes E$, the map $\iota$ is an isomorphism.
Moreover, $g\in{\bL}^{F_p^{em}}$ and so $\iota$ restricts to an automorphism of $\wt{\bL}^{F_p^{em}}\rtimes E$. 

As $(vF)^e=({\bf n}F)^e=F_p^{em}$, it follows that $\bG^{vF}\leq \bG^{F_p^{em}}$ and $\bG^{{\bf n}F}\leq \bG^{F_p^{em}}$ 
Moreover $\iota (v\wh{F})={\bf n}\wh{F}$ which implies $\iota (\bG^{vF})=\iota ((\bG^{F_p^{em}})^{vF})=(\bG^{F_p^{em}})^{{\bf n}F}=\bG^{{\bf n}F}$ and similarly $\iota(\wt{\bG}^{vF})=\wt{\bG}^{{\bf n}F}$, $\iota (\bL^{vF})=\bL^{{\bf n}F}$ and $\iota(\wt{\bL}^{vF})=\wt{\bL}^{{\bf n}F}$.

As seen above $\iota(\wh{F}_p)=gF_p(g)^{-1}\wh{F}_p$ and thus to show $\iota(\bG^{vF}\rtimes E)=\bG^{{\bf n}F}\rtimes E$, it suffices to see that $gF_p(g)^{-1}\in \bL^{{\bf n}F}$.
Furthermore, as ${\bf n}=gv (F(g))^{-1}$, it follows that $({\bf n}F)(gF_p(g)^{-1})={}^{gv}(F(F_p(g))^{-1}F(g))$ and thus showing $gF_p(g)^{-1}\in \bL^{{\bf n}F}$ is equivalent to $g^{-1}F_p(g)\in \bL^{{\bf v}F}$.
Since ${\bf n}v^{-1}$ is fixed by $E$, then $F_p(g) ((vF)(F_p(g)))^{-1}=F_p({\bf n}v^{-1})={\bf n}v^{-1}$ and hence $g^{-1}F_p(g)\in \bL^{{\bf v}F}$.

Set ${\bf N}:=\NNN_{\bG}(\bL)$, $\wt{\bf N}:=\NNN_{\wt{\bG}}(\bL)$ and $\wh{\bf N}:=\NNN_{\bG\rtimes E}(\bL)$.
Then by the above properties of $\iota$, it follows that $\iota({\bf N}^{vF})={\bf N}^{{\bf n}F}$, $\iota(\wt{\bf N}^{vF})=\wt{\bf N}^{{\bf n}F}$ and $\iota(\wh{\bf N}^{vF})=\wh{\bf N}^{{\bf n}F}$.
Hence applying the isomorphism $\iota$ proves the equivalence of the following two statements:
 \begin{enumerate}
\item There exists some $\wh{\bf N}^{vF}$-equivariant extension map with respect to $\wt{\bf L}^{vF} \lhd \wt{\bf N}^{vF}$ that is compatible with $\Irr(\wh{\bf N}^{vF}/{\bf N}^{vF})$. Moreover, for every $\chi\in \Irr (\wt{\bf N}^{vF})$ there exists some $\chi_0\in \Irr ({\bf N}^{vF}\mid \chi)$ such that
\begin{itemize}
\item $(\wt{\bf N}^{vF}\wh{\bf N}^{vF})_{\chi_0}=\wt{\bf N}^{vF}_{\chi_0}\wh{\bf N}^{vF}_{\chi_0}$, and
\item $\chi_0$ has an extension $\wt{\chi}_0$ to $\wh{\bf N}^{vF}_{\chi_0}$ with $v\wh{F}\in \ker (\wt{\chi}_0)$. 
\end{itemize}
\item There exists some $\wh{\bf N}^{{\bf n}F}$-equivariant extension map with respect to $\wt{\bf L}^{{\bf n}F} \lhd \wt{\bf N}^{{\bf n}F}$ that is compatible with $\Irr(\wh{\bf N}^{{\bf n}F}/{\bf N}^{{\bf n}F})$. Moreover, for every $\chi\in \Irr (\wt{\bf N}^{{\bf n}F})$ there exists some $\chi_0\in \Irr ({\bf N}^{{\bf n}F}\mid \chi)$ such that
\begin{itemize}
\item $(\wt{\bf N}^{{\bf n}F}\wh{\bf N}^{{\bf n}F})_{\chi_0}=\wt{\bf N}^{{\bf n}F}_{\chi_0}\wh{\bf N}^{{\bf n}F}_{\chi_0}$, and
\item $\chi_0$ has an extension $\wt{\chi}_0$ to $\wh{\bf N}^{{\bf n}F}_{\chi_0}$ with ${\bf n}\wh{F}\in \ker (\wt{\chi}_0)$. 
\end{itemize}
\end{enumerate}
\end{proof}

\subsection{Some constructions associated to $w$}\label{NotPart}
\indent

Assume $(\mathcal{I}_{-1},\mathcal{I})$ satisfies Remark~\ref{PartLabLevi} with respect to $w$ as defined Section~\ref{vwSylTwist}.
That is $\mathcal{I}_{-1}\subset \ul{n}$ and $\mathcal{I}\vdash \ul{n}\setminus\mathcal{I}_{-1}$ are $\ol{w}$-stable with each $\ol{w}$-orbit on $\mathcal{I}$ having length $d_0$ and the sign function $\epsilon_w$, defined in Remark~\ref{CanWeylGpHom}, is constant on each $\mathcal{J}\in\mathcal{I}$.

\subsubsection{Some notation}\label{wPartNot}
\indent

Let $\mathcal{J}\in \mathcal{I}$.
Due to the assumptions on $\mathcal{I}$, the minimal integer $k>0$ such that $\epsilon_{w^k}(i)=-1$ must be the same for all $i\in \mathcal{J}$.
Moreover, the description $w=\prod_{i=1}^a w_{\mathcal{J}_i^{d_0,a}}$, implies each $\ol{w}$-orbit on $\mathcal{I}$ must contain a set $\mathcal{J}\subseteq \ul{a}$.
Let $\mathcal{O}\subset \mathcal{I}$ be a $\ol{w}$-orbit and set $\mathcal{J}_{\mathcal{O}}:=\ul{\mathcal{O}}\cap \ul{a}\in\mathcal{O}$.
Then, using the notation from Section~\ref{NotSets},
$\ul{\mathcal{O}}=\bigsqcup_{i\in\mathcal{J}_{\mathcal{O}}} \mathcal{J}_i^{d_0,a}$.

For each $1\leq s\leq n$ the set $\mathcal{I}_s:=\{ \mathcal{J}\in\mathcal{I}\mid |\mathcal{J}|=s\}$ is $\ol{w}$-stable.
In particular, $\mathcal{I}_s$ consists of $t_s:=\frac{|\mathcal{I}_s|}{d_0}$ $\ol{w}$-orbits each of which is determined by a set $\mathcal{J}_{\mathcal{O}}\subset\ul{a}$ of size $s$.
Let $\mathcal{O}^s_1,\dots,\mathcal{O}^s_{t_s}$ denote the $\ol{w}$-orbits in $\mathcal{I}_s$ and for $1\leq s\leq t_s$ set
$$Q_r^s:=\{Q^s_{r,1},\dots, Q^s_{r,s}\}\subset \mathcal{J}^{d_0,a}=\{\mathcal{J}^{d_0,a}_{i}\mid 1\leq i\leq a\},$$ 
such that $\ul{\mathcal{O}}_r^s=\ul{Q}_r^s$.
In particular, $\mathcal{J}_{\mathcal{O}_r^s}=\{Q^s_{r,i}(1)\mid 1\leq i\leq s\}$.
Furthermore, up to a choice of labelling, assume $r\leq Q_{r,1}(1)$ and for $1\leq i\leq s-1$ that $Q^s_{r,i}(1)<Q^s_{r,i+1}(1)$.
The restriction of $w'$ to $\pm\ul{\mathcal{O}}_r^s$ is given by
$$w'|_{\pm\ul{\mathcal{O}}_r^s}=w'_{Q_r^s}:=\prod_{i=1}^s w'_{Q_{r,i}^s}\in \sgnSymm(\ul{\mathcal{O}}_r^s).$$
Moreover, the restriction of $w'$ to $\pm \ul{\mathcal{I}}_s$ is given by $w'|_{\pm\ul{\mathcal{I}}_s}=\prod_{r=1}^{t_s} w'_{Q_{r}^s}\in \sgnSymm(\ul{\mathcal{I}}_s)$.

Finally $w'=w'|_{\pm{\mathcal{I}_{-1}}} \times w'|_{\pm\ul{\mathcal{I}}}$, where $w'|_{\pm{\mathcal{I}_{-1}}}\in \sgnSymm(\mathcal{I}_{-1})$ and $w'|_{\pm\ul{\mathcal{I}}}=\prod_{s=1}^nw'_{\pm\ul{\mathcal{I}}_s}\in \sgnSymm(\ul {\mathcal{I}})$ denote the restriction of $w'$ to $\pm\mathcal{I}_{-1}$ and respectively $\pm\ul{\mathcal{I}}$. 
Note that an analogous decomposition can be given for $w$.

\subsubsection{The action of $w$ on $\mathcal{I}$}\label{wactsI}
\indent 

The assumptions on $(\mathcal{I}_{-1},\mathcal{I})$ under the action of $w$ imply that $w'$ induces a permutation in $\sgnSymm (\mathcal{I})$ (respectively $\sgnSymm (\mathcal{I}_s)$) which will be denoted by $w'\mid_{\pm\mathcal{I}}$ (respectively $w'\mid_{\pm\mathcal{I}_s}$).
That is for $\mathcal{O}$ a $\ol{w}$-orbit in $\mathcal{I}$ the induced action of $w'$ on $\sgnSymm(\mathcal{O})$ is given by the cycle 
 $$w'|_{\pm \mathcal{O}}=(\mathcal{J}_\mathcal{O},\dots, w'^{d_0-1}(\mathcal{J}_\mathcal{O}),-\mathcal{J}_\mathcal{O},\dots, - w'^{d_0-1}(\mathcal{J}_\mathcal{O}))\in \sgnSymm(\mathcal{O}),$$
where $\mathcal{J}_\mathcal{O}=\ul{\mathcal{O}}\cap \ul{a}$.
Furthermore,  
$$w'|_{\pm\mathcal{I}}=\prod_{s=1}^nw'|_{\pm \mathcal{I}_s}=\prod_{s=1}^n\left(\prod_{r=1}^{t_s}w'|_{\pm\mathcal{O}_r^s} \right).$$
Similarly $w|_{\pm\mathcal{I}}$, $w|_{\pm\mathcal{I}_s}$ and $w|_{\pm\mathcal{O}}$ are defined.

\subsubsection{Centralisers of $w|_{\pm\mathcal{I}}$ and $w|_{\pm\ul{\mathcal{I}}}$}\label{CentwIinCentwulI}
\indent

In Proposition~\ref{RelWeylGpPart} it will be proven that $\prod_{s=1}\Cent_{\sgnSymm(\mathcal{I}_s)}(w|_{\pm\mathcal{I}_s})$ describes the structure of the relative Weyl group arising from the $d$-split Levi subgroup associated to the pair $(\mathcal{I}_{-1},\mathcal{I})$.
To construct the normaliser of this $d$-split Levi subgroup in Proposition~\ref{SuppLinN}, it will be helpful to have an explicit description of $\prod_{s=1}\Cent_{\sgnSymm(\mathcal{I}_s)}(w| _{\pm \mathcal{I}_s})$ as a subgroup of $  \Cent_{\sgnSymm (\ul{\mathcal{I}})}(w|_{\pm\ul{\mathcal{I}}})$, which can then be mimicked in the extended Weyl group.
For $\mathcal{O}_r^s$ a $\ol{w}$-orbit in $\mathcal{I}_s$, the cycle $w'|_{\pm\mathcal{O}_r^s}$ corresponds to the product of $s$ cycles given by $w|_{\pm\ul{\mathcal{O}}_r^s}=w'_{Q_r^s}$.
Hence the aim is to provide an explicit injection which takes every element in $\Cent_{\sgnSymm(\mathcal{I}_s)}(w|_{\pm\mathcal{I}_s})$ and converts them into a product of $s$ elements inside $\Cent_{\sgnSymm(\ul{\mathcal{I}}_s)}(w|_{\pm\ul{\mathcal{I}}_s})$.

Take $Q_1^s,\dots, Q^s_{t_s}$ representing the $\ol{w}$-orbits in $\mathcal{I}_s$ as explained in Section~\ref{wPartNot}, with each $Q^s_i=\{Q_{i,1}^s,\dots, Q_{i,s}^s\}\subset\mathcal{J}^{d_0,a}$.
By Section~\ref{PermPart} the group $\Cent_{\sgnSymm(\ul{\mathcal{I}}_s)}(w|_{\pm\ul{\mathcal{I}}_s})=\Cent_{\sgnSymm(\ul{\mathcal{I}}_s)}(w'|_{\pm\ul{\mathcal{I}}_s})$ is isomorphic to $ C_{2d_0}\wr \Symm_{s\cdot t_s}$ and can be explicitly written as
$$\Cent_{\sgnSymm(\ul{\mathcal{I}}_s)}(w|_{\pm\ul{\mathcal{I}}_s})=\left( \prod_{i=1}^{t_s}\prod_{j=1}^{s} \GenGp{w'_{Q_{i,j}^s}}\right) \rtimes \langle \tau_{Q_{i,j}^s,Q_{i,j+1}},\tau _{Q_{i,s}^s,Q_{i+1,1}} \mid 1\leq j\leq s-1 \text{ and } 1\leq i\leq t_s\rangle.$$

Recall for $1\leq i\leq t_s$ each orbit $\mathcal{O}_i^s$ in $\mathcal{I}_s$ was determined by $\mathcal{J}_{\mathcal{O}_{i}^s}:=\ul{a}\cap \ul{\mathcal{O}}_i^s$.
Thus the bijection 
$$\begin{array}{ccl}
\mathcal{I}_s& \rightarrow & \ul{t_s d_0}\\
(w')^j(\mathcal{J}_{\mathcal{O}_i^s})& \mapsto & i+j\cdot t_s
\end{array}$$ 
yields an isomorphism $\sgnSymm(\mathcal{I}_s)\cong \sgnSymm_{d_0t_s}$ under which $w'|_{\pm \mathcal{I}_s}$ is mapped to $w'_{\mathcal{J}^{d_0,t_s}}$ (as defined in Section~\ref{PermPart}). 
Thus applying Section~\ref{PermPart} implies
$$\Cent_{\sgnSymm(\mathcal{I}_s)}(w|_{\pm\mathcal{I}_s})\cong \Cent_{\sgnSymm_{d_0t_s}}\left(w_{\mathcal{J}^{d_0,t_s}}\right)=\left( \prod_{i=1}^{t_s}\langle w'_{\mathcal{J}^{d_0,t_s}_i}\rangle \right)\rtimes \langle \tau_{\mathcal{J}^{d_0,t_s}_i}\mid 1\leq i\leq t_s-1\rangle.$$

The centraliser of $w|_{\pm\ul{\mathcal{I}}_s}$ above was described via the subsets $Q_i^s\subset\mathcal{J}^{d_0,a}$.
Thus to convert from elements defined with respect to $\mathcal{J}^{d_0,t_s}$ to $\mathcal{J}^{d_0,a}$, observe that there is an isomorphism $\kappa_s^\circ$ which is induced by conjugation with the element $\prod_{j=1}^{d_0-1}(j\cdot t_s,j\cdot a)(-j\cdot t_s,-j\cdot a)\in\sgnSymm_{d_0a}$ 
 $$\begin{array}{cccc}
\kappa_s^\circ:& \Cent_{\sgnSymm_{d_0t_s}}\left(w_{\mathcal{J}^{d_0,t_s}}\right) & \rightarrow & \Cent_{\sgnSymm(\mathcal{I}_s)}(w|_{\pm\mathcal{I}_s})^\circ:= \left( \prod_{i=1}^{t_s}\langle w'_{\mathcal{J}^{d_0,a}_i}\rangle \right)\rtimes \langle \tau_{\mathcal{J}^{d_0,a}_i}\mid 1\leq i\leq t_s-1\rangle \\
 &   w'_{\mathcal{J}_i^{d_0,t_s}} & \mapsto & w'_{\mathcal{J}_i^{d_0,a}}\\
 & \tau_{\mathcal{J}_{i}^{d_0,t_s}} & \mapsto &  \tau_{\mathcal{J}_{i}^{d_0,a}}
\end{array}.$$
For $1\leq j\leq s$ set $\iota_{j}=\prod_{i=1}^{t_s}\tau_{\mathcal{J}_i^{d_0,a},Q^s_{i,j}}$, where $\tau_{\mathcal{J}_i^{d_0,a},Q^s_{i,j}}\in\sgnSymm_{d_0a}$ is the involution as defined in Section~\ref{PermPart}.
Then by construction $\ol{\tau}_{\mathcal{J}_i^{d_0,a},Q^s_{i,j}}(\mathcal{J}_i^{d_0,a})=Q^s_{i,j}$ for $1\leq i\leq t_s$.
Hence $\sgnSymm(\sqcup_{i=1}^{t_s}\mathcal{J}_i^{d_0,a})^{\iota_{j}}=\sgnSymm(\sqcup_{i=1}^{t_s}Q_{i,j}^s)$ with $ (w'_{\mathcal{J}^{d_0,a}_i})^{\iota_j}=w'_{Q^s_{i,j}}$ and $(\tau_{\mathcal{J}^{d_0,a}_i})^{\iota_j}=\tau_{Q_{i,j}^s,Q_{i+1,j}^s}$ for any $1\leq i\leq t_s$.
Therefore the map 
$$
\begin{array}{cccc}
\kappa_s: & \Cent_{\sgnSymm(\mathcal{I}_s)}(w|_{\pm\mathcal{I}_s})^\circ & \rightarrow & \Cent_{\sgnSymm(\ul{\mathcal{I}}_s)}(w|_{\pm\ul{\mathcal{I}}_s})\\
 & & & \\
&  x& \mapsto & \prod_{j=1}^s x^{\iota_j}\\
\end{array}
$$
is an injective group homomorphism.
Define $$W_d^{\mathcal{I}_s}:=\kappa_s\circ\kappa_s^\circ (\Cent_{\sgnSymm_{d_0t_s}}\left(w'_{\mathcal{J}^{d_0,t_s}}\right))\leq \Cent_{\sgnSymm(\ul{\mathcal{I}_s})}(w|_{\pm\ul{\mathcal{I}}_s})$$ which stabilises $\mathcal{I}_s$ and by construction  $W_d^{\mathcal{I}_s}$ induces $\Cent_{\sgnSymm(\mathcal{I}_s)}(w|_{\pm\mathcal{I}_s})$.
Furthermore combining these maps over all possible $s$ provides an injection
$$
\kappa\circ\kappa^\circ: \prod_{s=1}^n \Cent_{\sgnSymm_{d_0t_s}}(w_{\mathcal{J}^{d_0,t_s}}) \rightarrow \prod_{s=1}^n \Cent_{\sgnSymm(\ul{\mathcal{I}}_s)}(w_{\pm\ul{\mathcal{I}}_s})\leq \Cent_{\sgnSymm(\ul{\mathcal{I}})}(w|_{\pm\ul{\mathcal{I}}})\\
$$
whose image induces $\prod_{s=1}^n\Cent_{\sgnSymm(\mathcal{I}_s)}(w|_{\pm\mathcal{I}_s})$.
Set $W_d^\mathcal{I}=\prod_{s=1}^n W_d^{\mathcal{I}_s}$.

\subsection{Structures associated to $v$}
\indent

Recall from Sections~\ref{WeylGp} and ~\ref{ExtWeylGp} that for $\mathcal{J}\subset \ul{n}$ there is root subsystem given by $\Phi_\mathcal{J}=\Phi_{\tC_n}\cap \GenGp{\pm e_i\mid i\in\mathcal{J}}$ and a subgroup of the extended Weyl group given by $V_{\Phi_\mathcal{J}}=\GenGp{{\bf n}_\alpha\mid \alpha\in \Phi_\mathcal{J}}$ (see  ).
For ease of notation, in the following set $V_\mathcal{J}:=V_{\Phi_\mathcal{J}}$.

\subsubsection{The regular case}\label{SylTorV}
\indent

In \cite[Theorem 5.3]{CabSpMcKTypC} minimal $d$-split Levi subgroups, those where $\mathcal{I}=\mathcal{I}_1$, were studied.
In particular some explicit notation and results were given for the regular case ($\mathcal{I}_{-1}=\emptyset$), which will now be recalled.

By definition, the element $v$ lies in $V_{\ul{l}}$, for $l$ as defined in Section~\ref{vwSylTwist}.
Set $V_d:=\Cent_{V_{\ul{l}}}(v)$.
Then $\rho(V_d)=\Cent_{\sgnSymm_{l}}(\rho(v))$ with kernel 
$$H_d:={\rm ker}(\rho|_{V_d})=\langle h_{1}\rangle\times \dots\times \langle h_{a}\rangle\lhd V_d,$$ 
where $h_{i}:=\prod_{j\in\mathcal{J}_i^{d_0,a}}h_{2e_j}(-1)$ and maximal extendibility holds with respect to $H_d\lhd V_d$.

There is an element $c_{1}\in V_{\mathcal{J}_1^{d_0,a}}\cap V_d$ such that  $\rho( c_{1})=w'_{\mathcal{J}_1^{d_0,a}}$.
In addition, for $1\leq k<a$ the elements 
$p_{k}:=\prod_{i=0}^{d_0-1} ({\bf n}_{e_k-e_{k+1}})^{v^i}\in V_{\mathcal{J}_k^{d_0,a}\sqcup \mathcal{J}_{k+1}^{d_0,a}}\cap V_d$ satisfy the braid relations with $(p_{k})^2=h_{k}h_{k+1}$ and the permutation $\rho(p_{k})=\tau_{\mathcal{J}_{k}^{d_0,a}}$ induces the permutation $(\mathcal{J}_k^{d_0,a},\mathcal{J}_{k+1}^{d_0,a})(-\mathcal{J}_k^{d_0,a},-\mathcal{J}_{k+1}^{d_0,a})\in\sgnSymm(\mathcal{J}^{d_0,a})$.
Moreover
$$V_d=H_d\langle  c_{1}, p_{k}\mid 1\leq k\leq a-1\rangle.$$
Finally, for $K_d:=\langle H_d,c_{k}\mid 1\leq k\leq a\rangle$ with $c_{k}:=(c_{1})^{p_{1}\dots p_{k-1}}\in V_{\mathcal{J}_k^{d_0,a}}$, the map $\rho$ induces an isomorphism $V_d/K_d \cong \mathfrak{S}_{a}$.

\begin{lm}\label{CentralCVv}
The element $v':=\prod_{i=1}^{a}c_{i}$ is a central element of $V_d$ and $\rho(v')=w'$.
\end{lm}
\begin{proof}
It can be seen that $K_d\GenGp{p_{k}\mid i\ne k,k+1}\leq \Cent_{V_d}(c_{i})$ as $[V_{\mathcal{J}_i^{d_0,a}},V_{\mathcal{J}_j^{d_0,a}}]=1$ for $1\leq i\ne j\leq a$ and $[h_{i},c_{i}]=1$.
The proof then follows by noting that 
\begin{align*}(c_{k}) ^{p_{k}}=c_{k+1} \text{ and } (c_{k+1})^{p_{k}}=c_{k}. &  \qedhere
\end{align*}
\end{proof}

For later applications, it will help if some additional elements are introduced.
For $k<k'$, define the elements $$p_{k,k'}:=(p_{k})^{p_{k+1}\dots p_{k'-1}}\in V_{\mathcal{J}_k^{d_0,a}\sqcup \mathcal{J}_{k'}^{d_0,a}}\cap V_d.$$
By construction conjugation with  $p_{k,k'}$ switches $V_{\mathcal{J}_k^{d_0,a}}$ and $V_{\mathcal{J}_{k'}^{d_0,a}}$, while centralising all other $V_{\mathcal{J}_j^{d_0,a}}$ whenever $j\ne k,k'$.

\subsubsection{Suitable subgroup of $\Cent_{V_{{\tC_n}}}(v)$}\label{SupplementLN}
\indent

The aim in this section is to construct a suitable subgroup of $V_d\leq C_{V_{{\tC_n}}}(v)$ which projects under $\rho$ on the embedding of $\prod_{s=1}^n\Cent_{\sgnSymm(\mathcal{I}_s)}(w|_{\pm\mathcal{I}_s})$ as a subgroup of $\Cent_{\sgnSymm(\ul{\mathcal{I}})}(w|_{\pm\ul{\mathcal{I}}})$ constructed in Section~\ref{CentwIinCentwulI}. 

\begin{lm}\label{ExtWeylGpExtMap}
There exists a subgroup $V_d^\mathcal{I}\leq \Cent_V(v)$ such that $\rho(V_d^\mathcal{I})=W_d^\mathcal{I}$ induces the group $\prod_{s=1}^n\Cent_{\sgnSymm(\mathcal{I}_s)}(w_{\mathcal{I}_s})$. Moreover, for $H_d^\mathcal{I}:={\rm ker}(\rho\mid_{V_d^\mathcal{I}})$ maximal extendibility holds with respect to $H_d^\mathcal{I}\lhd V_d^\mathcal{I}$.
\end{lm}

\begin{proof}
For $1\leq s\leq n$, in  Section~\ref{CentwIinCentwulI}, a subgroup $\Cent_{\sgnSymm(\mathcal{I}_s)}(w_{\mathcal{I}_s})^\circ\leq \Cent_{\sgnSymm(\ul{\mathcal{I}}_s)}(w_{\ul{\mathcal{I}}_s})$ was constructed.
In an analogous fashion define
$$
\begin{array}{ccl}
(V_d^{\mathcal{I}_s})^{\circ} &:= & \langle h_{i}\mid 1\leq i\leq t_s\rangle\langle c_{1},p_{k}\mid 1\leq k\leq t_s-1\rangle=V_d\cap V_{\mathcal{J}_1^{d_0,a}\sqcup \dots\sqcup \mathcal{J}_{t_s}^{d_0,a}}.
\end{array}$$
By construction $\rho((V_d^{\mathcal{I}_s})^{\circ})=\Cent_{\sgnSymm(\mathcal{I}_s)}(w_{\mathcal{I}_s})^\circ$ with $$(H_d^{\mathcal{I}_s})^{\circ}:={\rm ker}(\rho\mid_{(V_d^{\mathcal{I}_s})^{\circ}})=\langle h_{i}\mid 1\leq i\leq t_s\rangle.$$
Similarly set $(K_d^{\mathcal{I}_s})^{\circ}:= \langle h_{i}, c_{i} \mid 1\leq i\leq t_s\rangle$, so that $(V_d^{\mathcal{I}_s})^{\circ} /(K_d^{\mathcal{I}_s})^{\circ}\cong \mathfrak{S}_{t_s}$.

Take $Q_1^s,\dots, Q^s_{t_s}$ representing the $\ol{w}$-orbits in $\mathcal{I}_s$ as explained in Section~\ref{wPartNot}, where $Q^s_i=\{Q_{i,1}^s,\dots, Q_{i,s}^s\}\subset\mathcal{J}^{d_0,a}$.
By assumption, for $1\leq i\leq t_s$ and $1\leq j\leq s$, each $ Q^s_{i,j}(1)<Q^s_{i,j+1}(1)$ and if $Q^s_{i,1}=\mathcal{J}_{k}^{d_0,a}$ then $k\geq i$.
Hence for $1\leq i\leq t_s$ and  $1\leq j\leq s$, if $Q_{i,j}^s=\mathcal{J}_{k}^{d_0,a}$ then the elements $p_{i,k}$ from Section~\ref{SylTorV} are defined and their image under $\rho$ in $\sgnSymm_n$ is the involution $\tau_{\mathcal{J}_i^{d_0,a}, Q^s_{i,j}}$ as defined in Section~\ref{PermPart}.
Fix $1\leq k_{i,j}\leq a$ for $1\leq i\leq t_s$ and  $1\leq j\leq s$ such that $Q_{i,j}^s=\mathcal{J}_{k_{i,j}}^{d_0,a}$ and define $\eta_j:=\prod_{i=1}^{t_s}p_{i,k_{i,j}}$.
Then $\rho(\eta_j)=\iota_j$ for $\iota_j$ as defined in Section~\ref{CentwIinCentwulI}.
It follows for $x\in(V_d^{\mathcal{I}_s})^{\circ}$ and $1\leq j\leq s$ that $x^{\eta_j}\in V_d\cap V_ {Q_{1,j}^s\cup\dots\cup Q_{t_s,j}^s}$.
As for $1\leq j\ne j'\leq s$ the sets $\sqcup_{i=1}^{t_s}Q_{i,j}^s$ and $\sqcup_{i=1}^{t_s}Q_{i,j'}^s$ are disjoint, the map
$$
\begin{array}{cccc}
\kappa_s: & (V_d^{\mathcal{I}_s})^{\circ} & \rightarrow & V_d \cap V_{\ul{\mathcal{I}}_s} \\
 & & & \\
 & x& \mapsto & \prod_{j=1}^s x^{\eta_j}\\
\end{array}
$$
is an injective homomorphism by Corollary~\ref{IDisjVDirProd}.
Set $V_d^{\mathcal{I}_s}:=\kappa_s((V_d^{\mathcal{I}_s})^{\circ})$.
By construction $\rho(V_d^{\mathcal{I}_s})=W_d^{\mathcal{I}_s}$
and, as $h_{Q_i^s}:=\prod_{j=1}^{s} h_{Q_{i,j}^s}=\kappa_s(h_{i})$, it follows that $$H_d^{\mathcal{I}_s}:={\rm ker}(\rho|_{V_d^{\mathcal{I}_s}})=\langle h_{Q_i^s}\mid 1\leq i \leq t_s\rangle =\kappa_s((H_d^{\mathcal{I}_s})^{\circ}).$$ 

These maps can be combined over all possible $s$ to give an injection 
$$
\kappa: \prod_{s=1}^n (V_d^{\mathcal{I}_s})^{\circ} \rightarrow \prod_{s=1}^n V_d \cap V_{\ul{\mathcal{I}}_s} \leq C_V(v).
$$
Set $V_d^\mathcal{I}:=\kappa\left( \prod_{s=1}^n (V_d^{\mathcal{I}_s})^{\circ}\right)=\prod_{s=1}^n V_d^{\mathcal{I}_s}$ so that by construction $\rho(V_d^\mathcal{I})=W_d^\mathcal{I}$.
Furthermore $H_d^\mathcal{I}:={\rm ker}(\rho|_{V_d^{\mathcal{I}}})=\prod_{s=1}^nH_d^{\mathcal{I}_s}$.
Thus to prove maximal extendibility with respect to $H_d^\mathcal{I}\lhd V_d^\mathcal{I}$ it suffices to prove it with respect to $(H_d^{\mathcal{I}_s})^{\circ}\lhd (V_d^{\mathcal{I}_s})^{\circ}$.

The argument follows that presented in \cite[Theorem 5.3]{CabSpMcKTypC}.
Using the decomposition $(H_d^{\mathcal{I}_s})^{\circ}=\GenGp{h_{1}}\times\dots\times \GenGp{h_{t_s}}$, up to $(V_d^{\mathcal{I}_s})^{\circ}$-conjugation, any character $\lambda\in \Irr((H_d^{\mathcal{I}_s})^{\circ})$ can be written as $\lambda_1\times\dots\times \lambda_{t_s}$ where $o(\lambda_i)=1$ for all $i\leq k_0$ and $o(\lambda_i)=2$ for all $i>k_0$ for some $0\leq k_0\leq t_s$. 
As in the proof of Lemma~\ref{CentralCVv}, $(K_d^{\mathcal{I}_s})^{\circ}$ is abelian and thus $((V_d^{\mathcal{I}_s})^{\circ})_\lambda=(K_d^{\mathcal{I}_s})^{\circ}S$ for $S:=\GenGp{p_{k}\mid 1\leq k<t_s, \text{ } k\ne k_0}$.
Moreover, $$S\cap (K_d^{\mathcal{I}_s})^{\circ}=S\cap (H_d^{\mathcal{I}_s})^{\circ}=\GenGp{ p_{k}^2\mid 1\leq k\leq t_s, \text{ }k,\ne k_0}\leq {\rm ker}(\lambda),$$ that is, ${\rm Res}^{ (H_d^{\mathcal{I}_s})^{\circ}}_{S\cap (H_d^{\mathcal{I}_s})^{\circ}}(\lambda)$ is the trivial character.

Fix an extension map $\Lambda_1$ with respect to $\GenGp{ h_{1}}\lhd \GenGp{h_{1},  c_{1}}$  and denote by $\Lambda_k$ the extension map with respect to  $\GenGp{ h_{k}}\lhd \GenGp{h_{k},  c_{k}}$ given by conjugating $\Lambda_1$ with $p_{1}\cdots p_{k-1}$. 
Then $\wt{\lambda}:=\Lambda_1(\lambda_1)\times \cdots\times \Lambda_{t_s}(\lambda_{t_s})\in \Irr((K_d^{\mathcal{I}_s})^{\circ})$ is $p_{k}$-invariant for $1\leq k<t_s$ with $k\ne k_0$.
Hence $((V_d^{\mathcal{I}_s})^{\circ})_\lambda= ((V_d^{\mathcal{I}_s})^{\circ})_{\wt{\lambda}}$.
As ${\rm Res}^{ (K_d^{\mathcal{I}_s})^{\circ}}_{S\cap (H_d^{\mathcal{I}_s})^{\circ}}(\lambda)$ is the trivial character \cite[Lemma 4.1]{SpSylowTori2} implies $\wt{\lambda}$, and hence also $\lambda$, extends to $ ((V_d^{\mathcal{I}_s})^{\circ})_{{\lambda}}$.

\end{proof}

\subsection{Structure associated to $d$-split Levi subgroups of $\bG$}\label{StrdSpLev}
\indent

Take $v$ as defined in Section~\ref{vwSylTwist} so that $vF$ forms a Sylow $d$-twist and let $\bL$ be a $d$-split Levi subgroup with respect to $(\bG,vF)$.
As explained in Section~\ref{RootSysdLevi}, up to conjugation, it can be assumed that $\bT\leq \bL$ and the root system $\Phi_{\bL}$ coincides with $\Phi_{\mathcal{I}_{-1},\mathcal{I}}$ where $(\mathcal{I}_{-1},\mathcal{I})$ satisfies Remark~\ref{PartLabLevi} with respect to $w=\rho(v)\in\sgnSymm_n$. 
That is $\mathcal{I}_{-1}\subset \ul{n}$ and $\mathcal{I}\vdash \ul{n}\setminus\mathcal{I}_{-1}$ are $\ol{w}$-stable with each $\ol{w}$-orbit on $\mathcal{I}$ having length $d_0$ and the sign function $\epsilon_{w}$, defined in Remark~\ref{CanWeylGpHom}, is constant on each $\mathcal{J}\in\mathcal{I}$.
For such a pair $(\mathcal{I}_{-1},\mathcal{I})$, the corresponding $d$-split Levi subgroup in $\bG$ will be denoted 
\[
\bL_{\mathcal{I}_{-1},\mathcal{I}}:=\GenGp{\bT, {\bf X}_{\alpha}\mid \alpha\in\Phi_{\mathcal{I}_{-1},\mathcal{I}}}.
\]

As in Section~\ref{CnSetup}, $\bT=\bT_{\mathcal{I}_{-1}}\times \prod_{\mathcal{J}\in\mathcal{I}}\bT_\mathcal{J}$.
For $\mathcal{I}_{-1}$ and any $\mathcal{J}\in \mathcal{I}$ define 
$$
\bL_{\mathcal{I}_{-1}}:=\GenGp{ \bT_{\mathcal{I}_{-1}}, X_{\alpha}\mid \alpha\in \Phi_{\mathcal{I}_{-1}}}
\text{ and }
\bL_{\mathcal{J}}:=\GenGp{ \bT_{\mathcal{J}}, X_{\alpha}\mid \alpha\in \Phi_{A_\mathcal{J}}}.
$$
The definition of $\Phi_{\mathcal{I}_{-1},\mathcal{I}}=\Phi_{\mathcal{I}_{-1}}\sqcup \Phi_{A_{\mathcal{I}}}$ from Remark~\ref{PartLabLevi} implies that
\[
\bL_{\mathcal{I}_{-1},\mathcal{I}}=\bL_{\mathcal{I}_{-1}}\times \prod_{\mathcal{J}\in\mathcal{I}} \bL_{\mathcal{J}}\cong \Sp_{2|\mathcal{I}_{-1}|}(\ol{\F}_p)\times \prod_{\mathcal{J}\in\mathcal{I}} \GL_{|\mathcal{J}|}(\ol{\F}_p)\cong \Sp_{2|\mathcal{I}_{-1}|}(\ol{\F}_p)\times \prod_{s=1}^n \GL_{s}(\ol{\F}_p)^{|\mathcal{I}_s|}.
\]
Moreover the associated Weyl group as a subgroup of $W_\bG\cong \sgnSymm_n$ is given by 
$$
W_{\bL_{\mathcal{I}_{-1},\mathcal{I}}}= W_{C_{\mathcal{I}_{-1}}}\times \prod\limits_{\mathcal{J}\in\mathcal{I}} W_{A_\mathcal{J}}=\sgnSymm (\mathcal{I}_{-1})\times \prod\limits_{\mathcal{J}\in\mathcal{I}} \Symm(\mathcal{J}).
$$
\subsubsection{Fixed points of $\bL_{\mathcal{I}_{-1},\mathcal{I}}$}\label{ChoiceOfn}
\indent

Let ${\bf n}\in V$ such that $ \rho({\bf n})\cdot w^{-1}\in W_{{\bL}_{\mathcal{I}_{-1},\mathcal{I}}}=\NNN_{{\bL}_{\mathcal{I}_{-1},\mathcal{I}}}({\bf T})/{\bf T}$.
Then ${\bf n}v^{-1}\in {\bL}_{\mathcal{I}_{-1},\mathcal{I}}$ and thus $\iota({\bL}_{\mathcal{I}_{-1},\mathcal{I}}^{vF})={\bL}_{\mathcal{I}_{-1},\mathcal{I}}^{{\bf n}F}$ for the isomorphism $\iota$ from  Lemma~\ref{TwistTwist}.
Assume, furthermore, that $ \rho({\bf n})\cdot w^{-1}\in  \sgnSymm(\mathcal{I}_{-1})\leq W_{{\bL}_{\mathcal{I}_{-1},\mathcal{I}}}\cong \sgnSymm(\mathcal{I}_{-1})\times \prod_{\mathcal{J}\in\mathcal{I}}\Symm(\mathcal{J})$.
For each $\mathcal{J}\in\mathcal{I}$ it follows that ${}^{{\bf n}}(\bL_{\mathcal{J}})={}^v(\bL_{\mathcal{J}})=\bL_{\ol w(\mathcal{J})}$.
For each $\ol{w}$-orbit $\mathcal{O}$ in $\mathcal{I}$ set $\bL_{\mathcal{O}} := \prod_{\mathcal{J}\in \mathcal{O}} \bL_{\mathcal{J}}$ so that ${}^{\bf n}(\bL_{\mathcal{O}})=\bL_{\mathcal{O}}$.
There is a $\langle {\bf n}, F_p\rangle$-stable decomposition
\[
\bL_{\mathcal{I}_{-1},\mathcal{I}}=\bL_{\mathcal{I}_{-1}}\times \prod_{s=1}^n\left( \prod_{i=1}^{t_s} \bL_{\mathcal{O}^s_i}\right).
\]
Moreover, it follows that $$L_{\mathcal{O}_i^s}:=\bL_{\mathcal{O}^s_{i}}^{{\bf n}F_q}
\cong \left\{
     \begin{array}{lr}
       {\rm GL}_s(q^d) & d \text{ odd; }\\
       {\rm GU}_s(q^{d_0}) & d \text{ even. }\\
      \end{array}
   \right.  $$
Thus for $\epsilon=(-1)^{d+1}$
\[
L_{\mathcal{I}_{-1},\mathcal{I}}:=\bL_{\mathcal{I}_{-1},\mathcal{I}}^{{\bf n}F_q}
\cong \CSp_{2|\mathcal{I}_{-1}|}(q)\times 
\prod_{s=1}^n \GL_{s}(\epsilon q^{d_0})^{t_s}.
\]

\begin{lm}\label{CentreLO}
Let $\mathcal{O}$ be a $\ol{w}$-orbit in $\mathcal{I}$ and ${\bf n}\in V$ as above. 
Set $Z_\mathcal{O}:=Z(L_\mathcal{O})$, $\epsilon=(-1)^{d+1}$ and denote by $C_{q^{d_0}-\epsilon}$ the cyclic subgroup of order $q^{d_0}-\epsilon$ in $\ol{\mathbb{F}}_q^\times$.
There is an isomorphism $\Theta_\mathcal{O}:C_{q^{d_0}-\epsilon}\rightarrow Z_{\mathcal{O}}$, such that $\Theta_\mathcal{O}(t)=\prod_{k\in\ul{\mathcal{O}}}h_{2e_k}(t_k)$ where for $i\in\mathcal{J}_\mathcal{O}$ and $0\leq j\leq d_0-1$:
$$
t_{ja+i}:=\left\{
\begin{array}{lr}
     t^{q^j} & d \text{ even; }\\
\left\{
\begin{array}{lr}
     t^{q^{\frac{j}{2}}} & j \text{ even; }\\
t^{-q^{\frac{d_0+j}{2}}}  & j\text{ odd; }\\
      \end{array}
   \right.      & d \text{ odd.}\\   
      \end{array}
   \right.
   $$   
In particular, $h_{\mathcal{O}}:=\prod_{k\in \ul{\mathcal{O}}}h_{2e_k}(-1)=\Theta_\mathcal{O}(-1)\in Z(L_{\mathcal{O}})$.
\end{lm}
\begin{proof}
From the Chevalley relations, given two roots $\alpha,\beta\in \Phi$, $t\in \ol{\mathbb{F}}_q$ and $u\in \ol{\mathbb{F}}_q^\times$, then $x_{\alpha}(t)^{h_\beta(u)}=x_{\alpha}(u^{\langle \alpha,\beta\rangle} t)$; thus $[x_{e_i-e_j}(t),h_{2e_i}(u)h_{{2e_j}}(u)]=1$ and for $k\not\in \{i,j\}$ then $[x_{e_i-e_j}(t),h_{{2e_k}}(u)]=1$.
Hence for $\mathcal{J}\in\mathcal{O}$, then $Z({\bf L}_\mathcal{J})=\GenGp{\prod_{i\in\mathcal{J}} h_{2e_i}(u)\mid u\in \ol{\mathbb{F}}_q^\times}\cong \ol{\mathbb{F}}_q^\times$.

Let $h=\prod_{k\in\ul{\mathcal{O}}} h_{2e_k}(t_k)\in {\bf T}_\mathcal{O}:=\langle h_{2e_i}(t_i)\mid i\in\ul{\mathcal{O}} \text{ and }t_i\in\ol{\mathbb{F}}_q^\times \rangle$.
Set $\mathcal{L}_{{\bf n}F}$ to be the Lang map corresponding to ${\bf n}F$, that is $\mathcal{L}_{{\bf n}F}(h)=h^{-1}({\bf n}F)(h)$.
Then $\mathcal{L}_{{\bf n}F}(h)=\prod_{k\in\ul{\mathcal{O}}}h_{2e_k}(t'_k)$, for suitable $t'_k\in \ol{\mathbb{F}}_q^\times$.
As $h_{-\alpha}(t)=h_\alpha(t^{-1})$ for any root $\alpha\in \Phi$ and $t\in \ol{\mathbb{F}}_q^\times$ it follows for $i\in\mathcal{J}_\mathcal{O}$ and $0\leq j\leq d_0-1$ that:
$$
t'_{ja+i}=\left\{
\begin{array}{lr}
    \left\{ \begin{array}{lr}
    t_{(d_0-1)a+i}^{-q}t_i^{-1} & j=0; \\
    t_{(j-1)a+i}^qt_{ja+i}^{-1} & 1\leq j\leq d_0-1;\\  
      \end{array}
   \right.     & d \text{ even; }\\
 & \\
\left\{
\begin{array}{lr}
     t_{(d_0-2)a+i}^{-q}t_{i}^{-1} & j=0; \\
     -t_{(d_0-1)a+i}^{-q} t_{a+i}^{-1}& j=1; \\
     t_{(j-2)a+i}^qt_{ja+i}^{-1}  & 2\leq j\leq d_0-1;\\
      \end{array}
   \right.      & d \text{ odd.}\\   
      \end{array}
   \right.
   $$ 
Thus $\mathcal{L}_{{\bf n}F}(h)=1$ if and only if $t'_{ja+i}=1$ for $i\in\mathcal{J}_\mathcal{O}$ and $0\leq j\leq d_0-1$.
In other words:
$$
t_{ja+i}=\left\{
\begin{array}{lr}
    \left\{ \begin{array}{lr}
    t_{(d_0-1)a+i}^{-q} & j=0; \\
    t_{(j-1)a+i}^q & 1\leq j\leq d_0-1;\\  
      \end{array}
   \right.     & d \text{ even; }\\
 & \\
\left\{
\begin{array}{lr}
     t_{(d_0-2)a+i}^{-q} & j=0; \\
     t_{(d_0-1)a+i}^{-q} & j=1; \\
     t_{(j-2)a+i}^q  & 2\leq j\leq d_0-1;\\
      \end{array}
   \right.      & d \text{ odd.}\\   
      \end{array}
   \right.
   $$ 
In particular, each $t_k^{q^{d_0}-\epsilon}=1$.

Combining the description of $Z(\bL_\mathcal{J})$ above with the description for $\mathcal{L}_{{\bf n}F}(h)=1$ it follows that the map $\Theta_\mathcal{O}$ defined in the statement maps $C_{q^{d_0}-\epsilon}$ into $L_\mathcal{O}\cap Z(\bL_{\mathcal{O}})\leq Z_\mathcal{O}$.
In particular, $\Theta_\mathcal{O}$ is an isomorphism.
\end{proof}

\subsection{The normaliser of ${\bf L}_{\mathcal{I}_{-1},\mathcal{I}}$}
\indent

Assume the pair $(\mathcal{I}_{-1},\mathcal{I})$ satisfies Remark~\ref{PartLabLevi} with respect to $w=\rho(v)\in\sgnSymm_n$ from Section~\ref{vwSylTwist}.

\begin{prop}\label{RelWeylGpPart}
Let ${\bf n}\in V$ such that $w^{-1}\rho({\bf n})\in\sgnSymm(\mathcal{I}_{-1})$.
Set $L_{\mathcal{I}_{-1},\mathcal{I}}=\bL_{\mathcal{I}_{-1},\mathcal{I}}^{{\bf n}F}$ and $N_{\mathcal{I}_{-1},\mathcal{I}}:=\NNN_{\bG}(\bL_{\mathcal{I}_{-1},\mathcal{I}})^{{\bf n}F}$.
Then
\[
N_{\mathcal{I}_{-1},\mathcal{I}}/L_{\mathcal{I}_{-1},\mathcal{I}}\cong \prod\limits_{s=1}^n \Cent_{\sgnSymm(\mathcal{I}_s)}(w|_{\pm \mathcal{I}_s}) \cong \prod\limits_{s=1}^n C_{2d_0}\wr \Symm_{t_s}.
\]
\end{prop}
\begin{proof}
Set ${\bf N}_{\mathcal{I}_{-1},\mathcal{I}}:=\NNN_{\bG}(\bL_{\mathcal{I}_{-1},\mathcal{I}})$.
The relative Weyl group (Section~\ref{relWeylGp}) is given by 
$$
W_{\bG}(\bL_{\mathcal{I}_{-1},\mathcal{I}})=\bf N_{\mathcal{I}_{-1},\mathcal{I}}/{{\bL}}_{\mathcal{I}_{-1},\mathcal{I}}\cong \NNN_{W_\bG}(W_{\bL_{\mathcal{I}_{-1},\mathcal{I}}})/W_{\bL_{\mathcal{I}_{-1},\mathcal{I}}}.
$$
Applying Remark~\ref{NormProdsgnSymmSym}, to the description of $W_{{\bf L}_{\mathcal{I}_{-1},\mathcal{I}}}\leq W_{\bf G}$ given in Section~\ref{StrdSpLev} implies
$$\NNN_{W_{\bG}}( W_{\bL_{\mathcal{I}_{-1},\mathcal{I}}})/W_{\bL_{\mathcal{I}_{-1},\mathcal{I}}}\cong \prod_{s=1}^n \sgnSymm(\mathcal{I}_s).$$

By assumption, ${\bf n}F$ induces conjugation by $w|_{\pm\mathcal{I}}=\prod_{s=1}^{n}w|_{\pm\mathcal{I}_s}$ in the relative Weyl group $W_{\bG}(\wt{\bf L}_{\mathcal{I}_{-1},\mathcal{I}})$.
Combining Section~\ref{relWeylGp} with Lemma~\ref{CentsgnSymm}, it follows that 
\begin{align*}{\bf N}_{\mathcal{I}_{-1},\mathcal{I}}^{{\bf n}F}/{\bf L}_{\mathcal{I}_{-1},\mathcal{I}}^{{\bf n}F}\cong \Cent_{W_{\bG}({\bf L}_{\mathcal{I}_{-1},\mathcal{I}})}(w|_{\pm\mathcal{I}})=\prod\limits_{s=1}^n \Cent_{\sgnSymm(\mathcal{I}_s)}(w|_{\pm \mathcal{I}_s})\cong \prod\limits_{s=1}^n C_{2d_0}\wr \Symm_{t_s}. &  \qedhere
\end{align*}

\end{proof}

\subsubsection{Choosing a suitable $v_{\mathcal{I}}\in V_d^\mathcal{I}$ for ${\bf n}$}\label{vI}
\indent

Using the notation from the proof of Lemma~\ref{ExtWeylGpExtMap}, define $v'_\mathcal{I}:=\prod_{s=1}^n\kappa_s (\prod_{i=1}^{t_s} c_{i})\in V_d^\mathcal{I}$ and set 
$$v_\mathcal{I}:= \left\{
     \begin{array}{lr}
       v'_\mathcal{I} & d \text{ odd; }\\
       (v'_\mathcal{I})^2 & d \text{ even. }\\
      \end{array}
   \right.  $$
Observe that $\rho(v'_\mathcal{I})=w'|_{\pm\ul{\mathcal{I}}}$ and so $w^{-1}\cdot \rho(v_\mathcal{I})\in \sgnSymm(\mathcal{I}_{-1})$.

\begin{prop}\label{SuppLinN}
Assume $(\mathcal{I}_{-1},\mathcal{I})$ satisfies Remark~\ref{PartLabLevi} with respect to $w$ from Section~\ref{vwSylTwist}.
For $v_{\mathcal{I}}$ defined as above, set $L_{\mathcal{I}_{-1},\mathcal{I}}=\bL_{\mathcal{I}_{-1},\mathcal{I}}^{v_\mathcal{I}F}$ and $N_{\mathcal{I}_{-1},\mathcal{I}}=\NNN_{\bG}(\bL_{\mathcal{I}_{-1},\mathcal{I}})^{v_\mathcal{I} F}$.
Then for $V_d^\mathcal{I}$ as in Lemma~\ref{ExtWeylGpExtMap}, $v'_\mathcal{I}\in Z(V_d^\mathcal{I})$  and

\[
N_{\mathcal{I}_{-1},\mathcal{I}}=L_{\mathcal{I}_{-1},\mathcal{I}}V_d^\mathcal{I}.
\]
Furthermore $H_d^\mathcal{I}:=H\cap V_d^\mathcal{I}\leq Z(L_{\mathcal{I}_{-1},\mathcal{I}})$ and maximal extendibility holds with respect to $H_d^\mathcal{I}\lhd V_d^\mathcal{I}$.
\end{prop}
\begin{proof}
By Lemma~\ref{CentreLO} $H_d^\mathcal{I}\leq Z({L}_{\mathcal{I}_{-1},\mathcal{I}})$ and by Lemma~\ref{ExtWeylGpExtMap} maximal extendibility holds with respect to $H_d^\mathcal{I}\lhd V_d^\mathcal{I}$.
As in the proof of Lemma~\ref{CentralCVv}, the element $\prod_{i=1}^{t_s} c_{i}$ lies in $Z((V_d^{\mathcal{I}_s})^{\circ})$ and hence $v_\mathcal{I}'\in Z(V_d^\mathcal{I})$.

For $x\in V_d^\mathcal{I}$ and $\mathcal{J}\in\mathcal{I}$, it follows that $({\bL}_\mathcal{J})^x={\bL}_{\ol{\rho}(x)(\mathcal{J})}$.
Hence $V_d^\mathcal{I}\leq \NNN_{\bG}({\bL}_{\mathcal{I}_{-1},\mathcal{I}})$ and ${\bL}_{\mathcal{I}_{-1},\mathcal{I}}\cap V_d^\mathcal{I}=H_d^\mathcal{I}$.
Therefore $({\bL}_{\mathcal{I}_{-1},\mathcal{I}}V_d^\mathcal{I})^ {v_\mathcal{I}F}={L}_{\mathcal{I}_{-1},\mathcal{I}}V_d^\mathcal{I}\leq N_{\mathcal{I}_{-1},\mathcal{I}}$ and the quotient ${L}_{\mathcal{I}_{-1},\mathcal{I}} V_d^\mathcal{I}/{L}_{\mathcal{I}_{-1},\mathcal{I}}\cong W_d^\mathcal{I}$.
By construction $W_d^\mathcal{I}$ induces $\Cent_{W_{\bG}({\bf L}_{\mathcal{I}_{-1},\mathcal{I}})}(w|_{\pm\mathcal{I}})$ as a subgroup of $\Cent_{\sgnSymm (\ul{\mathcal{I}})}(w|_{\pm\ul{\mathcal{I}}})$.
Hence Proposition~\ref{RelWeylGpPart} implies $N_{\mathcal{I}_{-1},\mathcal{I}}=L_{\mathcal{I}_{-1},\mathcal{I}}V_d^\mathcal{I}$.
\end{proof}

\begin{rem}
The rational behind replacing $v$ in the previous setup is that $v$ need not always lie in $V_d^\mathcal{I}$, for example whenever $v$ acts non-trivial on $\Phi_{-1}$.
Moreover, the element $v'_\mathcal{I}$ here resembles the decomposition of the Sylow twist in type A \cite{BrSpAM}.
\end{rem}

\subsection{The structure of $d$-split Levi subgroups of $\wt{\bG}$}
\indent 

Assume the pair $(\mathcal{I}_{-1},\mathcal{I})$ satisfies Remark~\ref{PartLabLevi} with respect to $w=\rho(v)\in\sgnSymm_n$ from Section~\ref{vwSylTwist}.
Let $\bL_{\mathcal{I}_{-1},\mathcal{I}}$ denote the corresponding $d$-split Levi subgroup in $(\bG,vF)$ containing the maximal torus $\bT$, as described in Section~\ref{StrdSpLev}.

Recall that $\wt{\bG}$ has a connected centre with $\wt{\bG}=\bG Z(\wt{\bG})$. 
Thus $\wt{\bL}_{\mathcal{I}_{-1},\mathcal{I}}:=\bL_{\mathcal{I}_{-1},\mathcal{I}}Z(\wt{\bG})$ is a $d$-split Levi subgroup in $(\wt{\bG},vF)$.
Take $v_I$ defined in Section~\ref{vI} and set $\wt{L}_{\mathcal{I}_{-1},\mathcal{I}}:=\wt{\bL}_{\mathcal{I}_{-1},\mathcal{I}}^{v_IF}$ and $L_{\mathcal{I}_{-1},\mathcal{I}}:=\bL_{\mathcal{I}_{-1},\mathcal{I}}^{v_\mathcal{I}F}$.

\begin{rem}\label{FixPreImInv}
The quotient $\wt{L}_{\mathcal{I}_{-1},\mathcal{I}}/L_{\mathcal{I}_{-1},\mathcal{I}}\cong \wt{G}/G\cong C_{q-1}$ is cyclic.
Hence there is some $g\in {\bf L}_{\mathcal{I}_{-1},\mathcal{I}}$ and $x\in Z(\wt{\bG})$ such that $\wt{L}_{\mathcal{I}_{-1},\mathcal{I}}=\langle L_{\mathcal{I}_{-1},\mathcal{I}},gx\rangle$.
Let $\mathcal{L}_{v_\mathcal{I}F}$ denote the Lang map with respect to $v_\mathcal{I}F$.
As $x\in Z(\wt{\bG})$, it follows that $1=\mathcal{L}_{v_\mathcal{I}F}(gx)=\mathcal{L}_{v_\mathcal{I}F}(g)\mathcal{L}_{v_\mathcal{I}F}(x)$, that is $\mathcal{L}_{v_\mathcal{I}F}(g)$ lies in ${\bf L}_{\mathcal{I}_{-1},\mathcal{I}}\cap Z(\wt{\bG})=Z(\bG)=\langle z\rangle$, for an involution $z$ from Section~\ref{CnSetup}.


\end{rem}

\begin{cor}\label{IntTilLIndex2}
Let $x\in Z(\wt{\bG})$ and $g\in  {\bf L}_{\mathcal{I}_{-1},\mathcal{I}}$ such that  $\wt{L}_{\mathcal{I}_{-1},\mathcal{I}}=\langle L_{\mathcal{I}_{-1},\mathcal{I}},gx\rangle$.
Then $\wt{L}_{\mathcal{I}_{-1},\mathcal{I}}=\langle L_{\mathcal{I}_{-1},\mathcal{I}},hx\rangle$ for any $h\in  {\bf L}_{\mathcal{I}_{-1},\mathcal{I}}$ with $\mathcal{L}_{v_\mathcal{I}F}(h)=z$ and $\GenGp{L_{\mathcal{I}_{-1},\mathcal{I}},x^2}$ has index two in $\wt{L}_{\mathcal{I}_{-1},\mathcal{I}}$.
\end{cor}
\begin{proof}
By assumption $\mathcal{L}_{v_\mathcal{I}F}(gh^{-1})=hg^{-1}F(g)F(h)^{-1}=\mathcal{L}_{v_\mathcal{I}F}(g)\mathcal{L}_{v_\mathcal{I}F}(h)^{-1}$ as $hF(h)^{-1}=F(h)^{-1}h=(h^{-1}F(h))^{-1}=z$.
Hence $gh^{-1}\in L_{\mathcal{I}_{-1},\mathcal{I}}$ and $\langle L_{\mathcal{I}_{-1},\mathcal{I}},gx\rangle=\langle L_{\mathcal{I}_{-1},\mathcal{I}},hx\rangle$.
The second result follows from the fact that $g^2\in L_{\mathcal{I}_{-1},\mathcal{I}}$.
\end{proof}

\subsubsection{An element $z^+\in {\bf L}_{\mathcal{I}_{-1},\mathcal{I}}$ with $\mathcal{L}_{v_\mathcal{I}F}(z^+)=z$}\label{LangPreImage}
\indent

Recall from Section~\ref{CnSetup} that the central involution $z=h_{2e_1}(-1)\cdots h_{2e_n}(-1)$.
For each $\ol{w}$-orbit $\mathcal{O}\subseteq\mathcal{I}$, set $z_\mathcal{O}:=\prod_{k\in \ul{\mathcal{O}}}h_{2e_k}(-1)$ and similarly $z_{\mathcal{I}_{-1}}:=\prod_{k\in \mathcal{I}_{-1}}h_{2e_k}(-1)$ so that $z=z_{\mathcal{I}_{-1}}\times\prod_\mathcal{O} z_\mathcal{O}$.
In particular, using the notation from Section~\ref{CnSetup}, $z_\mathcal{O}\in \bT_\mathcal{O}$ and $z_{\mathcal{I}_{-1}}\in \bT_{\mathcal{I}_{-1}}$.
As $v_\mathcal{I}F$ fixes ${\bf T}_{\mathcal{I}_{-1}}$ and $\bf T_\mathcal{O}$, there will be a $z_\mathcal{O}^+\in {\bf T}_\mathcal{O}$ and $z_{\mathcal{I}_{-1}}^+\in \bT_{\mathcal{I}_{-1}}$ such that $\mathcal{L}_{v_\mathcal{I}F}(z_{\mathcal{O}}^+)=z_{\mathcal{O}}$ and $\mathcal{L}_{v_\mathcal{I}F}(z_{\mathcal{I}_{-1}}^+)=z_{\mathcal{I}_{-1}}$.
Assume that a choice of $z_{\mathcal{I}_{-1}}^+\in{\bf T}_{\mathcal{I}_{-1}}$ has been fixed.
Here an explicit description of a choice for $z_\mathcal{O}^+$ is given.

For an element $h=\prod_{k\in\ul{\mathcal{O}}} h_{2e_k}(t_k)\in \bT_\mathcal{O}$ repeating the argument in the proof of Lemma~\ref{CentreLO} shows that $\mathcal{L}_{v_\mathcal{I}F}(h)=z_\mathcal{O}$ if and only if for $i\in\mathcal{J}_\mathcal{O}$ and $0\leq j\leq d_0-1$:

$$
t_{ja+i}=\left\{
\begin{array}{lr}
    \left\{ \begin{array}{lr}
    -t_{(d_0-1)a+i}^{-q} & j=0; \\
    -t_{(j-1)a+i}^q & 1\leq j\leq d_0-1;\\  
      \end{array}
   \right.     & d \text{ even; }\\
 & \\
\left\{
\begin{array}{lr}
     -t_{(d_0-2)a+i}^{-q} & j=0; \\
     -t_{(d_0-1)a+i}^{-q} & j=1; \\
     -t_{(j-2)a+i}^q  & 2\leq j\leq d_0-1;\\
      \end{array}
   \right.      & d \text{ odd.}\\   
      \end{array}
   \right.
   $$
In particular, for $\epsilon=(-1)^{d+1}$ and $i\in\mathcal{J}_\mathcal{O}$ each $t_i^{q^{d_0}-\epsilon}=(-1)^{d_0}$.  
Fix $\psi\in\ol{\mathbb{F}}_q^\times$ such that $\psi^{q^{d_0}-\epsilon}=(-1)^{d_0}$ and define 
$z_{\mathcal{O}}^+:=\prod_{k\in\ul{\mathcal{O}}}h_{2e_k}(t_k)$, where for $i\in\mathcal{J}_\mathcal{O}$ and $0\leq j\leq d_0-1$:

$$
t_{ja+i}:=\left\{
\begin{array}{lr}
     (-1)^j & d_0 \text{ even; }\\
(-1)^j\psi^{q^j}  & d\text{ even and } d_0 \text{ odd; }\\
\left\{
\begin{array}{lr}
     (-1)^{\frac{j}{2}}\psi^{q^{\frac{j}{2}}} & j \text{ even; }\\
(-1)^{\frac{d_0+j}{2}}\psi^{-q^{\frac{d_0+j}{2}}}  & j\text{ odd; }\\
      \end{array}
   \right.      & d \text{ odd.}\\   
      \end{array}
   \right.
   $$   

The element $z^+:=z_{\mathcal{I}_{-1}}^+\times \prod z_{\mathcal{O}}^+$ satisfies $\mathcal{L}_{vF}(z^+)=z$. 
For each $1\leq s\leq n$, set $z_{s}^+=\prod_{\mathcal{O}}z_{\mathcal{O}}^+$ running over all $\ol{w}$-orbits $\mathcal{O}\subset \mathcal{I}_s$ so that $z^+=z_{\mathcal{I}_{-1}}^+\times\prod_{s=1}^n z_{s}^+$.

\begin{rem}\label{zcirc}
Take $z_\mathcal{O}^+$ as constructed above.
For a fixed $0\leq j\leq d_0-1$, the definition implies $t_{ja+i}=t_{ja+i'}$ for all $i,i'\in \mathcal{J}_\mathcal{O}$.
Applying the proof of Lemma~\ref{CentreLO} shows $\prod_{i\in\mathcal{J}_\mathcal{O}}h_{2e_{ja+i}}(t_{ja+i})\in Z({\bf L}_{\ol{w}^j(\mathcal{J}_\mathcal{O})})$.   
Hence $z_{\mathcal{O}}^+\leq Z(\bf{L}_\mathcal{O})$.

In the case $d_0$ is even, then $(z_{\mathcal{O}}^+)^2=1$.
When $d_0$ is odd then Lemma~\ref{CentreLO} implies that $\langle (z_{\mathcal{O}}^+)^2\rangle =Z_\mathcal{O}$, as $\langle \psi^2\rangle\cong C_{q^{d_0}-\epsilon}\leq \ol{\mathbb{F}}_p$.
\end{rem}

\subsubsection{A larger group containing $\wt{L}_{\mathcal{I}_{-1},\mathcal{I}}$}\label{TiL0}
 \indent
 
To study the group $\wt{L}_{\mathcal{I}_{-1},\mathcal{I}}$ and its irreducible characters in Corollary~\ref{IntTilLNhat} and Section~\ref{WtilLam}, it will be convenient to construct a larger group, $\wt{L}_{\mathcal{I}_{-1},\mathcal{I}}^{+}$.

Let $z^+\in{\bf T}$ be the element constructed in Section~\ref{LangPreImage}.
The group $\wt{L}_{\mathcal{I}_{-1},\mathcal{I}}=\GenGp{{L}_{\mathcal{I}_{-1},\mathcal{I}}, z^+ x}$ embeds into $\langle L_{\mathcal{I}_{-1},\mathcal{I}},z^+\rangle\circ \langle x\rangle$.
Using the description of $z^+$, for each $\ol{w}$-orbit $\mathcal{O}\subseteq \mathcal{I}$, set $\wt{L}_{\mathcal{O}}^+:=\langle L_{\mathcal{O}},z_{\mathcal{O}}^+\rangle$ and $\wt{L}_{\mathcal{I}_{-1}}^+:=\langle L_{\mathcal{I}_{-1}},z_{\mathcal{I}_{-1}}^+\rangle$.
Then
$$\wt{L}_{\mathcal{I}_{-1},\mathcal{I}}\hookrightarrow \wt{L}_{\mathcal{I}_{-1},\mathcal{I}}^+:=\left(\wt{L}_{\mathcal{I}_{-1}}^+ \times \prod_{\mathcal{O}} \wt{L}_{\mathcal{O}}^+\right)\circ \langle x\rangle.$$

By Remark~\ref{zcirc}, for each $\ol{w}$-orbit $\mathcal{O}\subseteq \mathcal{I}$, the group $Z_{\mathcal{O}}^+:=\langle z_{\mathcal{O}}^+\rangle$ is a subgroup of $Z(\bf{L}_\mathcal{O})$.
Thus $$\wt{L}_{\mathcal{O}}^+ \cong Z_{\mathcal{O}}^+\circ_{\langle (z_{\mathcal{O}}^+)^2\rangle}L_\mathcal{O}$$ which will be a direct product when $d_0$ is even.
Finally the group $\wt{L}_{\mathcal{I}_{-1}}^+$ induces diagonal automorphisms, arising from ${\rm CSp}_{2|\mathcal{I}_{-1}|}(q)$, on the factor $L_{\mathcal{I}_{-1}}\cong {\rm Sp}_{2|\mathcal{I}_{-1}|}(q)$.

\begin{lm}\label{InTilL}
The group $\wt{L}_{\mathcal{I}_{-1},\mathcal{I}}^+$ normalises $L_{\mathcal{I}_{-1},\mathcal{I}}$.
For each $\lambda\in\Irr(L_{\mathcal{I}_{-1},\mathcal{I}})$ denote by $\lambda_{-1}:={\rm Res}_{L_{\mathcal{I}_{-1}}}^{L_{\mathcal{I}_{-1},\mathcal{I}}}(\lambda)$.
Then $$(\wt{L}_{\mathcal{I}_{-1},\mathcal{I}}^+)_{\lambda}=\left( (\wt{L}_{\mathcal{I}_{-1}}^+)_{\lambda_{-1}} \times \prod_{\mathcal{O}} \wt{L}_{\mathcal{O}}^+\right)\circ \langle x\rangle.$$
In particular, $(\wt{L}_{\mathcal{I}_{-1},\mathcal{I}})_{\lambda_\mathcal{O}}=\wt{L}_{\mathcal{I}_{-1},\mathcal{I}}$, where $\lambda_{\mathcal{O}}:={\rm Res}_{L_{\mathcal{O}}}^{L_{\mathcal{I}_{-1},\mathcal{I}}}(\lambda)$ for each $\ol{w}$-orbit $\mathcal{O}\subset \mathcal{I}$.
\end{lm}
\begin{proof}
To show $\wt{L}_{\mathcal{I}_{-1},\mathcal{I}}^+$ normalises $L_{\mathcal{I}_{-1},\mathcal{I}}$, it suffices to show ${\wt{L}}_{\mathcal{I}_{-1}}^+$ normalises $L_{\mathcal{I}_{-1}}$. As $\mathcal{L}_{v_\mathcal{I}F}(z_{\mathcal{I}_{-1}}^+)=z_{\mathcal{I}_{-1}}$ it follows that $({\bf L}_{\mathcal{I}_{-1}}^{v_\mathcal{I}F})^{z_{\mathcal{I}_{-1}}^+}={\bf L}_{\mathcal{I}_{-1}}^{z_{\mathcal{I}_{-1}}v_\mathcal{I}F}={L}_{\mathcal{I}_{-1}}$.

The structure of $(\wt{L}_{\mathcal{I}_{-1},\mathcal{I}}^+)_{\lambda}$ follows from the fact that $Z_{\mathcal{O}}^+\leq Z(\bL_\mathcal{O})$ and $x\in Z(\wt{\bG})$.
\end{proof}

\begin{cor}\label{IntTilLNhat}
Let $\wt{N}_{\mathcal{I}_{-1},\mathcal{I}}:=\NNN_{\wt G}(\wt{\bf L}_{\mathcal{I}_{-1},\mathcal{I}})=\wt{L}_{\mathcal{I}_{-1},\mathcal{I}}N_{\mathcal{I}_{-1},\mathcal{I}}$ and $\wh{N}_{\mathcal{I}_{-1},\mathcal{I}}:=N_{\mathcal{I}_{-1},\mathcal{I}}E$, for $E$ defined in Section~\ref{endo}.
Then for any $\lambda\in\Irr(L_{\mathcal{I}_{-1},\mathcal{I}})$
\begin{enumerate}
\item $(\wt{N}_{\mathcal{I}_{-1},\mathcal{I}})_\lambda=(\wt{L}_{\mathcal{I}_{-1},\mathcal{I}})_\lambda (N_{\mathcal{I}_{-1},\mathcal{I}})_\lambda$ and
\item $(\wt{N}_{\mathcal{I}_{-1},\mathcal{I}}\wh{N}_{\mathcal{I}_{-1},\mathcal{I}})_\lambda=(\wt{N}_{\mathcal{I}_{-1},\mathcal{I}})_{\lambda}(\wh{N}_{\mathcal{I}_{-1},\mathcal{I}})_\lambda$.
\end{enumerate}
\end{cor}
\begin{proof}
For each $1\leq s\leq n$ set $\lambda_s$ to be the restriction of $\lambda$ to $L_s:=\prod_{i=1}^{t_s} L_{\mathcal{O}_i^s}$ and $\lambda_{-1}$ the restriction of $\lambda$ to $L_{\mathcal{I}_{-1}}$.
Recall from Proposition~\ref{SuppLinN} that $N_{\mathcal{I}_{-1},\mathcal{I}}=L_{\mathcal{I}_{-1},\mathcal{I}}V_d^\mathcal{I}$.
The construction of $V_d^\mathcal{I}$ given in the proof of Lemma~\ref{ExtWeylGpExtMap}, implies $[V_d^\mathcal{I},L_{\mathcal{I}_{-1}}]=1$ and $V_d^\mathcal{I}$ normalises $L_s$ for each $1\leq s\leq n$. 
Moreover, as $E$ acts as field automorphisms, it must normalise $L_{\mathcal{I}_{-1}}$ and each $L_s$ for $1\leq s\leq n$.
Thus for $\wt{l}\in \wt{L}_{\mathcal{I}_{-1},\mathcal{I}}$, $k\in \wh{N}_{\mathcal{I}_{-1},\mathcal{I}}$, Lemma~\ref{InTilL} implies
$$\lambda^{\wt{l}k}=\lambda_{-1}^{\wt{l}k} \times \prod_{s=1}^n \lambda_s^k.$$

If $\lambda^{\wt{l}k}=\lambda$, then $(\lambda_{-1})^{\wt{l}k}=\lambda_{-1}$.
On $L_{\mathcal{I}_{-1}}$, the element $\wt{l}$ induces a diagonal automorphism, while $k$ acts as a field automorphism arising from $E$.
Thus \cite[Theorem 3.1]{CabSpMcKTypC} implies that $(\lambda_{-1})^{\wt{l}}=\lambda_{-1}$.
Hence $\wt{l}\in (\wt{L}_{\mathcal{I}_{-1},\mathcal{I}})_\lambda$.
In other words $(\wt{L}_{\mathcal{I}_{-1},\mathcal{I}}\wh{N}_{\mathcal{I}_{-1},\mathcal{I}})_\lambda\leq  (\wt{L}_{\mathcal{I}_{-1},\mathcal{I}})_\lambda (\wh{N}_{\mathcal{I}_{-1},\mathcal{I}})_\lambda$.

From this it follows that $(\wt{N}_{\mathcal{I}_{-1},\mathcal{I}})_\lambda=(\wt{L}_{\mathcal{I}_{-1},\mathcal{I}})_\lambda (N_{\mathcal{I}_{-1},\mathcal{I}})_\lambda$ and 
\begin{align*}(\wt{L}_{\mathcal{I}_{-1},\mathcal{I}})_\lambda (\wh{N}_{\mathcal{I}_{-1},\mathcal{I}})_\lambda=(\wt{L}_{\mathcal{I}_{-1},\mathcal{I}})_\lambda (N_{\mathcal{I}_{-1},\mathcal{I}})_\lambda (\wh{N}_{\mathcal{I}_{-1},\mathcal{I}})_\lambda=(\wt{N}_{\mathcal{I}_{-1},\mathcal{I}})_\lambda (\wh{N}_{\mathcal{I}_{-1},\mathcal{I}})_\lambda. &  \qedhere
\end{align*}
\end{proof}
 
\subsubsection{Some properties of $\wt{L}_{\mathcal{I}_{-1},\mathcal{I}}^+$}  
 \indent
 
To finish this section, some properties relating to $\wt{L}_{\mathcal{I}_{-1},\mathcal{I}}^+$ are given.
First the following remark considers the action of $N_{\mathcal{I}_{-1},\mathcal{I}}$ on $Z(\wt{L}_{\mathcal{I}_{-1},\mathcal{I}}^+)$, which will be applied in Lemma~\ref{TilL0StabN} and in Section~\ref{WtilLCirc}.

\begin{rem}\label{VactZtilL}
Let $\mathcal{O}$ be a $\ol{w}$-orbit in $\mathcal{I}$.
Then $Z(\wt{L}_{\mathcal{O}}^+)=\langle Z_{\mathcal{O}}^+,Z_{{\mathcal{O}}}\rangle\leq \bT_\mathcal{O}$.
By Proposition~\ref{SuppLinN}, $N_{\mathcal{I}_{-1},\mathcal{I}}=L_{\mathcal{I}_{-1},\mathcal{I}}V_d^\mathcal{I}$, where $V_d^\mathcal{I}$ was constructed in the proof of Lemma~\ref{ExtWeylGpExtMap}.

As $\rho(V_d^\mathcal{I})=W_d^\mathcal{I}\leq \prod_{s=1}^n\sgnSymm(\ul{\mathcal{I}}_s)$ (see Section~\ref{CentwIinCentwulI}), it follows that $[V_d^\mathcal{I},{\bL}_{\mathcal{I}_{-1}}]=1$.
Recall that $V_d^\mathcal{I}=\prod_{s=1}^n V_d^{\mathcal{I}_s}$ and in a similar fashion, $[V_d^{\mathcal{I}_s},{\bL}_{\mathcal{O}}]=1$, whenever the $\ol{w}$-orbit $\mathcal{O}$ is not in $\mathcal{I}_s$.
As in Section~\ref{wPartNot}, for $1\leq s\leq n$, let $\mathcal{O}_1^s,\dots,\mathcal{O}_{t_s}^s$ denote the $\ol{w}$-orbits in $\mathcal{I}_s$ with $Q_i^s\subset \mathcal{J}^{d_0,a}$ corresponding to $\mathcal{O}_i^s$. 
Using the notation in the proof of Lemma~\ref{ExtWeylGpExtMap} and Section~\ref{CentwIinCentwulI},
$\rho(\kappa_s(h_{i}))=1$ and $\rho(\kappa_s(c_{i}))=w'_{Q_i^s}$ for $1\leq i\leq t_s$, while  $\rho(\kappa_s(p_{i}))=\prod_{j=1}^s \tau_{Q_{i,j}^s,Q_{i+1,j}^s}$ for $1\leq i\leq t_s-1$.
For $x\in V_d^{\mathcal{I}_s}=\kappa_s((V_d^{\mathcal{I}_s})^\circ)$, and $\rho(x)\in \sgnSymm (\ul{\mathcal{I}}_s)$, the Chevalley relations imply $h_{2e_i}(t)^x=h_{2e_{\rho(x)(i)}}(t)$. 
Hence it is immediate for any $1\leq s\leq n$ and $1\leq i\leq t_s$ that $[\kappa_s(h_{i}),Z_\mathcal{O}^+]=1$ for any $\ol{w}$-orbit $\mathcal{O}$ in $\mathcal{I}$.

Using the isomorphism $\Theta_{\mathcal{O}}:C_{q^{d_0}-\epsilon}\rightarrow Z_\mathcal{O}$ from Lemma~\ref{CentreLO} and the description of $z_\mathcal{O}^+$ given in Section~\ref{LangPreImage} it follows that 
$$(\Theta_{\mathcal{O}}(t))^{\kappa_s(p_{i})}=\left\{ \begin{array}{ll}
   \Theta_{\mathcal{O}_{i+1}^s}(t) & \ul{\mathcal{O}}=\ul{Q}_i^s;\\
 \Theta_{\mathcal{O}_{i}^s}(t) & \ul{\mathcal{O}}=\ul{Q}_{i+1}^s;\\
 \Theta_{\mathcal{O}}(t) & \text{otherwise};\\
      \end{array}
   \right.
\text{ and }   
   (z_{\mathcal{O}}^+)^{\kappa_s(p_{i})}=\left\{ \begin{array}{ll}
   z_{\mathcal{O}_{i+1}^s}^+ & \ul{\mathcal{O}}=\ul{Q}_i^s;\\
 z_{\mathcal{O}_{i}^s}^+ & \ul{\mathcal{O}}=\ul{Q}_{i+1}^s;\\
 z_{\mathcal{O}}^+ & \text{otherwise}.\\
      \end{array}
   \right.
   $$ 
In particular, $\kappa_s(p_{i})$ permutes the two groups $Z(\wt{L}_{\mathcal{O}_{i}^s}^+)$ and $Z(\wt{L}_{\mathcal{O}_{i+1}^s}^+)$ while the commutator $[\kappa_s(p_{i}), Z(\wt{L}_{\mathcal{O}}^+)]=1$ for every other $\ol{w}$-orbit $\mathcal{O}\subset\mathcal{I}$.

Finally for $1\leq s\leq n$ and $1\leq i\leq t_s$ consider the action of the elements $\kappa_s(c_{i})$.
Note that $[\kappa_s(c_{i}),{\bf L}_{\mathcal{O}}]=1$ whenever $\mathcal{O}\ne \mathcal{O}_i^s$.
Denote by $x_\mathcal{O}:=\kappa_s(c_{i})$, where $\mathcal{O}=\mathcal{O}_i^s$, so that $\rho(x_\mathcal{O})=w'_{Q_i^s}$.
Recall $\bT_\mathcal{O}=\langle h_{2e_{ja+i}}(t_{ja+i})\mid t_{ja+i}\in\ol{\mathbb{F}}_p^\times \text{, } i\in\mathcal{J}_\mathcal{O} \text{ and } 0\leq j\leq d_0-1\rangle$. 
For $i\in\mathcal{J}_\mathcal{O}$ and $0\leq j\leq d_0-1$,
$$(h_{2e_{ja+i}}(t_{ja+i}))^{x_\mathcal{O}^{-1}}=\left\{
\begin{array}{lr}
     h_{2e_{(j-1)a+i}}(t_{ja+i}) & j\ne 0; \\
     h_{2e_{(d_0-1)a+i}}(t_{i}^{-1}) & j=0.\\
      \end{array}
   \right.\\
$$
Note that $x_\mathcal{O}^{-1}$ is considered to make the following expressions easier to display.
Recall that the definition of $z_\mathcal{O}^+$ depended upon $d_0$ being odd or even.
These two cases will be studied separately. 

\indent

\ul{Case 1: \bf $d_0$ even}

In this case $(z_{\mathcal{O}}^+)^2=1$.
Take $\psi\in\ol{\mathbb{F}}_p^\times$ such that $\langle \psi\rangle= C_{q^{d_0}+1}$.
Then it follows that $Z(\wt{L}_{\mathcal{O}}^+)=\langle z_{\mathcal{O}}^+\rangle\times \langle \Theta_\mathcal{O}(\psi)\rangle$.
Direct calculation shows $(z_{\mathcal{O}}^+)^{x_\mathcal{O}^{-1}}=z_{\mathcal{O}}^+ \Theta_\mathcal{O}(-1)$ and $$\left(  \prod_{j=0}^{d_0-1} h_{2e_{ja+i}}(\psi^{q^{j}})\right)^{x_\mathcal{O}^{-1}}=\left(\prod_{j=0}^{d_0-2}h_{2e_{ja+i}}(\psi^{q^{j+1}})\right)h_{2e_{(d_0-1)a+i}}(\psi^{-q}).$$
Thus the action of ${x_\mathcal{O}^{-1}}$ on $Z(\wt{L}_{\mathcal{O}}^+)$ is given by:
$$
\begin{array}{ccc}
C_2\times C_{q^{d_0}+1} & \longrightarrow& C_2\times C_{q^{d_0}+1} \\

(z_\mathcal{O}^+,1) & \mapsto & (z_\mathcal{O}^+,\Theta_\mathcal{O}(\psi^{\frac{q^{d_0}+1}{2}}))\\
(1,\Theta_\mathcal{O}(\psi)) & \mapsto & (1,\Theta_\mathcal{O}(\psi^{q}))
\end{array}
$$

\ul{Case 2: \bf $d_0$ is odd}

\vspace*{1mm}
In this case 
$Z(\wt{L}_{\mathcal{O}}^+)=\langle z_{\mathcal{O}}^+\rangle$, where $z_{\mathcal{O}}^+$ was constructed via a fixed $\psi\in\ol{\mathbb{F}}_p^\times$ of order $2(q^{d_0}-\epsilon)$.
In particular, $(z_{\mathcal{O}}^+)^2=\Theta_\mathcal{O}(\psi^2)$ generates $Z_\mathcal{O}$.
Therefore $\Theta_\mathcal{O}$ extends to an isomorphism from $\langle \psi\rangle$ to  $Z_{\mathcal{O}}^+$, by defining $\Theta_\mathcal{O}(\psi)=z_{\mathcal{O}}^+$.
Direct calculation shows $$
 h_{2e_{a+i}}(-\psi^{q})^{x_\mathcal{O}^{-1}}=h_{2e_{i}}(-\psi^{q})$$
and
$$
h_{2e_{a+i}}((-1)^{\frac{d_0+1}{2}}\psi^{-q^{\frac{d_0+1}{2}}})^{x_\mathcal{O}^{-1}}=h_{2e_{i}}((-1)^{\frac{d_0+1}{2}}\psi^{-q^{\frac{d_0+1}{2}}}).
$$
Hence, as $-1=\psi^{q^{d_0}-\epsilon}$, the action of ${x_\mathcal{O}^{-1}}$ on $Z(\wt{L}_{\mathcal{O}}^+)$ is given by:
$$
\begin{array}{ccc}
C_{2(q^{d_0+1})} & \longrightarrow& C_{2(q^{d_0+1})} \\

\Theta_\mathcal{O}(\psi) & \mapsto & 
\left\{
\begin{array}{lr}
     \Theta_\mathcal{O}(\psi)^{q+q^{d_0}+1} & d \text{ even}\\
    \Theta_\mathcal{O}(\psi)^{(\frac{d_0+1}{2})(q^{d_0}-1)-q^{\frac{d_0+1}{2}}} & d \text{ odd}\\
      \end{array}
   \right.\\
\end{array}
$$

\end{rem}

\begin{lm}\label{TilL0StabN}
Let $\lambda\in \Irr(L_{\mathcal{I}_{-1},\mathcal{I}})$.
\begin{enumerate}
\item $\lambda$ extends to $(\wt{L}_{\mathcal{I}_{-1},\mathcal{I}}^+)_\lambda$.
\item  Both $\wt{L}_{\mathcal{I}_{-1},\mathcal{I}}^+$ and $(\wt{L}_{\mathcal{I}_{-1},\mathcal{I}}^+)_\lambda$ are normalised by ${N}_{\mathcal{I}_{-1},\mathcal{I}}=\NNN_G({\bL}_{\mathcal{I}_{-1},\mathcal{I}})$. 
\item Given extensions $\wt{\lambda}\in \Irr((\wt{L}_{\mathcal{I}_{-1}})_\lambda\mid \lambda)$ and $\wt{\lambda}^+\in \Irr((\wt{L}_{\mathcal{I}_{-1},\mathcal{I}}^+)_\lambda\mid \lambda)$ then 
$$(V_d^\mathcal{I})_{\wt{\lambda}^+}\leq (V_d^\mathcal{I})_{\wt{\lambda}} \leq (V_d^\mathcal{I})_{{\lambda}}.$$

\end{enumerate}
\end{lm}
\begin{proof}
Set  $\lambda_{\mathcal{I}_{-1}}:={\rm Res}_{L_{\mathcal{I}_{-1}}}^{L_{\mathcal{I}_{-1},\mathcal{I}}}(\lambda)\in \Irr(L_{\mathcal{I}_{-1}})$ and $\lambda_{\mathcal{O}}:={\rm Res}_{L_{\mathcal{O}}}^{L_{\mathcal{I}_{-1},\mathcal{I}}}(\lambda)\in \Irr(L_\mathcal{O})$, for each $\ol{w}$-orbit $\mathcal{O}\subset\mathcal{I}$.
Due to Lemma~\ref{InTilL} and the description of characters in a central product given in \cite[Section 5]{IMNRedMcKay}, to prove $(1)$ it suffices to see each $\lambda_\mathcal{O}$ extends to $\wt{L}_{\mathcal{O}}^+$ and $\lambda_{{-1}}$ extends to $(\wt{L}_{\mathcal{I}_{-1}}^+)_{\lambda_{-1}}$.
This is immediate as the quotients are cyclic.

Now consider $(2)$.
By Proposition~\ref{SuppLinN}, $N_{\mathcal{I}_{-1},\mathcal{I}}=L_{\mathcal{I}_{-1},\mathcal{I}}V_d^\mathcal{I}$, where $V_d^\mathcal{I}$ was constructed in the proof of Lemma~\ref{ExtWeylGpExtMap}.
As in Remark~\ref{VactZtilL}, $[V_d^\mathcal{I},{\bL}_{\mathcal{I}_{-1}}]=1$ and for $1\leq s\leq n$ the commutator $[V_d^{\mathcal{I}_s},{\bL}_{\mathcal{O}}]=1$, whenever the $\ol{w}$-orbit $\mathcal{O}$ is not in $\mathcal{I}_s$.
Thus it suffices to show for each $1\leq s\leq n$ that  $\wt{L}_s^+:=\prod_{i=1}^{t_s} \wt{L}_{\mathcal{O}_i^s}^+$ is normalised by $V_d^{\mathcal{I}_s}$.
Moreover, as $\wt{L}_s^+=\langle L_s, Z(\wt{L}_s^+)\rangle$, it is sufficient to show $Z(\wt{L}_s^+)$ is normalised by $V_d^{\mathcal{I}_s}$.
The result now follows from Remark~\ref{VactZtilL}, as $Z(\wt{L}_s^+)=\langle Z(\wt{L}_{\mathcal{O}_i^s}^+)\mid 1\leq i\leq t_s\rangle$. 

For $(3)$ note that there is some $\ol{\mu}\in \Irr( (\wt{L}_{\mathcal{I}_{-1},\mathcal{I}})_\lambda/{L}_{\mathcal{I}_{-1},\mathcal{I}})$ such that ${\rm Res}_{\wt{L}_{\mathcal{I}_{-1},\mathcal{I}}}^{\wt{L}_{\mathcal{I}_{-1},\mathcal{I}}^+}(\wt{\lambda}^+)=\mu\wt{\lambda}$, for $\mu$ the inflation of $\ol{\mu}$.
As $(\wt{L}_{\mathcal{I}_{-1},\mathcal{I}}^+)_\lambda$ is normalised by $V_d^\mathcal{I}$ it follows that $(V_d^\mathcal{I})_{\wt{\lambda}^+}\leq (V_d^\mathcal{I})_{\wt{\lambda}\mu}$.
Moreover $[\wt{L}_{\mathcal{I}_{-1},\mathcal{I}},N_{\mathcal{I}_{-1},\mathcal{I}}]\leq L_{\mathcal{I}_{-1},\mathcal{I}}$, implies that $(V_d^\mathcal{I})_{\wt{\lambda}\mu}=(V_d^\mathcal{I})_{\wt{\lambda}}$.
\end{proof}

\subsection{An extension map for Theorem~\ref{thm41ii}}
\indent

Assume the pair $(\mathcal{I}_{-1},\mathcal{I})$ satisfies Remark~\ref{PartLabLevi} with respect to $w=\rho(v)\in\sgnSymm_n$ from Section~\ref{vwSylTwist}.
Let $\bL_{\mathcal{I}_{-1},\mathcal{I}}$ denote the corresponding $d$-split Levi subgroup in $(\bG,vF)$ containing the maximal torus $\bT$, as described in Section~\ref{StrdSpLev}. Additionally take $v_\mathcal{I}$ as defined in Section~\ref{vI}.

For $E=\GenGp{\wh{F}_p}$ as defined in Section~\ref{endo}, recall that Theorem~\ref{MainThmReq}(ii) requires the existence of an extension map with respect to $\bL_{\mathcal{I}_{-1},\mathcal{I}}^{vF}\lhd N_{\bG^{vF}}(\bL_{\mathcal{I}_{-1},\mathcal{I}})E$, such that each extension contains $v\wh{F}$ in its kernel.
By Lemma~\ref{TwistTwist} it suffices to prove the existence of an extension map with respect to $\bL_{\mathcal{I}_{-1},\mathcal{I}}^{v_\mathcal{I}F}\lhd \NNN_{\bG}(\bL_{\mathcal{I}_{-1},\mathcal{I}})^{v_\mathcal{I}F}E$, such that each extension contains $v_{\mathcal{I}}\wh{F}$ in its kernel

By Proposition~\ref{SuppLinN}, $N_{\mathcal{I}_{-1},\mathcal{I}}:=N_{\bG^{v_\mathcal{I}F}}(\bL_{\mathcal{I}_{-1},\mathcal{I}})=L_{\mathcal{I}_{-1},\mathcal{I}}V_d^\mathcal{I}$, where $V_d^\mathcal{I}$ was constructed in the proof of Lemma~\ref{ExtWeylGpExtMap}.
The required extension map will be constructed by applying \cite[Proposition 4.2]{BrSpAM}.
This breaks the process up into the existence of an extension map with respect to $H_d^\mathcal{I}=L_{\mathcal{I}_{-1},\mathcal{I}}\cap V_d^\mathcal{I}\lhd V_d^\mathcal{I}E$  and another extension map with respect to  $L_{\mathcal{I}_{-1},\mathcal{I}}\lhd L_{\mathcal{I}_{-1},\mathcal{I}}\rtimes (V_d^{\mathcal{I}}E)/H_d^{\mathcal{I}}$.

\begin{lm}\label{ExtHdVdE}
There is an extension map with respect to $H_d^\mathcal{I}\lhd V_d^{\mathcal{I}}E$ with $v_{\mathcal{I}}\wh{F}$ in the kernel of each extension.  
\end{lm}
\begin{proof}
As $E=\GenGp{\wh{F}_p}$ which acts as a field automorphism, it follows that $V_d^{\mathcal{I}}E=V_d^{\mathcal{I}}\times E$.
Moreover, the exponent of $E={\rm Ord}(\wh{F}_p)$ is given by $|V_{{C_n}}|m$.

Let $\theta\in\Irr(H_d^{\mathcal{I}})$.
By Lemma~\ref{ExtWeylGpExtMap} there is an extension  $\tilde{\theta}\in (V_d^{\mathcal{I}})_{\theta}$ of $\theta$.
There is a character $\lambda_{\tilde{\theta}}\in\Irr (E)$ such that $\lambda_{\tilde{\theta}}(\wh{F})=\tilde{\theta}(v_{\mathcal{I}})^{-1}$.
Thus $\tilde{\theta}\times \lambda_{\tilde{\theta}}\in \Irr( (V_d^{\mathcal{I}}\times E)_{\theta})$ with $v_{\mathcal{I}}\wh{F}$ in its kernel.
\end{proof}

\begin{lm}\label{ExtLIVdE}
There is an extension map with respect to $L_{\mathcal{I}_{-1},\mathcal{I}}\lhd L_{\mathcal{I}_{-1},\mathcal{I}}\rtimes (V_d^{\mathcal{I}}E)/H_d^{\mathcal{I}}$ such that $\rho(v_{\mathcal{I}})\wh{F}$ lies in the kernel of each extension.
\end{lm}
\begin{proof}
Recall that $L_{\mathcal{I}_{-1},\mathcal{I}}=L_{\mathcal{I}_{-1}}\times\prod_{\mathcal{O}}L_{\mathcal{O}}$ where the product runs over all $\ol{w}$-orbits on $\mathcal{I}$ and $E=\GenGp{\wh{F}_p}$ where $\wh{F}_p$ acts as a field automorphism.
Given a  $\ol{w}$-orbit $\mathcal{O}\subseteq \mathcal{I}$, the group $L_{\mathcal{O}}$ is $\wh{F}_p$-stable and thus set $E_{\mathcal{O}}:=\GenGp{\wh{F}_{p,\mathcal{O}}}$ where $\wh{F}_{p,\mathcal{O}}$ denotes the automorphism of $L_{\mathcal{I}_{-1},\mathcal{I}}$ which acts on the factor $L_\mathcal{O}$ as $\wh{F}_p$ and as the trivial map on the other factors.
Similarly, define $E_{\mathcal{I}_{-1}}$ with respect to $L_{\mathcal{I}_{-1}}$.
This yields an embedding $$E\hookrightarrow E_{\mathcal{I}_{-1},\mathcal{I}}:= E_{\mathcal{I}_{-1}}\times \prod_{\mathcal{O}}E_{\mathcal{O}}$$
with the restriction of $\wh{F}_p$ to ${L_{\mathcal{I}_{-1},\mathcal{I}}}$ given by $\wh{F}_p\mid_{L_{\mathcal{I}_{-1},\mathcal{I}}}=\wh{F}_{p,\mathcal{I}_{-1}}\times \prod_{\mathcal{O}}\wh{F}_{p,\mathcal{O}}$.
By construction both $W_d^\mathcal{I}\cong V_d^\mathcal{I}/H_d^\mathcal{I}$ and $E_{\mathcal{I}_{-1},\mathcal{I}}$ act as automorphisms on $L_{\mathcal{I}_{-1},\mathcal{I}}$, however the actions of $E_{\mathcal{I}_{-1},\mathcal{I}}$ and $W_d^\mathcal{I}$ need not commute.

Recall from Section~\ref{CentwIinCentwulI} that $W_d^\mathcal{I}:=\prod_{s=1}^n W_d^{\mathcal{I}_s}$ and for $Q_1^s,\dots, Q^s_{t_s}$ representing the $\ol{w}$-orbits in $\mathcal{I}_s$, as explained in Section~\ref{wPartNot}, $$W_d^{\mathcal{I}_s}=\left( \prod_{i=1}^{t_s}\langle w'_{Q_i^s}\rangle \right)\rtimes \langle \tau_{Q_i^s}\mid 1\leq i\leq t_s-1\rangle.$$  
For notation purposes, denote by $ \wh{F}_{p,Q_i^s}:=\wh{F}_{p,\mathcal{O}_i^s}$ and $E_{Q_i^s}=E_{\mathcal{O}_i^s}$.

Let $x\in  W_d^\mathcal{I}$ and take a $\ol{w}$-orbit $\mathcal{O}\subseteq \mathcal{I}$.
Composition yields an automorphism $x\wh{F}_{p,\mathcal{O}}x^{-1}$ of ${L_{\mathcal{I}_{-1},\mathcal{I}}}$ which coincides with $\wh{F}_{p,\mathcal{O}'}$ for $\mathcal{O}':=x(\mathcal{O})$, another $\ol{w}$-orbit in $\mathcal{I}$, as ${}^{x^{-1}}L_{\mathcal{O}'}=L_\mathcal{O}$.
Thus $W_d^\mathcal{I}E_{\mathcal{I}_{-1},\mathcal{I}}$ is a well defined group and
$$
W_d^\mathcal{I}E_{\mathcal{I}_{-1},\mathcal{I}}= E_{\mathcal{I}_{-1}}\times \prod_{s=1}^n\left(\left( \prod_{i=1}^{t_s} \langle w'_{Q_i^s}, \wh{F}_{p,Q_i^s}\rangle\right) \rtimes \langle \tau_{Q_i^s}\mid 1\leq i\leq t_s-1\rangle\right).
$$
In particular, as shown above, the permutation $\tau_{Q_k^s}$ permutes $\wh{F}_{p,Q_k^s}$ and $\wh{F}_{p,Q_{k+1}^s}$, while fixing every other $\wh{F}_{p,Q_j^s}$.
Hence
 $$\left( \prod_{i=1}^{t_s} \langle w'_{Q_i^s}, \wh{F}_{p,Q_i^s}\rangle\right) \rtimes \langle \tau_{Q_i^s}\mid 1\leq i\leq t_s-1\rangle= \langle w'_{Q_1^s}, \wh{F}_{p,Q_1^s}\rangle\wr\mathfrak{S}_{t_s}$$
and 
$$w_\mathcal{I}\wh{F}=\wh{F}_{p, \mathcal{I}_{-1}}^m\times \prod_{s=1}^n \left( \prod_{i=1}^{t_s}w_{Q_i^s}\wh{F}_{p,Q_i^s}^m\right).$$

It therefore suffices to provide an extension map with respect to $$L_{\mathcal{I}_s}\lhd \left( L_{Q_1^s} \rtimes \langle w_{Q_1^s}',\wh{F}_{p,Q_1^s}\rangle/\langle w_{Q_1^s}\wh{F}_{p,Q_1^s}^m\rangle \right) \wr \mathfrak{S}_{t_s}.$$ 
From the explicit construction, $$\langle w_{Q_i^s}',\wh{F}_{p,Q_i^s}\rangle/\langle w_{Q_i^s}\wh{F}_{p,Q_i^s}^m\rangle$$ is a subgroup of ${\rm Aut}(L_{Q_i^s})$ given by its own field and graph automorphisms.
Hence applying Lemma~\ref{ExtGnSymm} yields the required extension map.
\end{proof}

\begin{prop}\label{NhatExtMap}
There exists an extension map $\Lambda$ with respect to $L_{\mathcal{I}_{-1},\mathcal{I}}\lhd N_{\mathcal{I}_{-1},\mathcal{I}}$ such that 
\begin{enumerate}
\item $\Lambda$ is $\wh{N}$-equivariant.
\item Every character $\lambda$ has an extension $\wh{\lambda}\in\Irr(\wh{N}_{\lambda})$ with $\Res_{\NNN_{\lambda}}^{\wh{N}_{\lambda}}(\wh{\lambda})=\Lambda(\lambda)$ and $v_{\mathcal{I}}\wh{F}\in \ker(\wh{\lambda})$.
\end{enumerate}
\begin{proof}
As in \cite[Theorem 4.3]{BrSpAM} it suffices to define $\Lambda$ on a $\wh{N}$-transversal in $\Irr(L)$.
Moreover, as $H_d^\mathcal{I}=L_{\mathcal{I}_{-1},\mathcal{I}}\cap V_d^\mathcal{I}\leq Z(L_{\mathcal{I}_{-1},\mathcal{I}})$, applying \cite[Proposition 4.2]{BrSpAM} with the extension maps given in Lemma~\ref{ExtHdVdE} and Lemma~\ref{ExtLIVdE} proves the result.
\end{proof}
\end{prop}

\section{The Inertial subgroups arising from the relative Weyl group}\label{CliffTh}

Take $v$ as defined in Section~\ref{vwSylTwist} so that $vF$ forms a Sylow $d$-twist of $\bG\cong {\rm Sp}_{2n}(\overline{\mathbb{F}}_p)$ and let $\bL$ a $d$-split Levi subgroup with respect to $(\bG,vF)$.
The aim of this subsection is to prove condition~\ref{thm41iii}, with respect to $\bL$ and $vF$.

From Section~\ref{StrdSpLev}, it can be assumed that $\bL=\bL_{\mathcal{I}_{-1},\mathcal{I}}$ where  $(\mathcal{I}_{-1},\mathcal{I})$ satisfies Remark~\ref{PartLabLevi} with respect to $w=\rho(v)\in\sgnSymm_n$. 
Moreover, for $v_\mathcal{I}$ as defined in Section~\ref{vI}, Lemma~\ref{TwistTwist} shows that it suffices to prove Theorem~\ref{MainThmReq} with respect to $\bL_{\mathcal{I}_{-1},\mathcal{I}}$ and $v_\mathcal{I}F$.

First some setup that will be used throughout this section is given.
For the maximal torus $\wt{\bT}$ from Section~\ref{endo}, the corresponding $v_\mathcal{I}F$-stable $d$-split Levi subgroup in $\wt{\bG}$ is given by $\wt{\bL}_{\mathcal{I}_{-1},\mathcal{I}}:=\wt{\bT}\bL_{\mathcal{I}_{-1},\mathcal{I}}$.
Set $L_{\mathcal{I}_{-1},\mathcal{I}}:=\bL_{\mathcal{I}_{-1},\mathcal{I}}^{v_\mathcal{I}F}$ and $\wt{L}_{\mathcal{I}_{-1},\mathcal{I}}:=\wt{\bL}_{\mathcal{I}_{-1},\mathcal{I}}^{v_\mathcal{I}F}$.
Let $N_{\mathcal{I}_{-1},\mathcal{I}}:=\NNN_{\bG}(\bL_{\mathcal{I}_{-1},\mathcal{I}})^{v_\mathcal{I}F}$, while for $E=\GenGp{\wh{F}_p}$ from Section~\ref{endo}, set $\wh{N}_{\mathcal{I}_{-1},\mathcal{I}}:=N_{\mathcal{I}_{-1},\mathcal{I}}E$.
According to Proposition~\ref{SuppLinN}, the group $N_{\mathcal{I}_{-1},\mathcal{I}}=L_{\mathcal{I}_{-1},\mathcal{I}}V_d^\mathcal{I}$ and combining Proposition~\ref{RelWeylGpPart} with the construction of $W_d^\mathcal{I}$ in Section~\ref{CentwIinCentwulI} yields $\rho(V_d^\mathcal{I})=W_d^\mathcal{I}\cong N_{\mathcal{I}_{-1},\mathcal{I}}/L_{\mathcal{I}_{-1},\mathcal{I}}$.

For $\lambda\in \Irr(L_{\mathcal{I}_{-1},\mathcal{I}})$ and $\wt{\lambda}\in \Irr(\wt{L}_{\mathcal{I}_{-1},\mathcal{I}}\mid \lambda)$ condition~\ref{thm41iii} is focused on the action of $\wh{N}_{\mathcal{I}_{-1},\mathcal{I}}/L_{\mathcal{I}_{-1},\mathcal{I}}$ on the irreducible characters of $(N_{\mathcal{I}_{-1},\mathcal{I}})_\lambda/L_{\mathcal{I}_{-1},\mathcal{I}}$ and $(N_{\mathcal{I}_{-1},\mathcal{I}})_{\wt{\lambda}}/L_{\mathcal{I}_{-1},\mathcal{I}}$.
As $(N_{\mathcal{I}_{-1},\mathcal{I}})_{{\lambda}}/(N_{\mathcal{I}_{-1},\mathcal{I}})_{\wt{\lambda}}$ and $E$ are cyclic groups, condition~\ref{thm41iii} parts $(1)$ and $(3)$ are automatically satisfied.
As $[E,V_d^\mathcal{I}]=1$, it is also true that $[E,W_d^\mathcal{I}]=1$.
Define $K(\lambda):=\NNN_{W_d^\mathcal{I}}((W_d^\mathcal{I})_\lambda,(W_d^\mathcal{I})_{\wt{\lambda}})$.
Then condition~\ref{thm41iii} reduces to verifying:

\begin{aim}
 
For every $\xi_0\in\Irr (W_{\wt{\lambda}})$ there exists  some $\xi\in \Irr (W_{\lambda}\mid \xi_0)$ which is ${K}(\lambda)_{\xi_0}$-invariant.
\end{aim}

Before studying the action of $W_d^\mathcal{I}$ on the finite groups $L_{\mathcal{I}_{-1},\mathcal{I}}$ and $\wt{L}_{\mathcal{I}_{-1},\mathcal{I}}$, the action on the underlying algebraic group $\bL_{\mathcal{I}_{-1},\mathcal{I}}$ is considered in the following remark.

\begin{rem}\label{WonBL}
As outlined in the start of Section~\ref{StrdSpLev}, $\bL_{\mathcal{I}_{-1},\mathcal{I}}=\bL_{\mathcal{I}_{-1}}\times \prod_{s=1}^n\bL_s$, where $\bL_s$ is the direct product of those subgroups $\bL_\mathcal{J}$ with $\mathcal{J}\in\mathcal{I}_s$.
According to the proof of Lemma~\ref{ExtWeylGpExtMap}, the group $W_d^\mathcal{I}=\prod_{s=1}^n W_d^{\mathcal{I}_s}$ with $\rho(V_d^{\mathcal{I}_s})=W_d^{\mathcal{I}_s}$.
Moreover, $[V_d^{\mathcal{I}_s},\bL_{s'}]=1$ whenever $s\ne s'$.
As in Section~\ref{wPartNot}, for $1\leq s\leq n$, let $\mathcal{O}_1^s,\dots,\mathcal{O}_{t_s}^s$ denote the $\ol{w}$-orbits in $\mathcal{I}_s$ with $Q_i^s\subset \mathcal{J}^{d_0,a}$ corresponding to $\mathcal{O}_i^s$. 
Using the notation in the proof of Lemma~\ref{ExtWeylGpExtMap} and Section~\ref{CentwIinCentwulI},
$$\rho(V_d^{\mathcal{I}_s})=W_d^{\mathcal{I}_s}=\left( \prod_{i=1}^{t_s}\langle w'_{Q_i^s}\rangle \right)\rtimes \langle \tau_{Q_k^s}\mid 1\leq k\leq t_s-1\rangle,$$
where $\rho(\kappa_s(c_{i}))=w'_{Q_i^s}$ for $1\leq i\leq t_s$, and $\tau_{Q_k^s}:=\rho(\kappa_s(p_{k}))=\prod_{j=1}^s \tau_{Q_{k,j}^s,Q_{k+1,j}^s}$ for $1\leq k\leq t_s-1$.
Moreover $[\kappa_s(c_{i}),\bL_\mathcal{J}]=1$ whenever $\mathcal{J}\not\in \mathcal{O}_i^s$; while $[\kappa_s(p_{k}),\bL_\mathcal{J}]=1$ whenever $\mathcal{J}\not\in \mathcal{O}_k^s\sqcup \mathcal{O}_{k+1}^s$.
\end{rem}

\begin{notation}\label{IsomWs}
For each $1\leq s\leq n$ let $\mathcal{G}_s=\langle g_s\rangle$ denote the cyclic group of order $2d_0$.
Denote by $\alpha_s$ the isomorphism  
$$
\begin{array}{cccc}
\alpha_s: &  \mathcal{G}_s\wr \Symm_{t_s} & \rightarrow & W_d^{\mathcal{I}_s}.\\
& (g_s^{j_1},\dots, g_s^{j_{t_s}}) & \mapsto & \prod_{i=1}^{t_s} (w'_{Q_i^s})^{j_i}\\
& (i,i+1) & \mapsto & \tau_{Q_i^s}\\
\end{array}
$$
This extends to an isomorphism $\alpha$ from $\prod_{s=1}^n (\mathcal{G}_s\rtimes \Symm_{t_s})$ to $W_d^\mathcal{I}$.

\end{notation}

\subsection{The action of $W_d^\mathcal{I}$ on $L_{\mathcal{I}_{-1},\mathcal{I}}$}\label{WonL}
\indent

Using the setup from Section~\ref{ChoiceOfn}, for $1\leq s\leq n$ set $L_s$ to be the product of those $L_\mathcal{O}:=\bL_{\mathcal{O}}^{v_\mathcal{I}F}=(\prod_{\mathcal{J}\in\mathcal{O}}\bL_\mathcal{J})^{v_\mathcal{I}F}$ with $\mathcal{O}$ a $\ol{w}$-orbit in $\mathcal{I}$ such that $\mathcal{J}_\mathcal{O}=\ul{\mathcal{O}}\cap \ul{a}$ contains $s$ elements.
Then $L_{\mathcal{I}_{-1},\mathcal{I}}=L_{\mathcal{I}_{-1}}\times \prod_{s=1}^n L_s$ with $L_s\cong {\rm GL}_s(\epsilon_d q^{d_0})^{t_s}$, for $\epsilon_d:=(-1)^{d+1}$. 

\begin{notation} \label{GORO}
Let $\mathcal{O}$ be a $\ol{w}$-orbit in $\mathcal{I}$.
If $\ul{\mathcal{O}}=\ul{Q}_i^s$, for a subset ${Q}_i^s\subset \mathcal{J}^{d_0,a}$ as in Section~\ref{wPartNot}, set $\mathcal{G}_\mathcal{O}:=\langle w'_{{Q}_i^s}\rangle$ and fix $\mathcal{R}_\mathcal{O}$ a $\mathcal{G}_\mathcal{O}$-transversal on $\Irr(L_{\mathcal{O}})$. 
Note that by construction $(w'_{Q_i^s})^{\tau_{Q_i^s}}=w'_{Q_{i+1}^s}$ and thus it can be assumed that $\mathcal{R}_{\mathcal{O}_i^s}^{\tau_{Q_i^s}}=\mathcal{R}_{\mathcal{O}_{i+1}^s}$.
\end{notation}

\begin{rem}\label{TranRs}
By Remark~\ref{WonBL}, whenever $\mathcal{O}$ and $\mathcal{O}'$ are two distinct $\ol{w}$-orbits in $\mathcal{I}$, it follows that $[\mathcal{G}_{\mathcal{O}'},\bL_\mathcal{O}]=1$.
Moreover $[\tau_{Q_k^s},\bL_\mathcal{O}]=1=[\tau_{Q_k^s},w'_\mathcal{O}]$ unless $\mathcal{O}\in\{\mathcal{O}_k^s,\mathcal{O}_{k+1}^s\}$ in which case $L_{\mathcal{O}_i^s}^{\tau_{Q_i^s}}=L_{\mathcal{O}_{i+1}^s}$ and $\mathcal{G}_{\mathcal{O}_i^s}^{\tau_{Q_i^s}}=\mathcal{G}_{\mathcal{O}_{i+1}^s}$.
Thus $\alpha_s$ from Notation~\ref{IsomWs} can be extended to an isomorphism 
$$
\alpha_s: \left( \GL_s(\epsilon_dq^{d_0})\rtimes \mathcal{G}_s \right)\wr\Symm_{t_s} \rightarrow  L_s\rtimes W_d^{\mathcal{I}_s}.
$$
Moreover, under the isomorphism $\alpha_s$, each $\mathcal{G}_{\mathcal{O}_i^s}$-transversal $\mathcal{R}_{Q_i^s}$ corresponds to a fixed $\mathcal{G}_s$-transversal $\mathcal{R}_s\subset\Irr(\GL_s(\epsilon_dq^{d_0}))$.
\end{rem}

\begin{notation}\label{Jlams}
Let $1\leq s\leq n$, $\mathcal{O}$ a $\ol{w}$-orbit in $\mathcal{I}_s$ and $\zeta\in \mathcal{R}_s$ the $\mathcal{G}_s$-transversal fixed in Remark~\ref{TranRs}.
\begin{enumerate}
\item Denote by $\lambda_{\mathcal{O},\zeta}\in \mathcal{R}_\mathcal{O}$ the character such that $\lambda_{\mathcal{O},\zeta}\circ \alpha_s=\zeta$ and set $\mathcal{G}_{\mathcal{O},\zeta}:=(\mathcal{G}_{\mathcal{O}})_{\lambda_{\mathcal{O},\zeta}}$.
\item For $\lambda\in \Irr (L_{\mathcal{I}_{-1},\mathcal{I}})$ set $\mathcal{J}_{\lambda,s}(\zeta):=\{1\leq i\leq t_s\mid \lambda_{\mathcal{O}_i^s}\in {\rm Orb}_{\mathcal{G}_\mathcal{O}}(\lambda_{\mathcal{O},\zeta})\}$, where $\mathcal{O}_1^s,\dots, \mathcal{O}_{t_s}^s$ are the $\ol{w}$-orbits in $\mathcal{I}_s$ and $\lambda_{\mathcal{O}}:={\rm Res}^{L_{\mathcal{I}_{-1},\mathcal{I}}}_{L_{\mathcal{O}}}(\lambda_s)$.
\end{enumerate}
\end{notation}

\begin{lm}\label{IntWonL}
Let $\lambda\in \Irr(L_{\mathcal{I}_{-1},\mathcal{I}})$ and assume that for each $\ol{w}$-orbit $\mathcal{O}\subseteq\mathcal{I}$ the character $\lambda_{\mathcal{O}}:={\rm Res}^{L_{\mathcal{I}_{-1},\mathcal{I}}}_{L_{\mathcal{O}}}(\lambda_s)\in \mathcal{R}_{\mathcal{O}}$.
Then $(W_d^\mathcal{I})_\lambda=\prod_{s=1}^n\left(\prod_{\zeta\in\mathcal{R}_s} (W_d^{\mathcal{I}_s})_{\lambda_s,\zeta}\right)$ where, for $\mathcal{J}_{\lambda,s}(\zeta)$ defined in Notation~\ref{Jlams},
 $$(W_d^{\mathcal{I}_s})_{\lambda_s,\zeta}:=\left( \prod_{i\in \mathcal{J}_{\lambda,s}(\zeta)}\mathcal{G}_{\mathcal{O}_i^s,\zeta}\right)\rtimes \alpha_s(\Symm(\mathcal{J}_{\lambda,s}(\zeta))).$$
\end{lm}
\begin{proof}
Set $\lambda_s:={\rm Res}^{L_{\mathcal{I}_{-1},\mathcal{I}}}_{L_s}(\lambda)$, so that by Remark~\ref{WonBL} $(W_d^\mathcal{I})_\lambda=\prod_{s=1}^n(W_d^{\mathcal{I}_s})_{\lambda_s}$.
For a $\ol{w}$-orbit $\mathcal{O}\subset\mathcal{I}_s$, the character $\lambda_\mathcal{O}=\lambda_{\mathcal{O},\zeta}$ for some $\zeta\in \mathcal{R}_s$ and as in the proof of Lemma~\ref{ExtensionwithWreath},
$$
\alpha_s^{-1}\left( (W_d^{\mathcal{I}_s})_{\lambda_s}\right)=\left( \mathcal{G}_s\wr \Symm_{t_s}\right)_{\lambda_s\circ\alpha_s}=\prod_{\zeta\in\mathcal{R}_s}  (\mathcal{G}_s)_\zeta\wr \mathfrak{S}(\mathcal{J}_{\lambda,s}(\zeta)). 
$$
\end{proof}

\subsection{The action of $W_d^\mathcal{I}$ on $\wt{L}_{\mathcal{I}_{-1},\mathcal{I}}^+$}\label{WtilLCirc}
\indent

Given $\lambda\in\Irr(L_{\mathcal{I}_{-1},\mathcal{I}})$ and $\wt\lambda\in \Irr((\wt{L}_{\mathcal{I}_{-1},\mathcal{I}})_\lambda\mid \lambda)$, to study the structure of  $W_{\wt\lambda}$ it will be again convenient to consider $\wt{L}_{\mathcal{I}_{-1},\mathcal{I}}^+$ defined in Section~\ref{TiL0}. 
In particular, this section will require the use of characters arising in a central product hence the following remark is used to label the irreducible characters of such groups.

\begin{rem}\label{IrrCentProd}
Let $X,A\leq Y$ be finite groups with $[X,A]=1$ and $A$ abelian.
Set $Z=X\cap A$.
As $\GenGp{X,A}\cong X\circ_Z A:=X\times A/Z'$, where $Z':=\{(z,z^{-1})\mid z\in Z\}$, the properties of central products (see for example \cite[Section 5]{IMNRedMcKay}) show that 
$$\Irr(X\circ_Z A)=\{ \theta\cdot\mu \mid \mu\in \Irr(A) \text{ with } \lambda\in \Irr(X\mid {\rm Res}_Z^A(\mu))\}.$$ 
with $\theta\cdot\mu (xa):=\theta(x)\mu(a)$ for all $x\in X$ and $a\in A$.
\end{rem}

In Section~\ref{WonL}, the isomorphism $W_d^{\mathcal{I}_s}\cong C_{2d_0}\wr \Symm_{t_s}$ given by $\alpha_s$ provided a helpful way to describe $(W_d^\mathcal{I})_\lambda$.
The following remark sets up the analogous notation in the context of $\wt{L}_{\mathcal{I}_{-1},\mathcal{I}}^+$.

\begin{rem}\label{StructTilLPlusW}
Recall that for a fixed $x\in Z(\wt{\bG})$, of order $2(q-\epsilon_d)$ with $\epsilon_d=(-1)^{d+1}$, then
$$\wt{L}_{\mathcal{I}_{-1},\mathcal{I}}^+:=\left(\wt{L}_{\mathcal{I}_{-1}}^+ \times \prod_{\mathcal{O}} \wt{L}_{\mathcal{O}}^+\right)\circ \langle x\rangle,$$
where the product runs over all $\ol{w}$-orbits on $\mathcal{I}$.
Furthermore, for each $\ol{w}$-orbit $\mathcal{O}\subset\mathcal{I}$, there was an element $z_\mathcal{O}^+\in Z(\bL_\mathcal{O})$, such that for $Z_\mathcal{O}^+:=\GenGp{z_\mathcal{O}^+}$, the group $\wt{L}_{\mathcal{O}}^+$ arose as the central product $\wt{L}_{\mathcal{O}}^+=\GenGp{{L}_{\mathcal{O}},z_{\mathcal{O}}^+}\cong Z_{\mathcal{O}}^+\circ_{\langle (z_{\mathcal{O}}^+)^2\rangle}L_\mathcal{O}$.
In particular, each quotient $\wt{L}_{\mathcal{O}}^+/{L}_{\mathcal{O}}$ is cyclic and so each character of ${L}_{\mathcal{O}}$ extends to $\wt{L}_{\mathcal{O}}^+$.
Thus for $\mathcal{O}$ a $\ol{w}$-orbit in $\mathcal{I}$ and $\theta\in \mathcal{R}_\mathcal{O}$ fix an extension $\wt{\theta}\in\Irr (\wt{L}_{\mathcal{O}}^+\mid \zeta)$ and set 
 $$\wt{\mathcal{R}}_\mathcal{O}:=\{ \wt{\theta}\mid \theta\in \mathcal{R}_\mathcal{O}\}.$$
The action of $V_d^\mathcal{I}$ on the elements $z_{\mathcal{O}}^+$ was described in Remark~\ref{VactZtilL}.
In particular, for each $\ol{w}$-orbit $\mathcal{O}\subset\mathcal{I}$, the group $\mathcal{G}_\mathcal{O}$ normalises $\wt{L}_\mathcal{O}^+$.
Thus, as with $\mathcal{R}_\mathcal{O}$, it  can be assumed that $\wt{\mathcal{R}}_{\mathcal{O}_i^s}^{\tau_{Q_i^s}}=\wt{\mathcal{R}}_{\mathcal{O}_{i+1}^s}$.
Moreover  if ${\rm GL}_s(\epsilon_d q^{d_0})^+$ denotes the isomorphism type of each $\wt{L}_{\mathcal{O}_i^s}^+$, then as in Remark~\ref{TranRs}, $\alpha_s$ can be extended to an isomorphism 
$$ \alpha_s: \left( \GL_s(\epsilon_dq^{d_0})^+\rtimes \mathcal{G}_s \right)\wr\Symm_{t_s}\rightarrow \wt{L}_s^+\rtimes W_d^{\mathcal{I}_s} .
$$
Hence each $\wt{\mathcal{R}}_{\mathcal{O}_i^s}$ corresponds to a fixed subset $\wt{\mathcal{R}}_{s}\subset \Irr ( {\rm GL}_s(\epsilon_d q^{d_0})^+)$.
Therefore, as in Notation~\ref{Jlams}, for $\wt{\zeta}\in \wt{\mathcal{R}}_s$ extending $\zeta\in \mathcal{R}_s$, denote by $\wt{\lambda}_{\mathcal{O},\zeta}\in \wt{\mathcal{R}}_{\mathcal{O}}$ the extension of $\lambda_{\mathcal{O},\zeta}$ with $\wt{\lambda}_{\mathcal{O},\zeta}\circ\alpha_s=\wt{\zeta}$ and set $\mathcal{G}_{\mathcal{O},\wt{\zeta}}:=(\mathcal{G}_{\mathcal{O}})_{\wt{\lambda}_{\mathcal{O},\wt{\zeta}}}$.

\end{rem}

\begin{lm}\label{IntWonTilL0}
Let $\lambda\in \Irr(L_{\mathcal{I}_{-1},\mathcal{I}})$, $\wt{\lambda}^+\in\Irr ((\wt{L}_{\mathcal{I}_{-1},\mathcal{I}}^+)_\lambda)$  and assume that for each $\ol{w}$-orbit $\mathcal{O}\subseteq\mathcal{I}$ the character $\wt{\lambda}_{\mathcal{O}}^+:={\rm Res}_{\wt{L}_\mathcal{O}^+}^{\wt{L}_{\mathcal{I}_{-1},\mathcal{I}}^+}(\wt{\lambda}^+)\in \wt{\mathcal{R}}_\mathcal{O}$.
Then $(W_d^\mathcal{I})_{\wt{\lambda}^+}=\prod_{s=1}^n\left( \prod_{\zeta\in\mathcal{R}_s} (W_d^{\mathcal{I}_s})_{\wt{\lambda}_s^+,\zeta} \right)$ where, for $\mathcal{J}_{\lambda,s}(\zeta)$ defined in Notation~\ref{Jlams} and the notation given in Remark~\ref{StructTilLPlusW},
 $$(W_d^{\mathcal{I}_s})_{\wt{\lambda}_s^+,\zeta}:= \left( \prod_{i\in \mathcal{J}_{\lambda,s}(\zeta)}\mathcal{G}_{\mathcal{O}_i^s,\wt{\zeta}}\right)\rtimes \alpha_s(\Symm(\mathcal{J}_{\lambda,s}(\zeta))).$$
\end{lm}
\begin{proof}
According to Lemma~\ref{InTilL},
$$(\wt{L}_{\mathcal{I}_{-1},\mathcal{I}}^+)_\lambda=\left((\wt{L}_{\mathcal{I}_{-1}}^+)_{\lambda_{-1}} \times \prod_{s=1}^n \wt{L}_{s}^+\right)\circ \langle x\rangle,$$
which is a central product.
For $1\leq s\leq n$ set $\wt{\lambda}_s:={\rm Res}^{\wt{L}^+_{\mathcal{I}_{-1},\mathcal{I}}}_{\wt{L}_s^+}(\wt{\lambda}^+)$.
Then by Remark~\ref{IrrCentProd} the character $\wt{\lambda}^+=(\wt{\lambda}_{-1}^+\times\prod_{s=1}^n\wt{\lambda}_s^+)\cdot\mu$ for some $\wt{\lambda}_{-1}^+\in \Irr\left((\wt{L}_{\mathcal{I}_{-1}}^+)_{\lambda_{-1}}\right)$ and $\mu\in\Irr(\langle x\rangle)$.
In particular, by applying Remark~\ref{StructTilLPlusW}, $(W_d^{\mathcal{I}})_{\wt{\lambda}^+}=\prod_{s=1}^n(W_d^{\mathcal{I}_s})_{\wt{\lambda}_{s}^+}$.
The same argument as in Lemma~\ref{IntWonL} proves the result. 
\end{proof}

\subsubsection{The inertia subgroups $(\mathcal{G}_{s})_\zeta$ and $(\mathcal{G}_s)_{\wt{\zeta}}$}
\indent

According to Lemma~\ref{IntWonL} and Lemma~\ref{IntWonTilL0} the difference between the inertia groups $(W_d^{\mathcal{I}_s})_{\lambda}$ and $(W_d^{\mathcal{I}})_{\wt{\lambda}^+}$ is determined by those $\zeta\in \mathcal{R}_s$ such that 
$(\mathcal{G}_s)_\zeta\ne (\mathcal{G}_{s})_{\wt{\zeta}}$ for $\wt{\zeta}\in \wt{\mathcal{R}}_s$ the fixed extension of $\zeta$.
An explicit description is possible through the associated central character.

For $\mathcal{O}$ a $\ol{w}$-orbit in $\mathcal{I}$ set $\wt{Z}_\mathcal{O}^+:=Z(\wt{L}^+_{\mathcal{O}})$, which is stable under $\mathcal{G}_{\mathcal{O}}$.
In particular, as $Z_\mathcal{O}=Z(L_\mathcal{O})$ is of index two in $\wt{Z}_\mathcal{O}^+$, it follows that $[\mathcal{G}_\mathcal{O},\wt{Z}_\mathcal{O}^+]\leq {Z}_\mathcal{O}$.
Moreover, as with $\wt{L}_\mathcal{O}^+$, the group $\wt{Z}_\mathcal{O}^+\cong Z_\mathcal{O}\circ_{\GenGp{(z_\mathcal{O}^+)^2}} Z_{\mathcal{O}}^+$.

\begin{lm}\label{GStabCharZ}
Let $\mathcal{O}$ be a $\ol{w}$-orbit in $\mathcal{I}$, $\eta\in \Irr(Z_\mathcal{O})$ and $\wt{\eta}\in \Irr(Z(\wt{L}_{\mathcal{O}}^+)\mid \eta)$.
Then $(\mathcal{G}_{\mathcal{O}})_\eta\ne (\mathcal{G}_{\mathcal{O}})_{\wt\eta}$ if and only if $o(\eta)= 2$.
\end{lm}
\begin{proof}
It follows that $|\mathcal{G}_{\mathcal{O},\eta}:\mathcal{G}_{\mathcal{O},\wt{\eta}}|$ divides $|\wt{Z}_\mathcal{O}^+:Z_\mathcal{O}|=2$.
Therefore if $(\mathcal{G}_{\mathcal{O}})_\eta\ne(\mathcal{G}_{\mathcal{O}})_{\wt\eta}$, then the involution of $\mathcal{G}_\mathcal{O}$ must fix $\eta$.
However the involution is given by $\prod_{i\in\ul{\mathcal{O}}}(i,-i)$ and acts by inversion on $Z_\mathcal{O}$.
Hence $2\mid \mathcal{G}_{\mathcal{O},\eta}$ if and only if $\eta(g)=\eta(g)^{-1}$ for all $g\in Z_\mathcal{O}$ which is equivalent  to $o(\eta)\mid 2$.
Thus $(\mathcal{G}_{\mathcal{O}})_\eta\ne (\mathcal{G}_{\mathcal{O}})_{\wt\eta}$ implies that $o(\eta)=2$.

Assume that $o(\eta)=2$ so that $\mathcal{G}_{\mathcal{O},\eta}=\mathcal{G}_\mathcal{O}$.
In Remark~\ref{VactZtilL} an element $x_\mathcal{O}\in V_d^\mathcal{I}$ was fixed such that $\mathcal{G}_\mathcal{O}=\GenGp{\rho(x_\mathcal{O})}$ and the action of $x_\mathcal{O}^{-1}$ on $\wt{Z}_\mathcal{O}^+$ was explicitly described.
As $\wt{\eta}$ is a linear character and $[x_\mathcal{O},g]\in Z_\mathcal{O}$ for all $g\in \wt{Z}_\mathcal{O}^+$, it follows that $\mathcal{G}_{\mathcal{O},\wt{\eta}}=\mathcal{G}_\mathcal{O}$ if and only if $[x_\mathcal{O}^{-1},z_\mathcal{O}^+]\in {\rm ker}(\eta)$.
In the case $d_0$ is even $[x_\mathcal{O}^{-1},z_\mathcal{O}^+]=z_\mathcal{O}$.
The element $z_\mathcal{O}$ is the unique $2$-element in $Z_\mathcal{O}$, as $(q^{d_0}+1)_2=2$, and thus $\eta([x_\mathcal{O}^{-1},z_\mathcal{O}^+])=-1$.
In the case $d_0$ is odd $[x_\mathcal{O}^{-1},z_\mathcal{O}^+]=(z_\mathcal{O}^+)^m$ with $$
m:=\left\{
\begin{array}{lr}
    {q+q^{d_0}} & d \text{ even}\\
  {(\frac{d_0+1}{2})(q^{d_0}-1)-q^{\frac{d_0+1}{2}}-1} & d \text{ odd}\\
      \end{array}
      \right.
$$ 
Thus $m/2$ is odd and $\eta([x_\mathcal{O}^{-1},z_\mathcal{O}^+])=\eta((z_\mathcal{O}^+)^2)^{\frac{m}{2}}=-1$.
\end{proof}

\begin{lm}\label{StabGCharTilL}
Let $\mathcal{O}$ be a $\ol{w}$-orbit in $\mathcal{I}$, $\eta\in \Irr(Z_\mathcal{O})$, $\theta\in \Irr(L_\mathcal{O}\mid \eta)$ and $\wt{\theta}\in \Irr(\wt{L}_{\mathcal{O}}^+\mid \theta)$.
Then $(\mathcal{G}_{\mathcal{O}})_{\wt{\theta}}\ne (\mathcal{G}_{\mathcal{O}})_{\theta}$ if and only if $o(\eta)=2$ and $\mathcal{G}_{\mathcal{O},\theta}$ contains the Sylow $2$-subgroup of $\mathcal{G}_\mathcal{O}$.
 
\end{lm}
\begin{proof}
By Remark~\ref{IrrCentProd}, there is a $\mu\in \Irr(Z_{\mathcal{O}}^+)$ such that $\theta\in  \Irr\left(L_{\mathcal{O}}\mid {\rm Res}_{\langle (z_{\mathcal{O}}^+)^2\rangle}^{Z_{\mathcal{O}}^+}(\mu)\right)$ and $\wt{\theta}(gy)=\theta (g)\mu(y)$ for all $g\in L_\mathcal{O}$ and $y\in Z_\mathcal{O}^+$.
Moreover, $\wt{\theta}\in \Irr(\wt{L}_\mathcal{O}^+\mid \wt{\eta})$, where the character $\wt{\eta}:=\eta\cdot\mu\in \Irr(\wt{Z}_\mathcal{O}^+)$.

Let $\sigma\in \mathcal{G}_{\mathcal{O},\theta}\leq \mathcal{G}_{\mathcal{O},\eta}$.
As $[z_\mathcal{O}^+,\sigma^{-1}]\in Z_\mathcal{O}$, it follows for any $g\in L_\mathcal{O}$ that $$\wt{\theta}^\sigma(z_\mathcal{O}^+ g)=\wt{\theta} (z_\mathcal{O}^+[z_\mathcal{O}^+,\sigma^{-1}] g^{\sigma^{-1}})=\mu(z_\mathcal{O}^+)\theta([z_\mathcal{O}^+,\sigma^{-1}]g^{\sigma^{-1}})=\mu(z_\mathcal{O}^+)\eta([z_\mathcal{O}^+,\sigma^{-1}])\theta(g)$$
and similarly for any $x\in Z_\mathcal{O}$ 
$$\wt{\eta}^\sigma(z_\mathcal{O}^+ x)=\mu (z_\mathcal{O}^+)\eta([z_\mathcal{O}^+,\sigma^{-1}])\eta(x).$$
Hence $\sigma\in \mathcal{G}_{\mathcal{O},\theta}$ if and only if $[z_\mathcal{O}^+,\sigma^{-1}]\in {\rm ker}(\eta)$, which  is equivalent to $\sigma\in \mathcal{G}_{\mathcal{O},\wt{\eta}}$.
Thus $\mathcal{G}_{\mathcal{O},\wt{\theta}}=\mathcal{G}_{\mathcal{O},\theta}\cap \mathcal{G}_{\mathcal{O},\wt{\eta}}$.
Moreover, as $|\mathcal{G}_{\mathcal{O},\theta}:\mathcal{G}_{\mathcal{O},\wt{\theta}}|\leq 2$, it follows that $\mathcal{G}_{\mathcal{O},\wt{\theta}}\ne \mathcal{G}_{\mathcal{O},\theta}$ if and only if $\mathcal{G}_{\mathcal{O},\wt{\eta}}<\mathcal{G}_{\mathcal{O},\theta}\mathcal{G}_{\mathcal{O},\wt{\eta}}=\mathcal{G}_{\mathcal{O},{\eta}}$.

By Lemma~\ref{GStabCharZ} if $o(\eta)\ne 2$, then $\mathcal{G}_{\mathcal{O},\wt{\theta}}=\mathcal{G}_{\mathcal{O},\theta}$.
Furthermore in the case that $o(\eta)=2$, then $\mathcal{G}_{\mathcal{O},\wt{\eta}}>\mathcal{G}_{\mathcal{O},\eta}=\mathcal{G}_{\mathcal{O}}$.
That is $\mathcal{G}_{\mathcal{O},\wt{\theta}}\ne\mathcal{G}_{\mathcal{O},\theta}$ if and only if $\mathcal{G}_{\mathcal{O},\theta}$ contains the Sylow $2$-subgroup of $\mathcal{G}_{\mathcal{O}}$.
\end{proof}

\begin{rem} \label{Rs1}
For a positive integer $s$ set $$\mathcal{R}_{s}^{1}:=\{\zeta\in \mathcal{R}_{s}\mid (\mathcal{G}_s)_\zeta\ne (\mathcal{G}_s)_{\wt{\zeta}}\}$$ and $\mathcal{R}_{s}^{2}:=\mathcal{R}_s\setminus\mathcal{R}_{s}^{1}$.
By Lemma~\ref{IntWonL} and Lemma~\ref{IntWonTilL0}, the difference between $(W_d^{\mathcal{I}})_{\lambda}$ and $(W_d^{\mathcal{I}})_{\wt{\lambda}^+}$ depends on the factors of $\lambda$ which are $W_d^\mathcal{I}$ conjugate to characters in $\mathcal{R}_s^1$.
Moreover, by Lemma~\ref{StabGCharTilL}, $\zeta\in \mathcal{R}_{s}^1$, if and only $\zeta$ lies above the central character of order $2$ and $(\mathcal{G}_s)_{\zeta}$ contains the Sylow $2$-subgroup of $\mathcal{G}_s$.
\end{rem}

\subsection{ The structure of $W_{\tilde{\lambda}}$}\label{WtilLam}
\indent

Let $\lambda\in \Irr(L_{\mathcal{I}_{-1},\mathcal{I}})$ and $\wt{\lambda}\in \Irr((\wt{L}_{\mathcal{I}_{-1},\mathcal{I}})_\lambda\mid \lambda)$.
Then $|(W_d^\mathcal{I})_\lambda:(W_d^\mathcal{I})_{\wt{\lambda}}|\leq 2$ and thus there exists a linear character $\nu\in \Irr((W_d^\mathcal{I})_\lambda)$ such that $(W_d^\mathcal{I})_{\wt{\lambda}}={\rm ker}(\nu)$.
The aim of this section is to make use of the groups $(\mathcal{G}_{s})_{\zeta}/(\mathcal{G}_{s})_{\wt{\zeta}}$ for $\zeta\in \mathcal{R}_s$ to construct this character $\nu$.

\begin{notation}\label{nulambda}
Let $\mathcal{O}$ be a $\ol{w}$-orbit in $\mathcal{I}$ and $\zeta\in \mathcal{R}_{|\mathcal{J}_{\mathcal{O}}|}$.
Set  ${\nu}_{\mathcal{O},\zeta}$ to be the lift of the generator for $\Irr\left( \mathcal{G}_{\mathcal{O},\zeta}/\mathcal{G}_{\mathcal{O},\wt{\zeta}}\right)$ to $\Irr( \mathcal{G}_{\mathcal{O},\zeta})$, for $\mathcal{G}_{\mathcal{O},\zeta}$ and $\mathcal{G}_{\mathcal{O},\wt{\zeta}}$ as defined in Notation~\ref{Jlams}.

Assume $\lambda\in \Irr(L_{\mathcal{I}_{-1},\mathcal{I}})$ such that $\lambda_\mathcal{O}:={\rm Res}_{L_\mathcal{O}}^{L_{\mathcal{I}_{-1},\mathcal{I}}}(\lambda)\in \mathcal{R}_\mathcal{O}$ for each $\ol{w}$-orbit $\mathcal{O}\subseteq\mathcal{I}$. 
According to Lemma~\ref{IntWonL} a character $\nu'_\lambda\in \Irr((W_d^\mathcal{I})_\lambda)$ can be defined such that for each $1\leq s\leq n$ and $\zeta\in \mathcal{R}_s$, 
$$
\begin{array}{rccc}
\nu_{\lambda,s,\zeta}:={\rm Res}_{(W_d^{\mathcal{I}_s})_{\lambda_s,\zeta}}^{(W_d^{\mathcal{I}})_{\lambda}}(\nu'_\lambda):&  (W_d^{\mathcal{I}_s})_{\lambda_s,\zeta} & \rightarrow &  \mathbb{C}^\times \\
& \left(\prod_{i\in \mathcal{I}_{\lambda_s}(\zeta)}h_i\right)\sigma & \mapsto & \prod_{i\in \mathcal{I}_{\lambda_s}(\zeta) }\nu_{\mathcal{O}_i^s,\zeta}(h_i)\\
\end{array}.
$$
In particular, $\nu_{\lambda,s,\zeta}$ is non-trivial only for those $\zeta\in\mathcal{R}_{s}^{1}$ with $|\mathcal{I}_{\lambda,s}(\zeta)|\ne 0$ and $\nu_\lambda'$ is a character of order at most two.
Finally define 
$$
\nu_\lambda:=
\left\{
\begin{array}{cl}
\nu_\lambda' & \text{ if } (\wt{L}_{\mathcal{I}_{-1},\mathcal{I}})_\lambda= \wt{L}_{\mathcal{I}_{-1},\mathcal{I}};\\
(\nu_\lambda')^2 & \text{ if } (\wt{L}_{\mathcal{I}_{-1},\mathcal{I}})_\lambda\ne \wt{L}_{\mathcal{I}_{-1},\mathcal{I}}.\\
\end{array}
\right.
$$
\end{notation}

\begin{prop}\label{WtilLamKerNu}
Take $\lambda\in \Irr (L_{\mathcal{I}_{-1},\mathcal{I}})$ as in Lemma~\ref{IntWonL}, $\nu_\lambda\in \Irr((W_d^\mathcal{I})_\lambda)$ as in Notation~\ref{nulambda} and $\wt{\lambda}\in \Irr( (\wt{L}_{\mathcal{I}_{-1},\mathcal{I}})_\lambda \mid \lambda)$.
Then $(W_d^\mathcal{I})_{\wt{\lambda}}={\rm ker}(\nu_\lambda)$.
\end{prop}
\begin{proof}
According to Corollary~\ref{IntTilLIndex2} $\wt{L}_{\mathcal{I}_{-1},\mathcal{I}}=\GenGp{ L_{\mathcal{I}_{-1},\mathcal{I}},z^+ x}$ with $(z^+)^2\in {\bf L}_{\mathcal{I}_{-1},\mathcal{I}}$ and $x\in Z(\wt{\bG})$.
Moreover, by Corollary~\ref{IntTilLIndex2}, $(\wt{L}_{\mathcal{I}_{-1},\mathcal{I}})_{\lambda}=\wt{L}_{\mathcal{I}_{-1},\mathcal{I}}$ or $\GenGp{L_{\mathcal{I}_{-1},\mathcal{I}},x^2}$.
In the latter case, as $x^2\in Z(\wt{\bG})$, the description of characters in a central product in Remark~\ref{IrrCentProd} implies $(W_d^\mathcal{I})_{\wt{\lambda}}=(W_d^\mathcal{I})_\lambda={\rm ker}(\nu_\lambda)$. 

Assume that $(\wt{L}_{\mathcal{I}_{-1},\mathcal{I}})_{\lambda}=\wt{L}_{\mathcal{I}_{-1},\mathcal{I}}$ which, by  Lemma~\ref{InTilL}, implies $(\wt{L}_{\mathcal{I}_{-1},\mathcal{I}}^+)_\lambda= \wt{L}_{\mathcal{I}_{-1},\mathcal{I}}^+$.
Furthermore take $\wt{\lambda}^+$ as in Lemma~\ref{IntWonTilL0} and from the proof of  Lemma~\ref{TilL0StabN} it can be assumed that $\wt{\lambda}^+$ restricts to $\wt{\lambda}$.
For $g\in (W_d^\mathcal{I})_{\lambda}$ there is a linear character $\ol{\mu}_g\in\Irr(\wt{L}_{\mathcal{I}_{-1},\mathcal{I}}^+/L_{\mathcal{I}_{-1},\mathcal{I}})$ such that $(\wt{\lambda}^+)^g=\wt{\lambda}^+\mu_g$.
Moreover, as $x\in Z(\wt{\bG})$, it follows that $\mu_g(x)=1$.
Hence $g\in (W_d^\mathcal{I})_{\wt{\lambda}}$ if and only if $\mu_g( z^+)=1$.

Recall that $z^+:=z^+_{-1}\times \prod_{\mathcal{O}}z_\mathcal{O}^+$ with the product running over all $\ol{w}$-orbits in $\mathcal{I}$. 
Set $\lambda_{\mathcal{I}_{-1}}:={\rm Res}_{{L}_{\mathcal{I}_{-1}}}^{{L}_{\mathcal{I}_{-1},\mathcal{I}}}({\lambda})$ and similarly define $\wt{\lambda}_{\mathcal{I}_{-1}}^+$.
As $[V_d^\mathcal{I},\bL_{\mathcal{I}_{-1}}]=1$, it follows that $\wt{\lambda}_{\mathcal{I}_{-1}}=(\wt{\lambda}_{\mathcal{I}_{-1}})^g=\wt{\lambda}_{\mathcal{I}_{-1}}\mu_g$ as an extension of $\lambda_{\mathcal{I}_{-1}}$.
Thus by Clifford theory $\mu_g(z_{-1})=1$ and so $\mu_g(z^+)=\prod_{\mathcal{O}}\mu_g(z_\mathcal{O}^+)$. 

For $1\leq s\leq n$ let $\mathcal{O}_1^s,\dots,\mathcal{O}_{t_s}^s$ denote the $\ol{w}$-orbits in $\mathcal{I}_s$ and $\wt{\lambda}_{\mathcal{O}_i^s}^+:={\rm Res}_{\wt{L}_{\mathcal{O}_i^s}^+}^{\wt{L}_{\mathcal{I}_{-1},\mathcal{I}}^+}(\wt{\lambda}^+)$.
Using the description of $(W_d^\mathcal{I})_\lambda$ from Lemma~\ref{IntWonL} write $g=\prod_{s=1} \left( \prod_{\zeta\in \mathcal{R}_s}\left( \prod_{j\in\mathcal{J}_{\lambda,s}(\zeta)}g_{s,\zeta,j}\right) \sigma_{s,\zeta}\right)$ with $g_{s,\zeta,j}\in (\mathcal{G}_{\mathcal{O}_j^s})_\zeta$ and $\sigma_{s,\zeta}\in\alpha_s(\mathfrak{S}(\mathcal{J}_{\lambda,s}(\zeta)))$.
There are linear characters ${\mu}_{g_{s,\zeta,j}}\in\Irr(\wt{L}_{Q_j^s}^+)$ such that $$\left( \prod_{j\in\mathcal{I}_{\lambda_s}(\zeta)} \wt{\lambda}_{Q_j^s}\right)^{g}=\left( \prod_{j\in\mathcal{I}_{\lambda_s}(\zeta)} (\wt{\lambda}_{Q_j^s})^{g_{s,\zeta,j}}\right)^{\sigma_{s,\zeta}}=\left( \prod_{j\in\mathcal{I}_{\lambda_s}(\zeta)} (\wt{\lambda}_{Q_j^s})\mu_{g_{s,\zeta,j}}\right)^{\sigma_{s,\zeta}}.$$
Therefore, as each $(z_{Q_{\sigma^{-1}(j)}^s}^+)^{\sigma_{s,\zeta}}= z_{Q_{j}^s}^+$, it follows that
$$\prod_{\mathcal{O}}\mu_g(z_{\mathcal{O}}^\circ)=\prod_{s=1}\left( \prod_{\zeta\in\mathcal{R}_s}\left(\prod_{j\in\mathcal{I}_{\lambda_s}(\zeta)} \mu_{g_{s,\zeta,j}} (z_{Q_{j}^s}^+)\right)\right).$$
While from the definition $$\nu_{\lambda}(g)=\prod_{s=1}^n\left( \prod_{\zeta\in\mathcal{R}_s} \left( \prod_{j\in \mathcal{I}_{\lambda_s}(\zeta)}\nu_{s,\zeta}(g_{s,\zeta,j})\right)\right).$$
Finally for each $j\in\mathcal{I}_{\lambda_s}(\zeta)$, by construction it can be seen that $\nu_{s,\zeta}(g_{s,\zeta,j})=\mu_{g_{s,\zeta,j}}(z_{Q_{j}^s}^+)$ as
$$
\begin{array}{ccl}
\nu_{s,\zeta}(g_{s,\zeta,j})=1 & \Leftrightarrow & g_{s,\zeta,j}\in (\mathcal{G}_{Q_j^s})_{\wt\zeta} \\
 & \Leftrightarrow & \mu_{g_{s,\zeta,j}}=1_{\wt{L}_{Q_j^s}}\\
  & \Leftrightarrow & \mu_{g_{s,\zeta,j}}(z_{Q_{j}^s}^+)=1
\end{array}
$$
Hence $\nu_\lambda(g)=\mu_g(z^+ x)$ and so $(W_d^\mathcal{I})_{\wt\lambda}={\rm ker}(\nu_{\lambda})$.
\end{proof}

\subsection{A criterion to verify condition~\ref{thm41iii}}
\indent

\begin{rem}\label{NormPreAlpha}
Assume $\lambda\in \Irr(L_{\mathcal{I}_{-1},\mathcal{I}})$ such that for each $\mathcal{O}$ a $\ol{w}$-orbit in $\mathcal{I}$ the character $\lambda_{\mathcal{O}}:={\rm Res}^{L_{\mathcal{I}_{-1},\mathcal{I}}}_{L_{\mathcal{O}}}(\lambda)\in \mathcal{R}_{\mathcal{O}}$.
In other words, for $s=|\mathcal{J}_\mathcal{O}|$ and $\alpha_s$ from Notation~\ref{IsomWs}, $\lambda_\mathcal{O} \circ \alpha_s\in \mathcal{R}_s$.
According to  Corollary~\ref{NormInWrOfWr},
$$
\NNN_{\mathcal{G}_s\wr \Symm_{t_s}}\left(\prod_{\zeta\in \mathcal{R}_s} (\mathcal{G}_{s})_\zeta \wr \Symm (\mathcal{J}_{\lambda,s}(\zeta))\right)=\left( Z_{\lambda_s\circ \alpha_s}\left(\prod_{\zeta\in \mathcal{R}_s} (\mathcal{G}_{s})_\zeta \wr \Symm (\mathcal{J}_{\lambda,s}(\zeta))\right)\right)\rtimes S_{\lambda_s\circ\alpha_s}
$$
where $Z_{\lambda_s\circ\alpha_s}:=\prod_{\zeta\in\mathcal{R}_s} \Delta_{\mathcal{J}_{\lambda,s}(\zeta)}\mathcal{G}_s$ and $$S_{\lambda\circ\alpha}:=\langle \tau_{\mathcal{J}_{\lambda,s}(\zeta), \mathcal{J}_{\lambda,s}(\zeta')}\mid |\mathcal{J}_{\lambda,s}(\zeta)|=|\mathcal{J}_{\lambda,s}(\zeta')| \text{ and } (\mathcal{G}_s)_{\zeta}= (\mathcal{G}_s)_{\zeta'} \rangle.$$
Set $Z_{\lambda\circ\alpha}:=\prod_{s=1}^n (Z_{\lambda_s\circ\alpha_s})$ and $S_{\lambda\circ\alpha}:=\prod_{s=1}^n (S_{\lambda_s\circ\alpha_s})$
\end{rem}

\begin{lm}\label{NWdWLam}\label{StructK}
Let $\lambda\in \Irr(L_{\mathcal{I}_{-1},\mathcal{I}})$ and assume that for each $\ol{w}$-orbit $\mathcal{O}\subseteq\mathcal{I}$ the character $\lambda_{\mathcal{O}}:={\rm Res}^{L_{\mathcal{I}_{-1},\mathcal{I}}}_{L_{\mathcal{O}}}(\lambda_s)\in \mathcal{R}_{\mathcal{O}}$.
Using the setup from Notation~\ref{NormPreAlpha}, set $Z_\lambda:=\alpha(Z_{\lambda\circ\alpha})$ and $S_\lambda:=\alpha(S_{\lambda\circ\alpha})$.
Then
$$\NNN_{W_d^{\mathcal{I}}}((W_d^\mathcal{I})_\lambda)=Z_\lambda\cdot (W_d^{\mathcal{I}})_{\lambda}\rtimes S_{\lambda}.$$
Additionally let $\wt{\lambda}\in \Irr( (\wt{L}_{\mathcal{I}_{-1},\mathcal{I}})_\lambda \mid \lambda)$ and take $\nu_\lambda\in \Irr((W_d^\mathcal{I})_\lambda)$ as defined in Notation~\ref{nulambda}.
Then
$$\NNN_{W_d^\mathcal{I}}((W_d^\mathcal{I})_\lambda,(W_d^\mathcal{I})_{\wt{\lambda}})=Z_\lambda\cdot (W_d^\mathcal{I})_\lambda\rtimes (S_{\lambda})_{\nu_{\lambda}}.$$
\end{lm}
\begin{proof}
The structure of $\NNN_{W_d^{\mathcal{I}}}((W_d^\mathcal{I})_\lambda)$ follows from Remark~\ref{NormPreAlpha} as 
$$\alpha^{-1}\left( \NNN_{W_d^{\mathcal{I}}}((W_d^\mathcal{I})_\lambda)\right)=\prod_{s=1}^n\alpha_s^{-1}\left( \NNN_{W_d^{\mathcal{I}_s}}((W_d^{\mathcal{I}_s})_{\lambda_s})\right)= \prod_{s=1}^n\NNN_{\mathcal{G}_s\wr \mathfrak{S}_{t_s}}\left(\prod_{\zeta\in\mathcal{R}_s}  (\mathcal{G}_s)_\zeta\wr \mathfrak{S}(\mathcal{I}_{\lambda_s}(\zeta))\right).$$
As $Z_{\lambda}$ acts trivially on $(W_d^{\mathcal{I}})_{\lambda}$ and $(W_d^\mathcal{I})_{\wt\lambda}\lhd (W_d^\mathcal{I})_\lambda$, it follows that 
$$
\NNN_{W_d^{\mathcal{I}}}((W_d^\mathcal{I})_\lambda,(W_d^{\mathcal{I}})_{\wt{\lambda}})=Z_\lambda\cdot  (W_d^{\mathcal{I}})_{\lambda}\rtimes \NNN_{S_\lambda}((W_d^{\mathcal{I}})_{\wt{\lambda}}).
$$
The result now follows from Proposition~\ref{WtilLamKerNu} as $(W_d^{\mathcal{I}})_{\wt{\lambda}}={\rm ker}(\nu_\lambda)$ for $\nu_\lambda\in{\rm Lin}((W_d^{\mathcal{I}})_{\lambda})$ of order at most two and so $\NNN_{S_\lambda}({\rm ker}(\nu_\lambda))=(S_\lambda)_{\nu_\lambda}$.
\end{proof}

\begin{prop}\label{KInva}
Let $\lambda\in \Irr(L_{\mathcal{I}_{-1},\mathcal{I}})$ and $\xi_0\in\Irr((W_d^\mathcal{I})_{\wt\lambda})$.
Assume for each $1\leq s\leq n$ and any pair of distinct characters $\zeta_1\ne \zeta_2\in \mathcal{R}_{s}^1$ whenever $|\mathcal{I}_{\lambda,s}(\zeta_1)|=|\mathcal{I}_{\lambda,s}(\zeta_2)|\ne 0$ then $(\mathcal{G}_s)_{\zeta_1}\ne (\mathcal{G}_s)_{\zeta_2}$.
Then every $\xi\in \Irr((W_d^\mathcal{I})_\lambda\mid \xi_0)$ is $K(\lambda)_{\xi_0}$-invariant where $K(\lambda)=\NNN_{W_d^\mathcal{I}}((W_d^\mathcal{I})_{\wt\lambda},(W_d^\mathcal{I})_{\lambda})$.
\end{prop}
\begin{proof}
After suitable $W_d^\mathcal{I}$-conjugation, it can be assumed that for each $\ol{w}$-orbit $\mathcal{O}\subset\mathcal{I}$ the character $\lambda_\mathcal{O}:={\rm Res}_{L_{\mathcal{O}}}^{L_{\mathcal{I}_{-1},\mathcal{I}}}(\lambda)\in \mathcal{R}_\mathcal{O}$.
Moreover it can also be assumed that $\nu_\lambda=\nu_\lambda'$,  for $\nu_\lambda\in \Irr((W_d^\mathcal{I})_\lambda)$ from Notation~\ref{nulambda}, otherwise $\nu_\lambda$ is trivial and so $(W_d^\mathcal{I})_{\wt\lambda}=(W_d^\mathcal{I})_{\lambda}$ by Proposition~\ref{WtilLamKerNu}.
Thus, from Corollary~\ref{IntTilLIndex2}, it follows that $(W_d^\mathcal{I})_\lambda/(W_d^\mathcal{I})_{\wt{\lambda}}$ has order  two. 
Hence ${\rm Res}_{(W_d^\mathcal{I})_{\wt{\lambda}}}^{(W_d^\mathcal{I})_\lambda}(\xi)=\xi_0$ or $\xi={\rm Ind}_{(W_d^\mathcal{I})_{\wt{\lambda}}}^{(W_d^\mathcal{I})_\lambda}(\xi_0)$.
In the latter case it follows that $\xi $ is $K(\lambda)_{\xi_0}$-stable.
Therefore it can additionally be assumed that  ${\rm Res}_{(W_d^\mathcal{I})_{\wt{\lambda}}}^{(W_d^\mathcal{I})_\lambda}(\xi)=\xi_0$, so that by Lemma~\ref{StructK}, $K(\lambda)_{\xi_0}=(Z_\lambda (\cdot W_d^\mathcal{I})_\lambda)\rtimes ((S_{\lambda})_{\nu_\lambda'})_{\xi_0}$.
It suffices to prove in this case that $((S_{\lambda})_{\nu_\lambda'})_{\xi_0}=((S_{\lambda})_{\nu_\lambda'})_{\xi}$.

Set $\nu_{\lambda\circ\alpha}'=\nu_\lambda'\circ \alpha\in \Irr(\alpha^{-1}((W_d^\mathcal{I})_\lambda))$.
Then the isomorphism $\alpha$ implies the statement $((S_{\lambda})_{\nu_\lambda'})_{\xi_0}=((S_{\lambda})_{\nu_\lambda'})_{\xi}$ is equivalent to $((S_{\lambda\circ \alpha})_{\nu_{\lambda\circ \alpha}'})_{\xi_0\circ \alpha}=((S_{\lambda\circ\alpha})_{\nu_{\lambda\circ\alpha}'})_{\xi\circ\alpha}$.
Note that $\alpha^{-1}((W_d^\mathcal{I})_\lambda)=H_1\times H_2$, where $H_i:=\prod_{s=1}^n\left( \prod_{\zeta\in \mathcal{R}_s^i}(\mathcal{G}_s)_\zeta\wr \Symm(\mathcal{J}_{\lambda,s}(\zeta))\right)$.
For the character $\nu_{\lambda\circ\alpha,1}':={\rm Res}_{H_1}^{\alpha^{-1}((W_d^\mathcal{I})_\lambda)}(\nu_{\lambda\circ\alpha}')$ it follows that $\nu_{\lambda\circ\alpha}'=\nu_{\lambda\circ\alpha,1}'\times 1_{H_2}$  as the factor $\nu_{\lambda,s,\zeta}\circ\alpha_s$ of $\nu_{\lambda\circ\alpha}'$ is non-trivial if and only if $\zeta\in\mathcal{R}_{s}^{1}$ with $|\mathcal{J}_{\lambda,s}(\zeta)|\ne 0$.
Hence ${\rm ker}(\nu_{\lambda\circ\alpha}')={\rm ker}(\nu_{\lambda\circ\alpha,1}')\times H_2$ and thus there are characters $\xi_{0,1}\in \Irr({\rm ker}(\nu_{\lambda\circ\alpha,1}))$, $\xi_{1}\in \Irr(H_1)$ and $\xi_2\in \Irr(H_2)$ such that $\xi_0\circ \alpha=\xi_{0,1}\times \xi_2$ and $\xi\circ\alpha=\xi_1\times \xi_2$.

By Remark~\ref{NWdWLam}, $ S_{\lambda_s\circ\alpha_s}=\langle \tau_{\mathcal{J}_{\lambda,s}(\zeta), \mathcal{J}_{\lambda,s}(\zeta')}\mid |\mathcal{J}_{\lambda,s}(\zeta)|=|\mathcal{J}_{\lambda,s}(\zeta')| \text{ and } (\mathcal{G}_s)_{\zeta}= (\mathcal{G}_s)_{\zeta'} \rangle$.
Hence for $\mathcal{J}_{\lambda,s}(\mathcal{R}_{s}^{1}):=\{\mathcal{J}_{\lambda,s}(\zeta)\mid \zeta\in\mathcal{R}_{s}^{1}\}$, it follows that 
$$(S_{\lambda\circ\alpha})_{\nu_{\lambda\circ\alpha}'}=\prod_{s=1}\{\sigma_s\in S_{\lambda_s\circ\alpha_s}\mid \sigma_s\left( \mathcal{J}_{\lambda,s}(\mathcal{R}_{s}^{1})\right)=\mathcal{J}_{\lambda,s}(\mathcal{R}_{s}^{1})\}.$$
The assumption that for any pair of distinct characters $\zeta_1\ne \zeta_2\in \mathcal{R}_{s}^1$ either $|\mathcal{J}_{\lambda,s}(\zeta_1)|\ne |\mathcal{J}_{\lambda,s}(\zeta_2)|$ or $(\mathcal{G}_s)_{\zeta_1}\ne (\mathcal{G}_s)_{\zeta_2}$ implies that the group $(S_{\lambda\circ\alpha})_{\nu_{\lambda\circ\alpha}'}$ acts trivially on $H_1$ and thus also on ${\rm ker}(\nu_{\lambda\circ\alpha,1}')$.
Therefore $((S_{\lambda\circ \alpha})_{\nu_{\lambda\circ \alpha}'})_{\xi_0\circ \alpha}=((S_{\lambda\circ \alpha})_{\nu_{\lambda\circ \alpha}'})_{\xi_{0,2}}=((S_{\lambda\circ\alpha})_{\nu_{\lambda\circ\alpha}'})_{\xi\circ\alpha}((S_{\lambda\circ\alpha})_{\nu_{\lambda\circ\alpha}'})_{\xi\circ\alpha}$ completing the proof.
\end{proof}

\subsection{\bf The proof of Theorem~\ref{thmA}}
\indent

\begin{proof}[Proof of Theorem~\ref{thmA}]
It suffices to prove the conditions of Theorem~\ref{MainThmReq} with respect to the Sylow twist $vF$ from Section~\ref{vwSylTwist}.
As explained in Section~\ref{RootSysdLevi}, up to conjugation, it can be assumed that $\bT\leq \bL$ and the root system $\Phi_{\bL}$ coincides with $\Phi_{\mathcal{I}_{-1},\mathcal{I}}$ where $(\mathcal{I}_{-1},\mathcal{I})$ satisfies Remark~\ref{PartLabLevi} with respect to $w=\rho(v)\in\sgnSymm_n$. 
Furthermore by Lemma~\ref{TwistTwist} it suffices to verify the conditions for $\bL_{\mathcal{I}_{-1},\mathcal{I}}$ with fixed points taken with respect to $v_\mathcal{I}F$ defined in Section~\ref{vI}.
Corollary~\ref{IntTilLNhat} and Proposition~\ref{NhatExtMap} prove that conditions~\ref{thm41i} and~\ref{thm41ii} holds for $\bL_{\mathcal{I}_{-1},\mathcal{I}}^{v_\mathcal{I}F}$.
Thus to prove Theorem~\ref{thmA} it remains to prove that condition~\ref{thm41iii} holds for any $d$-cuspidal character $\lambda\in \Irr (L_{\mathcal{I}_{-1},\mathcal{I}})$.

Let $\mathcal{O}$ be a $\ol{w}$-orbit in $\mathcal{I}$ and set $\lambda_\mathcal{O}:={\rm Res}_{L_\mathcal{O}}^{L_{\mathcal{I}_{-1},\mathcal{I}}}(\lambda)$.
As in \cite[Ex. 3.5.29]{GeckMal}, the character $\lambda_\mathcal{O}$ is a $d$-cuspidal character of $L_\mathcal{O}\cong {\rm GL}_s^\epsilon(q^{d_0})$ which when considered as a group over $\mathbb{F}_q^{d_0}$ becomes a $d/d_0$-cuspidal character.
Assume $s\geq 2$ and that $(\mathcal{G}_\mathcal{O})_{\lambda_\mathcal{O}}$ contains the unique involution of the cyclic group $\mathcal{G}_\mathcal{O}$.
Then Proposition~\ref{CuspTypeA} implies $Z(L_\mathcal{O})\leq {\rm ker}(\lambda_\mathcal{O})$ and hence $\lambda_\mathcal{O}$ lies above the trivial character of $Z(L_\mathcal{O})$. 
Thus for each $\zeta\in \mathcal{R}_s^1$ the set $\mathcal{J}_{\lambda,s}(\zeta)$ is empty.
Additionally, by Lemma~\ref{GStabCharZ}, $\mathcal{R}_1^1$ contains a unique character, the one of order $2$.
Thus $\lambda$ satisfies the requirements of Proposition~\ref{KInva}, which proves that condition~\ref{thm41iii} holds with respect to $\lambda$.
\end{proof}

\begin{rem}
Using the proof above, the condition of Theorem~\ref{thmA} holds if $d$-cuspidal character is replaced by any character satisfying the assumptions of Proposition~\ref{KInva}. 
Additionally, as in the proof of Proposition~\ref{KInva}, it will also hold whenever $(\wt{L}_{\mathcal{I}_{-1},\mathcal{I}})_\lambda\ne   \wt{L}_{\mathcal{I}_{-1},\mathcal{I}}$.
\end{rem}

\bibliographystyle{plain}
\bibliography{bibfile}

\end{document}

%% file: shortcuts.tex
\newcommand{\bL}{{{\mathbf L}}}

\newcommand{\FF}{\ensuremath{\mathbb{F}}}

\newcommand{\bT}{{\mathbf T}}
\newcommand{\NNN}{\operatorname{N}}

\newcommand{\GF}{ {\bG^F}}

\newcommand{\Cent}{\operatorname{C}}

\newcommand{\wt}{\widetilde}
\newcommand{\wbG}{{\widetilde \bG}}

\usepackage{cleveref,enumitem}
\newlist{asslist}{enumerate}{1} 
\setlist[asslist]{label=(\roman*), ref=\thethm(\roman*)}

\crefalias{asslistenumi}{Assumption} 

\newlist{thmlist}{enumerate}{1} 
\setlist[thmlist]{label=(\alph*), ref=\thethm(\alph*)}